\documentclass[12pt,leqno]{amsart}
\usepackage{a4wide}
\usepackage{amssymb}
\usepackage{amsmath}
\usepackage{amsthm}
\usepackage{verbatim}
\usepackage{txfonts}
\usepackage{fullpage}
\usepackage[all,cmtip]{xy} 
\usepackage{yfonts}
\usepackage[final]{changes}

\usepackage[T1]{fontenc}
\usepackage[utf8]{inputenc}
\usepackage{mathtools}   
\usepackage{amsfonts}
\usepackage{mathrsfs, todonotes}
\usepackage[mathscr]{euscript}

\usepackage{hyperref}
\numberwithin{equation}{section}

\theoremstyle{plain}

\newcommand{\yh}[1]{}
\newcommand{\yhn}[1]{}
\renewcommand{\todo}[1]{}
\newcommand{\pn}[1]{}

\newcommand{\A}{\ensuremath{{\mathbb{A}}}}
\newcommand{\C}{\ensuremath{{\mathbb{C}}}}
\newcommand{\Z}{\ensuremath{{\mathbb{Z}}}}

\newcommand{\Q}{\ensuremath{{\mathbb{Q}}}}
\newcommand{\R}{\ensuremath{{\mathbb{R}}}}
\newcommand{\F}{\ensuremath{{\mathbb{F}}}}
\newcommand{\G}{\ensuremath{{\mathbb{G}}}}
\newcommand{\I}{\ensuremath{{\mathbb{I}}}}
\newcommand{\D}{\ensuremath{{\mathbb{D}}}}
\newcommand{\B}{\ensuremath{{\mathbb{B}}}}\newcommand{\PB}{\ensuremath{{\text{P}\mathbb{B}}}}
       
\newcommand{\g}{\ensuremath{{\mathfrak{g}}}}
\newcommand{\h}{\ensuremath{{\mathfrak{h}}}}
\newcommand{\vv}{\ensuremath{v_{\B,\L}}}
\newcommand{\ddager}{\ensuremath{\star}}
\newcommand{\pr}{\ensuremath{\text{Pr}}}
\newcommand{\E}{\ensuremath{{\mathbb{E}}}}
\renewcommand{\L}{\ensuremath{{\mathbb{L}}}}
\renewcommand{\l}{\ensuremath{d}}

\newcommand{\Vol}{\text{Vol}}
\newcommand{\GL}{\ensuremath{{\text{GL}}}}
\newcommand{\PGL}{\ensuremath{{\text{PGL}}}}
\newcommand{\Char}{\ensuremath{{\text{char}}}}

\newtheorem{theo}{Theorem}[section]
\newtheorem{lem}[theo]{Lemma}
\newtheorem{prop}[theo]{Proposition}
\newtheorem{cor}[theo]{Corollary}

\theoremstyle{remark}
\newtheorem{rem}[theo]{Remark}
\newtheorem{example}[theo]{Example}
\theoremstyle{definition}
\newtheorem{defn}[theo]{Definition}

\newtheorem*{cor*}{Corollary}

\newtheorem{asmp}[theo]{Setting}

\newcommand{\zxz}[4]{\begin{pmatrix} #1 & #2 \\ #3 & #4 \end{pmatrix}}

\newcommand{\Hom}{\operatorname{Hom}}
\newcommand{\Ind}{\operatorname{Ind}}

\newcommand{\Tr}{\text{Tr}}
\newcommand{\Nm}{\text{Nm}}
\newcommand{\Supp}{\text{Supp }}

\newcommand{\z}{\mathfrak{z}}
\newcommand{\epsiBL}{\ensuremath{{\epsilon_{\B,\L}}}}

\newcommand{\imaginary}{\text{minimal}} 
\newcommand{\epsi}{j}

\begin{document}
\bibliographystyle{plain}
\title{New test vector for Waldspurger's period integral, relative trace formula, and hybrid subconvexity bounds}
\author{Yueke Hu}
\address{Yau Mathematical Sciences Center\\ Tsinghua University\\
	 Beijing 100084\\
	  China}

\author{Paul D. Nelson}
\address{Department of Mathematics\\
  ETH Zurich\\
  Zurich\\
  Switzerland}

\begin{abstract}
  In this paper we give quantitative local test vectors for Waldspurger's period integral (i.e., a toric period on $\GL_2$) in new cases with joint ramifications.  The construction involves minimal vectors, rather than newforms and their variants.   
  This paper gives a uniform treatment for the matrix algebra and division algebra cases under mild assumptions, and establishes an explicit relation between the size of the local integral and the finite conductor $C(\pi\times\pi_{\chi^{-1}})$. As an application, we combine the test vector results with the relative trace formula, and prove a hybrid type subconvexity bound which can be as strong as the Weyl bound in proper range.
\end{abstract}

\maketitle
\tableofcontents

\section{Introduction}
\yhn{Notation changes: cyclic subgroups $\mu_\E$ lifting the nonzero elements of residue field is changed to $u_\E$. Characters $\nu$ which is confused with valuation are changed to $\mu$.}
{This paper provides quantitative test vectors for the following
  local integral of Waldspurger's period integral for the cases
  untouched in the previous literatures: \pn{maybe $\varphi$
    rather than
    $\varphi^{\mathbb{B}}$?
    looks a bit noisy.  similarly for $\pi$}\yhn{I think this is a minor issue and would leave it be for the moment}
  \begin{equation}
    I(\varphi^\B,\chi)=\int\limits_{\F^\times \backslash\E^\times }\Phi_{\varphi^\B}(t)\chi^{-1}(t)dt.
  \end{equation}
  Here $\F$ is a 
  $p-$adic field with odd residue field characteristic and $\E=\F(\sqrt{D})$ is an \'{e}tale quadratic algebra over $\F$, embedded into a quaternion algebra $\B$. $\Phi_{\varphi^\B}$ is the matrix coefficient associated to $\varphi^\B\in \pi^\B$, where $\pi^\B$ is the image of a representation $\pi$ of $\GL_2$ under Jacquet-Langlands correspondence. 
 We shall always enforce the normalization  $\Phi_{\varphi^\B}(1)=1$.
  $\chi$ is a character over $\E$ such that $w_\pi=\chi|_{\F^\times }$.
  By a proper twist on both $\pi$ and $\chi$,  we can assume
 that $\pi$ is minimal, i.e., its level satisfies $c(\pi)\leq c(\pi\otimes\mu)$ for any character $\mu$ of $\F^\times$. 

  Explicit knowledge of the test vector and the resulting size of the local integral is very useful for applications. In some recent development for example, it has been used for the study of mass equidistribution  in \cite{Hu:a},  moments and subconvexity bound of L-functions in \cite{FileMartin:17a} and \cite{Wu:16a}, as well as congruent number and cube sum problems in \cite{TianYuanZhang:14a} and \cite{CaiShuTian:14a}. The basic goal of this paper is to provide a more complete study of test vectors and local integrals compared to the previously known cases, which the future applications can directly make use of. 

  The new input for this problem is the compact induction theory
  for supercuspidal representations, and we choose so-called
  minimal vectors as test vectors, which naturally arise from
  this theory and can be identified as eigenvectors for certain characters on large compact open subgroups.
  Such test vectors, and especially their matrix
  coefficients, have nice properties which have simple
  interpretations in terms of Lie algebras, whose applications to
  period integrals and analytic number theory problems 
  have not been fully exploited. When the nonzero contribution to
  the local integral is from the Lie algebra range,
 we give a
  uniform construction of test vectors
  and compute the size of the local integral.
  The proof does not directly
  rely on Tunnell-Saito's $\epsilon$-value test,
  and provides alternative perspective or reflection for that test. 
  When the nonzero contribution comes from the whole torus, then
  $\pi$ and $\pi_{\chi^{-1}}$ must be completely related, 
  and the choice of minimal vector is more direct. 

  The method  is a natural $p$-adic analogue and extension of some of the
  local calculations
  at the real place in \cite{NelsonVenkatesh:}. In that paper, the Kirillov trace formula is the main tool and one has to stay away from the conductor dropping range. In this paper we make direct use of minimal vectors, which enables us to  deal with the whole conductor dropping range. 
 Hopefully the method applies also to a broader class of groups.

  As a byproduct or direct application, one can conveniently work out the relation between the standard newforms and minimal vectors, and use representation-theoretical approach to evaluate the local integral also for newforms. This strategy is carried out in \cite{HuYinShu18} and applied to study the 3-part full BSD conjecture in \cite{HuYinShu180}.
  
 We also present another application in this paper to an interesting hybrid subconvexity bound for the twisted base change L-function, which can be as strong as Weyl type bound in proper range. The strategy makes use of the relative trace formula developed in \cite{jacquet_sur_1987}, and is close to that of \cite{feigon_averages_2009} while focusing on joint ramifications. The main new  ingredient is to choose the test function to be essentially the matrix coefficient of a minimal vector. On the spectral side, this choice of test function allows us to not only make use of the local Waldspurger's period integral developed in this paper, but also shorten the length of the spectral sum when compared to, for example, choosing as test function the characteristic function of the congruence subgroup for a newform. Meanwhile, the geometric side is still simple enough to get a good upper bound. We shall use consistent languages for controlling the geometric side as those for the local Waldspurger's period integrals. 


\subsection{A brief history of test vectors for Waldspurger's period integral}
The study of test vectors when there are ramifications was initiated in \cite{GrossPrasad:91a}. It assumes disjoint ramifications, and describes a test vector in terms of invariance by proper compact open subgroup. For example when the level of the representation is $c(\pi)=4n$, $\E$ is inert and $\chi$ is unramified, the Gross-Prasad test vector is defined to be the unique vector invariant by 
\begin{equation}\label{eq:GrossPrasadTest}
 K_\E(2n)=\E^\times (1+\varpi^{2n}M_2(O_\F)).
\end{equation}


The work in \cite{FileMartin:17a}  gives test vectors in more general situations on $\GL_2$ side. In particular it solves the case when $\E$ is split. When $\E$ is a quadratic field extension, it gives test vector for the range $c(\chi)\geq c(\pi)$. The test vectors used are essentially newforms (more precisely, the diagonal translates of newforms), which can be explicitly given in the Kirillov model and also described by invariance under proper compact subgroups (precisely, conjugates of standard congruence subgroups by a diagonal matrix).  Their method should be directly applicable to a larger range when $c(\pi_{\chi^{-1}})>c(\pi)$, where $\pi_{\chi^{-1}}$ is the representation of $\GL_2$ associated to $\chi^{-1}$ via the theta correspondence or Langlands correspondence.

In the arXiv version \cite{Hu:a} the first author gives a
partial study of test vectors when $c(\pi_{\chi^{-1}})\leq
c(\pi)$.  When $\pi$ is a supercuspidal representation, it was
found that newforms will fail in certain situations, and one has
to make use of twisted newforms (i.e., vectors associated to
newforms of twisted representations). Explicit evaluation for
the local integral is complicated and was only done in
easier situation when
$c(\pi)$ is much larger than $c(\pi_{\chi^{-1}})$.
\cite{Hu:a} also solves the case when $\pi$ is a principal series using twisted newforms.

Another issue is that \cite{FileMartin:17a} and \cite{Hu:a} heavily rely on newforms, which is not available on the division algebra side. The only case previously known on the division algebra side is from \cite{GrossPrasad:91a}, with disjoint ramifications.

Of course one can assume that Tunnell-Saito's $\epsilon-$value test passes and take the test vector to be the eigenvector for $\chi$ on $\E^\times$, then the local integral is simply the volume of the torus. 
Such test vectors however have not be given explicitly in the Kirillov model, provide no additional information on the local period integral for newforms, and are unwieldy for the relative trace formulae. Our goal is to construct test vectors which overcome these disadvantages and are independent of Tunnell-Saito's $\epsilon-$value test.

\subsection{Alternative choice of test vector}\label{subsec1-3:newtestvec}
In \cite{HuNelsonSaha:17a} a different type of test vectors from supercuspidal representations is used to study the sup-norm problem and the equivalence between QUE and subconvexity. Such test vectors exhibit particularly nice properties for their Whittaker functions and matrix coefficients. The case $c(\pi)=4n$ is considered in \cite{HuNelsonSaha:17a}, where $\pi$ is constructed from a character $\theta$ over another inert quadratic field extension $\L$. More explicitly $\pi$ is compactly induced from a character  $\tilde{\theta}$ of a neighbourhood $J$ of $\L^\times $ ($J=\L^\times (1+\varpi^{n}M_2(O_\F))$ for example for suitable embedding of $\L$), where  $\tilde{\theta}$  is an extension of $\theta $ to $J$. 
Naturally there exists an element $\varphi_\theta$ such that $J$ acts on it by the character  $\tilde{\theta}$, and it can be identified intrinsically in this way. 

It is the analogue of lowest weight vector from discrete series representation, on which the compact subgroup $\text{SO}(2,\R)$ acts by a character.
It is referred to as \textit{minimal vector} in \cite{HuNelsonSaha:17a}.
The explicit description of its Whittaker function is the key ingredient used in \cite{HuNelsonSaha:17a} to get strong upper and lower bounds of the sup norm of the associated global modular form. Another interesting and important feature is that its associated matrix coefficient is multiplicative on the support. Further more if we consider a single translate of this element $\pi(k)\varphi_\theta$, the new matrix coefficient is a conjugate of the old one and is still multiplicative on support. 

In this paper we work with similar test vectors for the remaining supercuspidal representations for $\GL_2$ or $\D^\times $. 
Using the language of \cite{BushnellHenniart:06a}, a cuspidal type is a triple $(\mathfrak{A}, J, \Lambda)$ where $\mathfrak{A}$ is a chain order, $J$ is a compact (always understood as mod center) open subgroup, and $\Lambda$ is an irreducible representation of $J$ of dimension $1$, $q-1$ or $q$. A cuspidal type always contains a simple character $\tilde{\theta}$ of a subgroup $H^1$, with the property that $g\in G$ intertwines $\tilde{\theta}$ if and only if $g\in J$. In particular $\Lambda|_{H^1}$ is a multiple of $\tilde{\theta}$. The cuspidal type $(\mathfrak{A}, J, \Lambda)$ is associated to $\pi$ if $\pi\simeq c-\Ind_J^G\Lambda$. 
Then an element $\varphi\in\pi$ is called a minimal vector if it is an eigenvector for a simple character $(H^1,\tilde{\theta})$ contained in a cuspidal type $(\mathfrak{A}, J, \Lambda)$ associated to $\pi$.
See Section \ref{Sec:CompactinductionSummary+Kirillov} for exact
definitions and more details. The minimal vectors are also given in the Kirillov model in Appendix \ref{Section:Kirillov} using an explicit intertwining operator.

\begin{rem}\label{Rem1-2:Minimalvec}
\begin{enumerate}	
\item Note that the dimension of minimal vectors for a given cuspidal type is the same as $\dim\Lambda$. One can however require stronger equivariance property to uniquely identify a particular basis for the minimal vectors in the case $\dim \Lambda>1$. See Definition \ref{Defn:generalMinimalVec} for two ways to do this.
\item Any single translate $\pi(g)\varphi$ for a minimal vector $\varphi$ is a minimal vector for a conjugated cuspidal type associated to $\pi$, and all minimal vectors arise in this way since by  \cite{BushnellHenniart:06a} all  cuspidal types associated to $\pi$ are conjugate to each other. 
\item The matrix coefficient of a minimal vector, as given in Lemma \ref{Cor:MCofGeneralMinimalVec}, is supported only on $J$ and almost multiplicative on the support. In particular with a basis identified by Definition \ref{Defn:generalMinimalVec}, minimal vectors for all cuspidal types associated to $\pi$ form an orthonormal basis for $\pi$.
\item  For comparison, let $\varphi_{new}$ be the standard newform for $\pi$.
On one hand $\varphi_{new}$ is invariant by $K_0(p^{c(\pi)})$ with $\Vol(K_0(p^{c(\pi)}))\asymp \frac{1}{q^{c(\pi)}}$. The compact open subgroup defining the Gross-Prasad test vector in \eqref{eq:GrossPrasadTest} has similar volume. 
On the other hand for minimal vectors from the same representation, we have $\Vol(H^1)\asymp_q \frac{1}{q^{c(\pi)/2}}$.
From this perspective, we lose the invariance property by allowing a character, and what we gain is that a much larger compact open subgroup behaves well on the test vector.
\end{enumerate}
\end{rem}

Note that minimal vectors with similar properties exist in greater generality according to \cite{BushnellKutzko:a} \cite{stevens_supercuspidal_2008}. In a recent work \cite{hu_sup_2018} by the first author, a sub-local sup norm bound was obtained for minimal vectors on $\PGL_n$, which is a direct generalisation of \cite{HuNelsonSaha:17a}.


\subsection{Main results}
The first purpose of this paper is to provide test vectors for Waldspurger's period integral when $\pi$ is supercuspidal, in a relatively uniform and simple fashion, covering almost all possible situation (especially when $c(\pi)\geq c(\pi_{\chi^{-1}})$).

Unlike the cases $c(\pi)< c(\pi_{\chi^{-1}})$ considered before, in general the local integral may be nonvanishing on the matrix algebra side or  the division algebra side according to Tunnell-Saito's $\epsilon-$value test \cite{Tunnell:83a} \cite{Saito:93a}, which can be tricky to compute.
A particular interesting challenge solved in this paper  is to give a uniform way to find test vectors whose local integral can automatically reflect the $\epsilon-$value test.

Instead of newforms, the nice properties of the minimal vectors discussed in  Remark \ref{Rem1-2:Minimalvec} 
motivate us to use them as test vectors. 
The minimal vectors exist also on the division algebra side, allowing uniform treatment as in the matrix algebra case. 
In fact as minimal vectors exist for general supercuspidal representations of classical groups, one can get a necessary condition for the local period integral to be nonzero using stationary phase analysis.
 The main goal of this paper is to also get sufficient condition for the non-vanishing of the local integral, using a method which hopefully can be applied in more general settings. 

A simple observation in the case $c(\pi)=4n$ is that, as the matrix coefficients $\Phi$ of the minimal vectors are multiplicative on its support , then $\Phi\chi^{-1}$ must be constant $1$ on the common support for the local integral to be nontrivial.
In general the matrix coefficient for minimal vectors may not be completely multiplicative on the support, but we still get the following main theorem.
\begin{theo}\label{Theo:Main1}
Suppose  that $p$ is large enough. Let $\pi$ be a supercuspidal representation of $\GL_2$ with central character $w_\pi$ and $c(\pi)>2$, 
and $\pi^\B$ be the associated supercuspidal representation for a quaternion algebra $\B$ under the Jacquet-Langlands correspondence. (We allow $\B^\times =\GL_2$ and $\pi^\B=\pi$.) Let $\E$ be a quadratic extension over $\F$ embedded into $\B$ and $\chi$ be a character over $\E$ such that $w_\pi=\chi|_{\F^\times }$. 
Then there exists a nontrivial test vector for the local Waldspurger's period integral, if and only if there exists a \textbf{minimal vector }$\varphi^\B$ such that
\begin{equation}
\Phi_{\varphi^\B}\chi^{-1}=1 \text{\ on } \Supp \Phi_{\varphi^\B} \cap \E^\times .
\end{equation}
Moreover if $l$ is an integer such that $C(\pi\times\pi_{\chi^{-1}})=q^{c(\pi)+l}$, then
\begin{equation}
\max_{\text{minimal vector } \varphi'\in \pi^\B
}\{I(\varphi',\chi)\}\asymp_q I(\varphi^\B,\chi)=\Vol(\Supp
\Phi_{\varphi^\B} \cap \E^\times )\asymp_q \frac{1}{q^{l/4}}
\asymp_q
\left(
  \frac{C(\pi \times \bar{\pi})}{C(\pi\times\pi_{\chi^{-1}})} 
\right)^{1/4}.
\end{equation}

\end{theo}
\begin{rem}
\begin{enumerate}
\item Here the implied constant can be a bounded power of $q$. 
But the precise dependence on $q$ in the relation $I(\varphi^\B,\chi)\asymp_q \frac{1}{q^{l/4}}$ will be given for different cases in the proof of Proposition \ref{Prop:SizeofLocalint}, or in the Appendix.
\item In the case $\dim\Lambda=1$, for any minimal vector $\varphi'$ we actually have either $I(\varphi',\chi)=0$ or $I(\varphi',\chi)=I(\varphi^\B,\chi)$. When $\dim\Lambda>1$, our proof makes use of the particular basis for minimal vectors as in Definition \ref{Defn:generalMinimalVec}. 
Writing a general minimal vector as a linear combination of these basis, one can then easily obtain the part $$\max_{\text{minimal vector } \varphi'\in \pi^\B }\{I(\varphi',\chi)\}\asymp_q I(\varphi^\B,\chi).$$
\end{enumerate}

\end{rem}
We  sketch the ideas for the computations here. It appears natural to divide the discussions according to whether $\Supp \Phi_{\varphi^\B} \cap \E^\times =\E^\times $. 
When $\Supp \Phi_{\varphi^\B} \cap \E^\times =\E^\times $, $\E^\times $ must act on a type 2 minimal vector (as in Definition \ref{Defn:generalMinimalVec}) exactly by $\chi$, and this occurs only if $\theta $ and $\chi$ are defined over the same quadratic extension and $c(\theta\chi^{-1})$ or $c(\theta\overline{\chi}^{-1})\leq 1$. 
One can easily find test vectors from the construction of $\Lambda$ in this case. 
One the other hand when $\Supp \Phi_{\varphi^\B} \cap \E^\times \subsetneq \E^\times $, 
we show in Lemma \ref{Lem:JcapEwholeorwhat} that when $p$ is large enough,  $\Supp \Phi_{\varphi^\B} \cap \E^\times$ is within the Lie algebra range. 
See Definition \ref{Defn:LiealgRange} for more precise meaning. Then the character $\tilde{\theta}$ used to determine the minimal vector can be identified with an element $\alpha_\theta$ in the dual Lie algebra, which is associated to $\theta$ by Lemma \ref{Lem:DualLiealgForChar}(2) and embedded into $\B$. Let $O_{\pi^\B}=\{g^{-1}\alpha_\theta g\}$ be the coadjoint orbit associated to $\pi^\B$.
The main new idea is to represent the matrix coefficient $\Phi_{\varphi^\B}$ of a type 1 minimal vector in a sufficiently large neighbourhood
as an integral over a ball in the coadjoint orbit of
$\alpha_\theta$ as in Corollary \ref{lem7.1:char=int}, Lemma 
\ref{lem:char=intforB1}. This way one can  avoid the
discussion about $\Supp \Phi_{\varphi^\B}\cap \E^\times $, and obtain an equivalent description for the local integral to be nonzero in a soft and uniform fashion. In particular we obtain the following result:

\begin{theo}\label{Theo:Liealgformulation}
Suppose that the nonzero contribution to the local integral comes from Lie algebra range. 
Then there exists a nontrivial test vector for the local Waldspurger's period integral, if and only if there exists an element $\alpha$ in the coadjoint orbit $O_{\pi^\B}$ of $\alpha_\theta$ such that
\begin{equation}\label{eq:introLiealgCondition}
\alpha\in \alpha_\chi+\h_0^\dagger \cap O_{\pi^\B}.
\end{equation}
Here  $\h_0$ is a lattice in the Lie algebra for $\E$ as in Definition \ref{Defn:LiealgRange}, and $\h_0^\dagger$ is the dual module for $\h$ as in Definition \ref{Defn:dualLattice}. 
In that case we have as in the previous version
\begin{equation}
\max_{\text{minimal vector } \varphi'\in \pi^\B }\{I(\varphi',\chi)\}\asymp I(\varphi^\B,\chi)=\Vol(\Supp \Phi_{\varphi^\B} \cap \E^\times )\asymp_q \frac{1}{q^{l/4}}.
\end{equation}
\end{theo}

\begin{rem}
	When $p$ is large enough, the vague condition on nonzero contribution coming from Lie algebra will be replaced by a more precise condition ($*$) before Corollary \ref{Cor4.2:unrelatedcriterion}.
	
	In Lemma \ref{Lem4.2:unrelatedbetter}, we will also show that \eqref{eq:introLiealgCondition} is equivalent to the existence of $\alpha\in\alpha_\chi+\E^\perp\cap O_{\pi^\B}$.
\end{rem}

\begin{rem}
 One can take a more direct approach to parameterize the family of minimal vectors, carefully identify the common support $\Supp \Phi_{\varphi^\B} \cap \E^\times $ for them, and check whether $\Phi_{\varphi^\B}\chi^{-1}=1$ on the common support. 
 In this fashion, one can actually assume $p\neq 2$ only, and deal with the case $c(\pi)=2$ as well. But this method depends on explicit parametrizations and requires case by case computations. So for conciseness we leave them in Appendix and consider only certain settings.
\end{rem}
\begin{rem}
\eqref{eq:introLiealgCondition} encodes a lot of information regarding the relation between $\pi^\B$ and $\chi$.
We show in Proposition \ref{Prop:Dichotomy} that \eqref{eq:introLiealgCondition} has solution at exactly one side of the matrix algebra or division algebra.
Assuming multiplicity one, this is equivalent to the dichotomy that
\begin{equation}
\dim \Hom_{\E^\times }(\pi, \chi)+\dim \Hom_{\E^\times }(\pi^\B, \chi)=1.
\end{equation}
Furthermore, whether $\Phi_{\varphi^\B}\chi^{-1}=1$ on  $\Supp \Phi_{\varphi^\B} \cap \E^\times $ is possible, or whether there exists $\alpha$ in the coadjoint orbit $O_{\pi^\B}$ satisfying \eqref{eq:introLiealgCondition} are both equivalent to whether certain quadratic equation has solutions after proper parametrisation, which in turn is equivalent to  Tunnell-Saito's $\epsilon-$value test,
providing a geometric interpretation of the test in the related setting.
\end{rem}
\subsection{Application to  hybrid subconvexity bounds}
Now we show a global application for the test vector for Waldspurger's period integral we have developed so far. We shall work in the number field setting and use subscript $v$ to denote the local components in this subsection.

Let $\F$ be a totally real number field now.
Let $\pi$ be an automorphic representation of $\GL_2$ over $\F$ with trivial central character, and let $\chi$ be a Hecke character over a fixed quadratic field extension $\E$ such that $\chi|_{\A_\F^\times}=1$. Suppose that $L(\Pi\otimes\chi^{-1},1/2)\neq 0$, then by \cite{Walds} there exists a quaternion algebra $\B$ (as decided by Tunnell-Satio's $\epsilon-$value test)  such that for $\varphi^\B\in \pi^\B$, 
\begin{equation}\label{Intro1Eq:GlobalWalds}
\left|\int_{t\in [\F^\times\backslash\E^\times]}\varphi^\B(t)\chi^{-1}(t)dt\right|^2 \approx L(\Pi\otimes\chi^{-1},1/2)\prod\limits_{v}I^0_v(\varphi^\B,\chi).
\end{equation}
Here $\approx$ means up to some unimportant factors, and $I^0_v(\varphi^\B,\chi)$ is essentially the local Waldspurger's period integral as in Theorem \ref{Theo:Main1}, normalised by local L-factors. Both sides can be nontrivial with proper choice of test vectors.
Let $C(\pi),C(\pi_{\chi^{-1}})$ be the finite conductors of $\pi$ and $\pi_{\chi^{-1}}$.

The basic idea now, is to use Jacquet's relative trace formula for Waldspurger's period integral from \cite{jacquet_sur_1987}, under some specific settings from \cite{feigon_averages_2009}.
In accordance to the purpose of this paper, we restrict ourselves to the setting that $\pi_v$ is a supercuspidal representation at a fixed place $v=\mathfrak{p}$, and is unramified at other finite places. For simplicity, we shall also assume additional local conditions in Setting  \ref{Assumption:Localdisjoint}, \ref{Assumption:Local}, so that we can directly use the computations from \cite{feigon_averages_2009} at other places. For convenience, we consider two extreme scenarios, where $\pi$ and $\chi$ have either completely disjoint ramifications, or completely joint ramifications.
\begin{theo}\label{Theo:HybridsubDisjointram}
	For notations as above, suppose that $\pi$ and $\chi$ have  disjoint ramifications as in Setting \ref{Assumption:Localdisjoint}. Let $C(\pi)=C(\pi_{\chi^{-1}})^\delta$ for $0<\delta<\infty$. Then $C(\Pi\otimes\chi^{-1})=C(\pi_\chi)^{2(1+\delta)}$, and
	\begin{equation}\label{Eqintro:Hybridsubdisjoint}
	L(\Pi\otimes \chi^{-1},\frac{1}{2})\ll_{q,\epsilon} C(\pi_\chi)^{\frac{1}{2}\max\{\delta, \frac{\delta}{2}+1\}+\epsilon}.
	\end{equation}
	In particular we obtain a subconvexity bound for $L(\Pi\otimes \chi^{-1},\frac{1}{2})$ in the hybrid range $0<\delta<\infty$, which can be as strong as the Weyl bound when $\delta=2$.
\end{theo}
\begin{rem}
	The main new ingredient for this result is to choose the test function $f_\mathfrak{p}$ in the relative trace formula to be essentially the matrix coefficient of a minimal vector, so that the local test vector picked out on the spectral side is the test vector as in Theorem \ref{Theo:Main1}. A minor benefit for such choice is that there will be no residue spectrum or continuous spectrum. One also does not have to worry about old forms.
	
	The more important feature resulting from this choice of test vector is that the length of the spectral sum becomes significantly shorter, when compared to, for example, using the characteristic function of the congruence subgroup for a newform (ignoring the various issue about newforms being a test vector). Thus one can, in principle, obtain better bound as long as one can show that the error terms (more precisely, those coming from regular orbit integrals on the geometric side) are controlled by the main term (from irregular orbits). 
	This `benifit' is however balanced by the fact that the sum of the error terms becomes longer, and more complicated to control in certain ranges.
	
	This explains the difference between Theorem \ref{Theo:HybridsubDisjointram} and \cite[Theorem 1.4]{feigon_averages_2009}, upon which Theorem \ref{Theo:HybridsubDisjointram} is supposed to generalize. \cite[Theorem 1.4]{feigon_averages_2009} considers the disjoint ramification scenario but with  $C(\pi)$ square free. In our terminologies, their result translates into the following  bound
	\begin{equation}\label{Eqintro:Hybridsubdisjointold}
	L(\Pi\otimes \chi^{-1},\frac{1}{2})\ll_{\epsilon} C(\pi_\chi)^{\max\{\delta, \frac{1}{2}\}+\epsilon},
	\end{equation}
	which is a subconvexity bound when $0<\delta<1$. Our result \eqref{Eqintro:Hybridsubdisjoint} is weaker than \eqref{Eqintro:Hybridsubdisjointold} in the range $0<\delta<2/3$, but is stronger in the complementary range, and in particular gives a subconvexity bound when $\delta\geq 1$.
\end{rem}

 The flexibility of our test vector theory also allows us to consider the other extreme situation, when the ramifications of $\pi$ and $\chi$  are concentrated at the same place $\mathfrak{p}$.
\begin{theo}\label{Theo:Hybridsubconvexity}
Suppose that $\pi$ and $\chi$ are both ramified only above $\mathfrak{p}$ as in Setting \ref{Assumption:Local}, and suppose
that
\pn{double check that this is okay}
$C(\pi) < C(\pi_\chi)$,
so that we may write $C(\pi)= C(\pi_\chi)^\delta$ for some $0 < \delta < 1$. Then $C(\pi\otimes\pi_{\chi^{-1}})=C(\pi_\chi)^2$, and
\begin{equation}
L(\Pi\otimes \chi^{-1},\frac{1}{2})\ll_{q,\epsilon}  C(\pi_\chi)^{\frac{1}{2}(\max\{\delta, 1-\frac{\delta}{2}\})+\epsilon}.
\end{equation}
In particular we obtain a subconvexity bound for $L(\Pi\otimes \chi^{-1},\frac{1}{2})$ in the hybrid range $0<\delta<1$, which can be as strong as the Weyl bound when $\delta=2/3$.
\end{theo}
\begin{rem}
	It is not surprising that the above subconvexity bound does not extends to $\delta=1$, where the conductor $C(\pi\otimes\pi_{\chi^{-1}})$ could drop due to the possible relation between $\pi$ and $\pi_{\chi^{-1}}$, making the problem more difficult.
\end{rem}

\begin{rem}
	As indicated by the discussions above, the main technical task (and the most complicated part of this paper) is to control the (local) regular orbit integrals for our choice of test function. This is done in Section \ref{Section:LocalLemOrbitInt}. 
	We shall use similar language and strategy (though slightly more complicated) as for the proof of Theorem \ref{Theo:Liealgformulation}. Hopefully this will alleviate the complicated nature of the task for readers.
	
	It might be possible to further improve the control for the regular orbit integrals on the geometric side and get a sub-Weyl bound in proper range. We shall explore such possibility in future works.
\end{rem}

\begin{rem}
	Note that when $2<\delta<\infty$ in Theorem \ref{Theo:HybridsubDisjointram}, and when $2/3<\delta<1$ (and actually when $\delta>1$ as well) in Theorem \ref{Theo:Hybridsubconvexity}, our computations actually give  asymptotic formulae for the first moment of $L(\pi\times \pi_{\chi^{-1}},1/2)$ ( weighted by $L^{-1}(\pi,Ad,1)$, see \eqref{Sec5.1Eq:GlobalWalds}), with a power saving on error terms, from which one can further get nonvanishing results for the special values of $L-$functions.
	
	On the  same range of $\delta$,  one can further apply amplifications to balance the main term and the error terms, and improve the above subconvexity bounds (and also obtains subconvexity bound for $\delta>1$ in Theorem \ref{Theo:Hybridsubconvexity}).
	
	On the other hand for the complementary range of $\delta$, our method might be combined with proper spectral average (i.e. choose proper $f_\mathfrak{p}$ so that more supercuspidal representations can be picked out on spectral side) to further improve the bound, and possibly bridge the difference between, for example, Theorem \ref{Theo:HybridsubDisjointram} and \eqref{Eqintro:Hybridsubdisjointold}.
	
	We leave the details to the interested readers.
\end{rem}

%
%
%
\begin{rem}
	We expect the similar strategy to work when $\pi_\mathfrak{p}$ is a principal series representation. One can also apply the similar strategy for the relative trace formula as in \cite{RR:05} and obtain a hybrid subconvexity bound for the standard twisted $L-$function, though one probably need more accurate control for the Archimedean integrals, and the resulting hybrid subconvexity bound may not be as strong as in the non-split case.
	It is also interesting to see if the strategy works for more general Gan-Gross-Prasad pairs, which is of course more challenging if ever possible. 
\end{rem}
\begin{rem}
	Besides \cite{feigon_averages_2009} one can also compare it with several other known Weyl bounds in level aspect for $\GL_2$. In \cite{BlomerMili:15} \cite{Assing18}, the Weyl bound was obtained for $L(\pi\otimes\chi^{-1},1/2)$ for fixed $\pi$ and increasing depth for $\chi$. On the other hand, \cite{IanYoung18}\cite{IanYoung19} proved a Weyl bound for $L(\pi,1/2)$ without twists, but the representation $\pi_v$ at the place of interest is a principal series representation.
\end{rem}

\subsection{Organisation of the paper and acknowledgement}
The paper will be organised as follows. Section \ref{Sec:Notation} will introduce basic notations and review  Tunnell-Saito's theorem. In Section \ref{Sec:CompactinductionSummary+Kirillov} we summarise the compact induction theory and the properties of the minimal vectors. 
We also introduce the Lie algebra language and represent 
 the matrix coefficients for minimal vectors as integrals on coadjoint orbit.
In Section \ref{Sec:mainmethod} we present the soft and uniform method making use of the Lie algebra language when $p$ is large enough, discussing the dichotomy and Tunnell-Saito's $\epsilon-$value test along the way. 
In Section \ref{Sec:hybridSubConv} we recall the relative trace formula and prove the hybrid subconvexity bound, using the test vector developed so far.

The appendix contains authors' original approach without using Lie algebra language except Lemma \ref{Lem:DualLiealgForChar} (1).
In Appendix \ref{Sec:AppendixA} we use a particular embedding of $\L$, and give more details on matrix coefficient for compact induction, Kirillov model for minimal vectors, local Langlands/Jacquet-Langlands correspondence and explict Tunnell-Saito's $\epsilon-$value test.
In Appendix \ref{Sec:AppendixB}  we prove the variant Theorem \ref{Appendix:Maintheo} on the $\GL_2$ side case by case under slightly stronger conditions for $p\neq 2$.  We identify the common support and also check that whether the resulting quadratic equation has solutions is equivalent to Tunnell-Saito's $\epsilon-$value test. The integral itself is then simply the volume of the common support, consistent with the computations in Section \ref{Sec:mainmethod}. 

The authors are supported by SNF-169247.
Part of this work was completed
while the authors were
in residence at the
Mathematical Sciences Research Institute in Spring 2017 under National Science Foundation Grant No. DMS-1440140.

\section{Notations and preliminary results}\label{Sec:Notation}\label{Basic notations}
\subsection{Basics}
For a real number $a$, let $\lfloor a\rfloor\leq a$ be the largest possible integer, and $\lceil a\rceil \geq a$ be the smallest possible integer.

Let $\F$ be a p-adic field with  residue field $k_\F$ of order
$q$, uniformizer $\varpi=\varpi_\F$, ring of integers $O_\F$ and
p-adic valuation $v=v_\F$. Let $p$ be the characteristic of the
residue field.
We assume throughout the paper that $p\neq 2$. 
Let $\psi$ be an additive character of $\F$. 
$c(\psi)$ is the smallest integer such that  $\psi$ is trivial on $\varpi_\F^{c(\psi)}O_\F$. Usually $\psi$ is chosen to be unramified, or equivalently,   $c(\psi)=0$.
For $n\geq 1$, let $U_\F(n)=1+\varpi_\F^n O_\F$, and $U_\F(0)=O_\F^\times$.  

Let $\L$  be a quadratic field extension over $\F$. Let $e_\L=e(\L/\F)$ be the ramification index. Let $\varpi_\L$ be a uniformizer for $\L$ and $v_\L$ be the valuation on $\L$. When $\L$ is unramified we shall identify $\varpi_\L$ with $\varpi_\F$. Otherwise we suppose that $\varpi_\L^2=\varpi_\F$. Let $x\mapsto \overline{x}$ be the unique nontrivial involution of $\L/\F$.
 Let $\psi_\L=\psi\circ \Tr_{\L/\F}$. One can easily check that
\begin{equation}\label{Eq:CpsiL}
 c(\psi_\L)=-e_\L+1.
\end{equation}
For $\chi$ a multiplicative character on $\F^\times $, let $c(\chi)$ be  the smallest integer such that $\chi$ is trivial on $U_\F(c(\chi))$. 

Let $\E$ be an \'{e}tale
quadratic algebra over $\F$. When $\E$ is a field extension, one can define similar notations as above. 
  Note that we shall assume that $\varpi_\E^2=\xi\varpi_\F$ for $\xi\in O_\F^\times -(O_\F^\times)^2$ if $\E$ and $\L$ are both ramified and distinct.
 When $\E$ splits, we fix an isomorphism $\iota:\E\rightarrow \F\times\F$. 
 In this case, denote
 $\varpi_\E=\iota^{-1}(\varpi_\F,\varpi_\F)$, $O_\E=\iota^{-1}(O_\F\times O_\F)$, $U_\E(0)=O_\E^\times=\iota^{-1}(O_\F^\times\times O_\F^\times)$, and $U_\E(n)=\iota^{-1}(1+\varpi^n O_\F\times \varpi^n O_\F)$ for $n\geq 1$.

A character $\mu$ over a p-adic field $\F$ is a multiplicative function $\mu:\F^\times \rightarrow \C^\times$, which is not necessarily unitary.

 \pn{
   Would it be possible for us to somehow reduce
   all of our results
   to some convenient case, like $c(\psi) = 0$?
   Doing so would improve the readability here,
   I think.
 }
\yhn{I think we also need $\psi_\L$, so a general discussion is better}
 \begin{lem}\label{Lem:DualLiealgForChar}
For a character
  $\mu$ over $\F$ with $c(\mu)\geq 2$, there exists $\alpha_\mu\in \F^\times $ with $v_\F(\alpha_\mu)=-c(\mu)+c(\psi_\F)$ such that:
\begin{enumerate}
\item 
\begin{equation}
 \mu(1+u)=\psi_\F(\alpha_\mu u) \text{\ for $u\in \varpi_\F^{\lceil c(\mu)/2\rceil} O_\F$},
\end{equation}
\begin{equation}\label{eq:alphathetaless}
 \mu(1+u)=\psi_\F(\alpha_\mu (u-\frac{u^2}{2})) \text{\ for $u\in \varpi_\F^{\lfloor c(\mu)/2\rfloor} O_\F$}.
\end{equation}
\item Suppose that either $p$ is large enough and $i=1$, or $i$
  is large enough
. Then there
  exists
  \pn{This ``unit group''
    does not immediately make sense,
    so we should define it separately.
    Also, I find the notation as a whole a bit ``noisy'':
    I would prefer simply $\mathfrak{p}^n$ or perhaps
    $\mathfrak{p}_F^n$
    to $\varpi^n \mathcal{O}_F$.}\yhn{$\mathfrak{p}^n$ or
    $\mathfrak{p}_F^n$ are nice but changing to these require a large amount of work. So I will leave them alone for the moment.}
  $\alpha_\mu\in \varpi^{-c(\mu)+c(\psi_\F)}O_\F^\times/U_\F(c(\mu)-i)$ such that
\begin{equation*}
\mu(1+u)=\psi_\F(\alpha_\mu \log(1+u)), \forall u\in \varpi_\F^i O_\F,
\end{equation*}
where $\log(1+u)$ is defined by the standard Taylor expansion for logarithm 
\begin{equation*}
\log(1+u)=u-\frac{u^2}{2}+\frac{u^3}{3}+\cdots.
\end{equation*}

\end{enumerate}

\end{lem}
Note that the condition $p$ or $i$ large enough guarantees that the $\log$ map  is a group isomorphism from the multiplicative group $U_\F(i)$ to the additive group $\varpi^i O_\F$. Part (2) directly recovers part (1) once $u$ is inside the corresponding domains.



For a multiplicative character $\theta $ over $\L$ we can associate a similar element $\alpha_\theta\in \varpi_\L^{-c(\theta)-e_\L+1} O_\L^\times $.

We also need a lemma in the inverse direction. 
\begin{lem}\label{Lem:alphagivechi}
Assume $p$ or $i$ to be large enough as in Lemma \ref{Lem:DualLiealgForChar}(2).
For any $n\geq 2$ and $\alpha\in\F^\times$ with $v_\F(\alpha)=-n+c(\psi_\F)$, there exists a character $\mu$ of $\F^\times$ such that $c(\mu)=n$, and
\begin{equation}\label{Eq:characterliner2}
 \mu(1+u)=\psi_\F(\alpha \log(1+u)) \text{\ for $u\in \varpi_\F^{i } O_\F$}.
\end{equation}
\end{lem}
\begin{proof}
One can easily verify by counting that the association $\mu \mapsto \alpha_\mu$ satisfying \eqref{Eq:characterliner2}
 gives an isomorphism $$\{\mu\}_{c(\mu)= n}/\{\mu\}_{ c(\mu)\leq i }\rightarrow(\varpi^{-n+c(\psi_\F)}O_\F/\varpi^{-i+c(\psi_\F)} O_\F)^\times.$$
\end{proof}

Let $\B$ be a quaternion algebra, either the matrix algebra or the division algebra. It is equipped with the standard involution $x\mapsto \overline{x}$, which is consistent with the involution of any embedded quadratic algebra $\E$ in $\B$.
Let $\Tr$, $\Nm$ be the standard trace and norm on $\B$. Let $G=\B^\times$, and $Z$  be the center of $G$. 
Let $\pi$ be a smooth irreducible unitary representation of $\GL_2(\F)$  with central character $w_\pi$, and let $c(\pi)$ be the power of the conductor of $\pi$. 
Let $\pi^\B$ be the image of $\pi$ under the Jacquet-Langlands correspondence.
\begin{defn}\label{Defn:dualLattice}
For $x,y\in \B$, there is a non-degenerate pairing given by $$<x,y>=\Tr(xy).$$
For a fixed non-trivial additive character $\psi$, denote $$e^{<x,y>}=\psi(<x,y>).$$
For fixed $\psi$ and any $O_\F$-module $A$ of $\B$, denote
the dual module
$$A^\dagger=\{x\in \B,e^{<x,a>}=1 \forall a\in A\}.$$
When $\E\subset \B$ is a vector subspace, denote $$\E^\perp=\{x\in \B, <x,y>=0\text{\ for any }y\in \E\}.$$ 
Note that $\E^\perp=\E^\dagger$ in this case.
\end{defn}

We collect some basic properties of quaternion algebras with respect to a quadratic sub-algebra.
\begin{lem}\label{Lem2.1:Eperpbasic}
	Fix an embedding of an \'{e}tale quadratic algebra $\L$  
in $\B$. Then for any $x_1, x_2\in\L^\perp$, we have $x_1 x_2\in \L$.  For any $x=x_\L+x^\perp$ with $x_\L\in \L$ and $x^\perp\in \L^\perp$,
	\[\Nm(x)=\Nm(x_\L)+\Nm(x^\perp).\] Furthermore,
	there exists $j$ such that 
	$j\in \L^\perp$, $\Nm(j)\neq 0$. Then we have the following:
	\begin{enumerate}
		\item $\overline{j}=-j$ (in particular  $j^2\in \F$), $l j=j\text{\ } \overline{l}$ for any $l\in \L$.
		\item 	$\B=\L+\L j$,  $\L^\perp=\L j$, with involution on $\B$ given by $l_1+l_2j\mapsto \overline{l_1}-l_2j$ for $l_i\in \L$.  
	\end{enumerate}

\end{lem}
\begin{proof}
	First of all, assuming the existence of $j\in \L^\perp$ with $\Nm(j)\neq 0$, one can easily check the remaining claims. For example, by Definition \ref{Defn:dualLattice}, $j\perp l$ for $l\in \L$ implies that $lj=-\overline{lj}=-\overline{j}\text{\ }\overline{l}=j\text{\ }\overline{l}$, where the last step follows from that $j\perp 1$.
	
	 $\L^\perp=\L j$ because  on one hand, $\L j$ is a two dimensional subspace (since multiplication by $j$ is an invertible linear transformation as $\Nm(j)\neq 0$), while $$<l, l'j>=\Tr(ll'j)=<ll', j>=0, \forall l,l'\in \L.$$
	Thus $\L j\subset \L^\perp$. On the other hand, by the non-degeneracy of the trace pairing in Definition \ref{Defn:dualLattice}, $\L^\perp$ is two dimensional, forcing $\L^\perp=\L j$.
	
	For $x_1=l_1 j, x_2=l_2 j\in \L^\perp$, we have $x_1x_2=l_1jl_2j=l_1\overline{l_2} j^2\in \L$.
	
	Now we show the existence of $j\in \L^\perp$ with $\Nm(j)\neq 0$. When $\B$ is the division algebra, this is obvious as any nontrivial element $x\in \B$ has $\Nm(x)\neq 0$. For the matrix algebra, we can check case by case for some fixed embeddings, since all embeddings are conjugate to each other,  while the statements are independent of the conjugations.
	
	If $\L$ is a field we can assume after a conjugation that $\L=\F(\sqrt{D})$ is embedded into $\B$ as
	\begin{equation}\label{Eq:StandardEmbeddingL}
	a+b\sqrt{D}\mapsto \zxz{a}{b}{bD}{a}.
	\end{equation}
	Then one can pick $j=\zxz{-1}{0}{0}{1}$. If $\L\simeq\F\times \F$,  we can assume after a conjugation that $\L$ is the diagonal torus, while picking $j=\zxz{0}{1}{1}{0}$.
\end{proof}
\begin{lem}\label{Lem:ConjugacyRep}
	Let $f(x)=x^2+ax+b\in \F[x]$ be a quadratic polynomial which is either irreducible or has distinct roots. Then for any $x_1,x_2\in \B^\times$ with given characteristic polynomial $f(x)$, they are conjugate to each other.
\end{lem}
\pn{The lemma is an immediate consequence of the Skolem--Noether
  theorem.
I trust Vigernas's book mentions this somewhere}
\yhn{I googled Skolem--Noether   theorem, which do seem relavant. But I think it is not exactly the same. I think I will keep this lemma for the moment.}
\begin{proof}
	Consider first the case $\B^\times=\GL_2(\F)$, which acts on a 2-dimensional vector space $V$ over $\F$. For $x_i$ with given characteristic polynomial $f(x)$, $x_i$ can not be a scalar by the condition on $f$. Then there exists a non-trivial element $e_2$ such that $e_1=x_i\cdot  e_2$ is not a scalar multiple of $e_2$. Choosing $(e_1,e_2)$ as a new basis, we see that $x_i$ is conjugate to $\zxz{\Tr(x_i)}{1}{-\Nm(x_i)}{0}$. 
	
	When $\B$ is the division algebra, we fix a quadratic field extension $\L$ in $\B$, and let $\B^\times$ acts on $\B/\L$ by the usual multiplication. This gives rise to an embedding $\B^\times \mapsto \GL_2(\L)$. Then by the same argument as above, we get that $x_i$ is conjugate to $\zxz{\Tr(x_i)}{1}{-\Nm(x_i)}{0}\in \GL_2(\L)$ by elements in $\B^\times$.
\end{proof}
\begin{defn}
For $i=1,2$, let $H_i$ be compact open subgroups of $G$ and $\rho_i$ be irreducible representations of $H_i$. We say $g\in G$ intertwines $\rho_1$ with $\rho_2$ if
$$\Hom_{H_1\cap H_2^g}(\rho_1,\rho_2^g)\neq 0.$$
Here $H_2^g=g^{-1}H_2 g$ and $\rho_2^g$ is the representation of $H_2^g$ on the same space as $\rho_2$ with the action given by $\rho_2^g(g^{-1}hg)=\rho_2(h)$. In the case $H_1=H_2=H$ and $\rho_1=\rho_2=\rho$, we simply say $g$ intertwines $\rho$.
\end{defn}

We also recall Tunnell-Saito's $\epsilon-$value test.

\begin{theo}[\cite{Tunnell:83a} \cite{Saito:93a}]\label{Tunnell}
Let $\chi$ be a character of $\E^\times$ such that $\chi|_{\F^\times }=w_\pi$.
The space $\Hom_{\E^\times }(\pi^\B\otimes \chi^{-1},\C)$ is at most one-dimensional. It is nonzero if and only if 
\begin{equation}
\epsilon(\pi_\E\times\chi^{-1})=\chi(-1)\epsilon(\B).
\end{equation}
Here $\pi_\E$ is the base change of $\pi$ to $\E$. $\epsilon(\B)=1$ if $\B$ is a matrix algebra, and $-1$ if it is a division algebra.
\end{theo}


\subsection{Imaginary and minimal elements}
\begin{defn}\label{Sec2.1Def:imaginarypart}
Let $\E$ be an \'{e}tale quadratic algebra over $\F$ and $x\mapsto \overline{x}$ be the nontrivial involution on $\E$. $x\in \E^\times$ is called imaginary  if $\overline{x}=-x$. For general $x$, denote 
\begin{equation}\label{Sec2.1Eq:trace0part}
x_0=x-\frac{\Tr(x)}{2}
\end{equation} to be the imaginary part (or trace 0 part) of $x$.
\end{defn}
It is straightforward to verify that
\begin{equation}\label{Sec2.1Eq:Nmxandx0}
\Nm(x)=\Nm(x_0)+\frac{\Tr(x)^2}{4}=\frac{\Tr(x)^2}{4}-x_0^2.
\end{equation}

\begin{lem}\label{Sec2.1Lem:disjointimaginaryNm}
Suppose 
that $\E\not\simeq \L$ are two distinct quadratic algebras. Let $\E_0^\times$ (respectively $\L_0^\times$) be the set of nontrivial imaginary elements of $\E$ (respectively $\L$). 
Then for any $x\in \Nm(\E_0^\times), y\in \Nm(\L_0^\times)$, we
have
\begin{equation}
v(x-y)=\min\{v(x),v(y)\}.
\end{equation}
\end{lem}
\begin{proof}
We know for $p-$adic valuations $v(x-y)\geq \min\{v(x),v(y)\}$. It suffices to show that there is no congruence between $x$ and $y$.
Let $\mathcal{E}=\{1,D, \varpi, \varpi D\}$ be a set of representatives of $\F^\times/(\F^\times)^2$, with $v(D)=0$. Then according to $\E$ being different quadratic algebras, $-\Nm(\E_0^\times)= \epsilon (\F^\times)^2$ for $\epsilon \in \mathcal{E}$ have disjoint images. For example, the split case for $\E$ corresponds to $\epsilon=1$, the inert quadratic field extension case corresponds to $\epsilon =D$.
\end{proof}

Now we talk about minimal characters/representations and minimal elements.
Let $\widehat{\F^\times}$ be the set of
characters on $\F^\times$. 
For  $\mu\in \widehat{\F^\times}$, 
let $\mu_\L$ denote the character $\mu\circ\Nm_{\L/\F}$, and
let $\pi\otimes \mu$ denote the representation of $G$ on the same space of representation as $\pi$ with the action given by $\pi\otimes\mu(g)=\mu(\det g)\pi(g)$.
\begin{defn}\label{Defn:minimalcharrep}
Let $\L$ be a quadratic field extension over $\F$, and $\chi$ be a character of $\L^\times$. Then $\chi$ is called minimal if $c(\chi)=\min\{c(\chi\mu_\L)| \mu \in \widehat{\F^\times}\}$. 
Let $\pi$ be a smooth irreducible 
representation of $G$.
Then $\pi$ is called minimal if $c(\pi)=\min\{c(\pi\otimes\mu)| \mu \in\widehat{\F^\times}\}$.
\end{defn}
We shall see later on in the compact induction theory that a character $\theta$ of $\L^\times$ is minimal if and only if its associated supercuspidal representation is minimal.

\begin{defn}\label{Defn:imaginary}
Let $\L$ be a quadratic field extension and $x\in \L$. $x$ is said to be \imaginary\text{\ } if
$v_\L(x)=v_\L(x_0)$.
\end{defn}

\begin{cor}\label{Sec2.1Cor:vofimaginary}
	If $\L$ is a field and $x\in \L^\times$ is minimal, then for  $\forall a,b \in \mathbb{F}$,
 $$v_\L(a+bx)=\min\{v_\L(a),v_\L(bx)\}.$$
\end{cor}
\begin{proof}
	We first prove the claim for $x$ imaginary.
	Note that
	$$v_\L(a+bx)=\frac{e_\L}{2}v(\Nm(a+bx))=\frac{e_\L}{2}v(\Nm(bx)-(-a^2)).$$
	Here $-a^2$ can viewed as an element in $ \Nm(\E_0^\times)$ for a split quadratic algebra $\E$. Then the claim follows from Lemma \ref{Sec2.1Lem:disjointimaginaryNm}.
	
	In general when $x$ is minimal, we can write $x=x_0+\Tr(x)/2$ with $v_\L(x_0)=v_\L(x)\leq v_\L(\Tr(x)/2)$. Then $a+bx=(a+b\Tr(x)/2)+bx_0$. We have either $v_\L(a)\geq v_\L(bx)=v_\L(bx_0)$, in which case $v_\L(a+bx)=v_\L(bx_0)=\min\{v_\L(a),v_\L(bx)\}$ by the imaginary case; or we have $v_\L(a)< v_\L(bx)$, in which case $v_\L(a+bx)=v_\L(a+b\Tr(x)/2)=v_\L(a)=\min\{v_\L(a),v_\L(bx)\}$. In either case, the claim is true.
\end{proof}

We show now the relation of a minima character and a minimal element.

\begin{lem}\label{Lem:MinimalEqv}
Let $\theta$ be a character over $\L$ with
$c(\theta)\geq 2$. Let $\alpha_\theta$ be associated to $\theta$ as in Lemma \ref{Lem:DualLiealgForChar}(2). Then
$\theta$ is minimal iff $\alpha_\theta$ is minimal.
\end{lem}
\begin{proof}
First of all, note that if $\alpha_\mu$ is associated to a character $\mu$ of $\F^\times$ by Lemma \ref{Lem:DualLiealgForChar}(2), then it is also the constant associated to $\mu_\L$, as
$$\mu_{\L}(1+u)=\mu((1+u)(1+\overline{u}))=\psi(\alpha_\mu \log((1+u)(1+\overline{u})))=\psi_\L(\alpha_\mu \log(1+u)).$$
Note that $c(\theta\mu_\L)=-v_\L(\alpha_{\theta\mu_\L})+c(\psi_\L)$, and $\alpha_{\theta\mu_\L}=\alpha_\theta+\alpha_\mu$. Thus $\theta$ is minimal iff for any $\alpha_\mu\in\F^\times$, 
\begin{equation}
v_\L(\alpha_\theta+\alpha_\mu)\leq v_\L(\alpha_\theta).
\end{equation}
Here we have used Lemma \ref{Lem:DualLiealgForChar} for one direction, and Lemma \ref{Lem:alphagivechi} for the other. 
Equivalently, $\theta$ is minimal if and only if
\begin{equation}
\max\{v_\L(\alpha_\theta+\alpha_\mu)\}=\max\{v_\L((\alpha_\theta-\alpha_{\theta,0}+\alpha_\mu)+\alpha_{\theta,0})\}=v_\L(\alpha_{\theta,0})\leq v_\L(\alpha_\theta).
\end{equation}
Here the second equality follows from Corollary \ref{Sec2.1Cor:vofimaginary}, as $\alpha_\theta-\alpha_{\theta,0}+\alpha_\mu\in \F$ and $\alpha_{\theta,0}$ is imaginary.

On the other hand,
 $$v_\L(x_0)=v_\L(x-\overline{x})\geq \min\{v_\L(x), v_\L(\overline{x})\}=v_\L(x)$$ is always true. 

Thus $\theta$ is minimal if and only if $v_\L(\alpha_\theta)=v_\L(\alpha_{\theta,0})$, which by definition is equivalent to $\alpha_\theta$ being minimal.
\end{proof}

\begin{rem}
Definition \ref{Defn:minimalcharrep} and \ref{Defn:imaginary} for the minimalities of a character and an element are not exactly the same as in \cite{BushnellHenniart:06a}. But one can easily check that Definition \ref{Defn:imaginary} is equivalent to \cite[Definition 13.4]{BushnellHenniart:06a} for an element. Then by Lemma \ref{Lem:MinimalEqv} and \cite[Proposition 18.2]{BushnellHenniart:06a}, the definitions for the minimality of a character are also equivalent. 
\end{rem}

\subsection{Normalization of measures}
The Haar measure on $\F^\times\backslash \E^\times$ will be normalized so that 

$$\Vol(O_\F^\times\backslash O_\E^\times)=1.$$
\begin{lem}\label{Sec2.3Lem:VolumeofE}

For any $x\in \E^\times$ and integer $m>0$,
\begin{equation}
\Vol(\F^\times\backslash \F^\times U_\E(e_\E m +e_\E-1)x)=\Vol(\F^\times\backslash \F^\times U_\E(e_\E m +e_\E-1))=
\frac{1}{q^m}L(\eta_{\E/\F},1).
\end{equation}
Here $\eta_{\E/\F}$ is the quadratic character associated to the extension $\E/\F$ and $L(\eta_{\E/\F},s)$ is the associated local $L-$factor. Note that when $\E$ is ramified, $\F^\times U_\E(2n)=\F^\times  U_\E(2n+1)$.
\end{lem}
\begin{proof}

The first equality follows from the property of a Haar measure. For the second equality, 
let $\beta\in \E$ be an imaginary element with
$v(\Nm(\beta))=e_\E-1$ (where $e_\E=1$ if $\E$ splits).
Note that $\F^\times U_\E(e_\E m +e_\E-1)=\F^\times\{a+b\beta| v(a)=0, v(b)\geq m\}.$
So the lemma amounts to a case by case check for the index of the set $ \F^\times\{a+b\beta| v(a)=0, v(b)\geq m\}$ in $\F^\times O_\E^\times$. We leave the details to interested readers.
\end{proof}

Denote $\E^1=\{e\in \E^\times, \Nm(e)=1\}$. By Hilbert 90 when $\E$ is a field, or direct computation when $\E$ splits, we get an isomorphism $\gamma: \F^\times \backslash \E^\times \rightarrow \E^1$, $x\mapsto x/\overline{x}$. Let
the Haar measure on $\E^1$ be normalized so that its pull back to $\F^\times \backslash \E^\times$ agrees with the Haar measure normalized as above.
\begin{lem}\label{Sec2.3Lem:VolumeE1}
For any $x\in \E^1$ and $m>0$,
\begin{align}
\Vol(\E^1\cap  x U_\E(e_\E m+e_\E-1))=\Vol(\E^1\cap   U_\E(e_\E m+e_\E-1))=\Vol(\F^\times\backslash \F^\times U_\E(e_\E m +e_\E-1))
\end{align}
\end{lem}

\begin{proof}
The first equality follows again from the property of the Haar measure.
For the second equality it suffices to check that $\gamma^{-1}( U_\E(e_\E m+e_\E-1))=\F^\times U_\E(e_\E m+e_\E-1)$. It is clear that $\gamma^{-1}( U_\E(e_\E m+e_\E-1))\supset \F^\times U_\E(e_\E m+e_\E-1)$. For the other direction, let $x=a+b\beta\in \E^\times$ be such that $x/\overline{x}\in U_\E(e_\E m+e_\E-1)$. 
Then either $v(a)\leq v(b)$, or $v(a)>v(b)$. Recall that $0\leq v(\Nm(\beta))=e_\E-1\leq 1$. When $v(a)>v(b)$, one can check that $x/\overline{x}\in -U_\E(1)$, and $-U_\E(1)\cap U_\E(1)=\emptyset$ when $p\neq 2$.
So up to a constant, one can assume that $v(a)=0, v(b)\geq 0$.

Now from $x/\overline{x}\in U_\E(e_\E m+e_\E-1)$, we get that
\begin{equation}
a+b\beta\equiv a-b\beta \mod \varpi_\E^{e_\E m+e_\E-1},
\end{equation}
i.e. $b\beta\equiv 0 \mod \varpi_\E^{e_\E m+e_\E-1}$. This implies that $v(b)\geq m$ and thus the claim is true.
\end{proof}

\section{Compact induction theory and Kirillov formula in a neighbourhood}\label{Sec:CompactinductionSummary+Kirillov}
\subsection{An overview of compact induction theory}\label{Sec:CompactInd}
Our equivalent formulations here differ slightly from those in \cite{BushnellHenniart:06a}, as we shall introduce  semi-valuations on $\B$ first (with different normalizations in some cases) and then the compact open subgroups, while \cite{BushnellHenniart:06a} introduced them in reverse order. The reason for this reformulation is that we would like to emphasize the properties of the semi-valuations, which are critical for later computations. 
\begin{defn}\label{Defn3.1:semivaluation}
	For any fixed embedding of a quadratic field extension $\L\hookrightarrow \B$, define a $\frac{1}{2}\Z-$valued function $\vv$ on $\B$ as follows:
	\begin{enumerate}
		\item[(i)] $\vv(l)=v_\L(l)$ for any $l\in \L$,
		\item[(ii)] $\vv(l^\perp)=\frac{1}{2}v_\L((l^\perp)^2)$ for any $l^\perp\in \L^\perp$.
		\item[(iii)] In general for  $x=x_\L+x^\perp\in \B$ with $x_\L\in \L$ and $x^\perp\in \L^\perp$,
		\[\vv(x)=\min\{\vv(x_\L),\vv(x^\perp)\}.\]
	\end{enumerate}
\end{defn}
 Note that $(l^\perp)^2\in \L$ by Lemma \ref{Lem2.1:Eperpbasic} so (ii) makes sense. Actually when we write $l^\perp=l j$ for $j$ as in Lemma \ref{Lem2.1:Eperpbasic} and some $l\in\L$, then we have
 \begin{equation}\label{Eq3.1:valuationhalf}
 \vv(l^\perp)=\frac{1}{2}v_\L((l^\perp)^2)=v_\L(l)+\frac{1}{2}v_\L(j^2)=v_\L(l)+\vv(j).
 \end{equation}
\begin{lem}\label{Lem3.1semivaluation}
$\vv$ is a semi-valuation on $\B$ satisfying the following properties:
\begin{enumerate}
	\item $\vv(xl)=\vv(lx)=v_\L(l)\vv(x)$ for any $l\in\L$. $\vv(\overline{x})=\vv(x)$.
	\item For any $x,y\in \L^\perp$, we have $v_\L(xy)=\vv(x)+\vv(y)$, $\vv(x)=\frac{e_\L}{2}v_\F(\Nm(x))$.
	\item $\vv(xy)\geq \vv(x)+\vv(y)$.  
	\item $\vv(x+y)\geq\min\{\vv(x),\vv(y)\}$
\end{enumerate}
\end{lem}
\begin{proof}
	(1) and (2) follow directly from the definition, Lemma \ref{Lem2.1:Eperpbasic} and \eqref{Eq3.1:valuationhalf}.
	Now we check (3) and  (4).
	Write $x=l_1+l_2j$ and $y=l_3+l_4j$ for $j$ as in Lemma \ref{Lem2.1:Eperpbasic}. By definition and the property of $v_\L$, 
	\begin{align*}
	\vv(xy)&=\vv((l_1l_3+l_2\overline{l_4}j^2)+(l_2\overline{l_3}+l_1l_4)j)=\min\{v_\L(l_1l_3+l_2\overline{l_4}j^2), v_\L(l_2\overline{l_3}+l_1l_4)+\vv(j)\}\notag\\
	&\geq\min\{v_\L(l_1l_3), v_\L(l_2\overline{l_4}j^2), v_\L(l_2\overline{l_3})+\vv(j), v_\L(l_1l_4)+\vv(j)\}\notag\\
	&=\min\{v_\L(l_1),\vv(l_2 j)\}+\min\{v_\L(l_3),\vv(l_4j)\}=\vv(x)+\vv(y).\notag
	\end{align*}
	\begin{align*}
	\vv(x+y)&=\min\{ v_\L(l_1+l_3), v_\L(l_2+l_4)+\vv(j)\}\geq \min\{
v_\L(l_1), v_\L(l_3), v_\L(l_2)+\vv(j),v_\L(l_4)+\vv(j)\}\\
&=\min\{\min\{v_\L(l_1),v_\L(l_2)+\vv(j)\}, \min\{v_\L(l_2),v_\L(l_4)+\vv(j)\}\}=\min\{\vv(x),\vv(y)\} \notag	\end{align*}

\end{proof}
\eqref{Eq3.1:valuationhalf} motivates us to make the following definition:
\begin{defn}
	Let $\epsiBL\in \{0,1\}$ be such that $\epsiBL\equiv v_\L(j^2)\mod(2)$.
\end{defn}
By \eqref{Eq3.1:valuationhalf}, $\frac{\epsiBL}{2}$ measures the difference between the sets $\vv(\L^\times)$ and $\vv((\L^{\perp})^{\times})$, or equivalently,
$\epsiBL$ measures the difference between the sets $v_\L((\Nm(\L))^\times)$ and $v_\L((\Nm(\L^\perp))^\times)$.

Note that by multiplying an element in $\L^\times$, we can in particular choose $j$ such that $v_\L(j^2)=\epsiBL$. 

\begin{lem}
	\begin{equation}\label{Sec3.1Eq:vjsquare}
	\epsiBL=\begin{cases}
	0, \text{\ if $\B$ is the matrix algebra},\\
	2-e_\L, \text{\ if $\B$ is the division algebra.}
	\end{cases}
	\end{equation}
\end{lem}
\begin{proof}
	As $\Nm(lj)=\Nm(l)\Nm(j)$, $(\Nm(\L^\perp))^\times$ is a single coset of $\Nm(\L^\times)$ in $\F^\times$. Recall that $\B$ is the matrix algebra iff $0=\Nm(x)$ for some non-trivial $x\in \B$. Since $\Nm(x)=\Nm(x_\L)+\Nm(x^\perp)$ by Lemma \ref{Lem2.1:Eperpbasic}, we have that $(\Nm(\L^\perp))^\times=-\Nm(\L^\times)$ when $\B$ is the matrix algebra, and the complementary of $-\Nm(\L^\times)$ in $\F^\times$ when $\B$ is the division algebra. Then the lemma is easy to check case by case.
\end{proof}
\begin{rem}
	From this lemma we see that $\vv$ takes integer values except when $\B$ is the division algebra and $e_\L=1$, in which case it takes half integer values.
\end{rem}
\begin{defn}\label{Defn3.1:orders+groups}
	The semi-valuation $\vv$ further gives rise to a chain order $\mathfrak{A}=\{b\in \B,\vv(b)\geq 0\}$ and its ideals $\mathcal{B}^n=\{b\in \B,\vv(b)\geq n\}$, and a filtration of compact open subgroups $$K_\mathfrak{A}(0)=\mathfrak{A}^\times, K_\mathfrak{A}(n)=I+\mathcal{B}^n\text{\ for }n\geq 0, $$ 
	all of which are normalised by $\L^\times $. 
	
\end{defn}


\begin{lem}\label{Cor3.1:psitrace=1}
	If $\vv(x)>c(\psi_\L)-1$, then $\psi\circ\Tr(x)=1$.
\end{lem}
\begin{proof}
	Write $x=x_\L+x^\perp$ for $x_\L\in \L$ and $x^\perp\in \L^\perp$. Then $\vv(x)>c(\psi_\L)-1$ implies that $\vv(x_\L)\geq c(\psi_\L)$. Thus $\psi\circ\Tr(x)=\psi_\L(x_\L)=1$.
\end{proof}
\begin{cor}\label{Sec3.1Cor:DualofBn}
	In the notation of Definition \ref{Defn:dualLattice}, we have 
	\begin{equation}
	(\mathcal{B}^n)^\dagger=\mathcal{B}^{-n+c(\psi_\L)}.
	\end{equation}
\end{cor}

\begin{proof}
	To show $(\mathcal{B}^n)^\dagger\supset\mathcal{B}^{-n+c(\psi_\L)},$ let $x\in \mathcal{B}^n$, $y\in\mathcal{B}^{-n+c(\psi_\L)}$. Then $xy=l_1+l_1^\perp$ for $l_1\in \L$ and $l_1^\perp\in \L^\perp$. Then by Definition \ref{Defn3.1:semivaluation} and Lemma \ref{Lem3.1semivaluation}, we have $\vv(l_1)\geq \vv(xy)\geq \vv(x)+\vv(y)\geq c(\psi_\L)$. Thus
	\begin{equation}
	e^{<x,y>}=\psi\circ\Tr(l_1)=\psi_\L(l_1)=1.
	\end{equation}
	Conversely let $z=l_2+l_2^\perp \in (\mathcal{B}^n)^\dagger$. Suppose that $\vv(l_2)< -n+c(\psi_\L)$. Then we can pick $l_3$ such that $v_\L(l_3)\geq n$ and $e^{<z,l_3>}=\psi_\L(l_2l_3)\neq 1$, contradiction.
	Similarly if  $ \vv(l_2^\perp)< -n+c(\psi_\L)$, we can pick $l_3^\perp$ such that $\vv(l_3^\perp)\geq n$ and $e^{<z,l_3^\perp>}=\psi_\L(l_2^\perp l_3^\perp)\neq 1$. Here we have used Lemma \ref{Lem3.1semivaluation}(2).
	
\end{proof}

\begin{example}\label{Example:explicitideal}
	When $\B$ is the matrix algebra $M_2(\F)$,  $e_\L=1$ and $\L$ is embedded as in \eqref{Eq:StandardEmbeddingL} with  $v_\F(D)=0$, we have that $\mathfrak{A}=M_2(O_\F)$ and  $\mathcal{B}=\varpi M_2(O_\F)$. When $e_\L=2$ and $\L$ is embedded similarly with $v_\F(D)=1$, we have that $\mathfrak{A}=\zxz{O_\F}{O_\F}{\varpi O_\F}{O_\F}$  and $\mathcal{B}=\zxz{\varpi O_\F}{O_\F}{\varpi O_\F}{\varpi O_\F}$. 
	
	When $\B$ is the division algebra $\D$, it is equipped with a valuation $v_\D$. Then $\vv$ is actually  the scaled valuation $\frac{e_\L v_\D}{2}$, and the chain orders $\mathfrak{A}=\{x,v_\D(x)\geq 0\}$ defined for inert and ramified quadratic field extensions are the same.
\end{example}
\begin{rem}\label{Rem3.1:OrderToValuation}
	Alternatively, as in \cite{BushnellHenniart:06a}, one can start with a general definition for a chain order $\mathfrak{A}$, and then construct $\mathcal{B}$ and define a semi-valuation $v_{\B,\mathfrak{A}}(x)$ to be the minimal integer $n$ such that $x\in \mathcal{B}^n$. ($v_{\B,\mathfrak{A}}(x)$ will come from some quadratic field extension $\L$ and $\vv$, though the normalisation is different in the case $\epsilon(\B)=-1$ and $e_\L=1$.) 
\end{rem}
It is natural to raise the question that when two quadratic field extensions $\L$ and $\E$ give rise to the same semi-valuation. As discussed in Example \ref{Example:explicitideal}, when $\B$ is the division algebra, this happens if and only if $e_\L=e_\E$. For the matrix algebra, we have the following criterion:

\begin{lem}\label{Lem3.1:whensamevalua}
Let $\L, \E$ be two quadratic field extensions over $\F$, embedded into $\B$. Let $\mathfrak{A}_\L$ $\mathfrak{A}_\E$ and other compact open subgroups be defined for $\L,\E$ respectively as in Definition \ref{Defn3.1:orders+groups}. Then $\vv=v_{\B,\E}	$ iff $\E^\times \subset \L^\times \mathfrak{A}_\L^\times$ iff $\L^\times \subset \E^\times\mathfrak{A}_\E^\times$.
	
\end{lem}
\begin{proof}
	By symmetry, we only prove that $\vv=v_{\B,\E}	$ if and only if $\E^\times \subset \L^\times \mathfrak{A}_\L^\times$.
	First of all, suppose that $\vv=v_{\B,\E}$. For any $e\in \E^\times$, let $l\in \L^\times$ such that $v_\L(l)=v_\E(e)$. Thus it suffices to show that $l^{-1}e\in \mathfrak{A}_\L^\times$. 
	Indeed $\vv(l^{-1}e)\geq \vv(l^{-1})+\vv(e)=-v_\L(l)+v_\E(e)=0$, while 
	$\vv(e^{-1}l)\geq -v_\E(e)+v_\L(l)=0$.
	For the other direction, as mentioned in Remark \ref{Rem3.1:OrderToValuation}, it suffices to show that $\mathfrak{A}_\L=\mathfrak{A}_\E$. 
	Note that $\E^\times \subset \E^\times \mathfrak{A}_\E^\times$ is automatic, and $\E^\times \subset \L^\times \mathfrak{A}_\L^\times$ by condition. Then by \cite[Proposition 12.4]{BushnellHenniart:06a}(2)(3), such possible $\mathfrak{A}$ is unique.
\end{proof}

Given a  quadratic field extension $\L$ embedded in the quaternion algebra $\B$ and a minimal character $\theta$ over $\L$, we can now start the construction of a minimal supercuspidal representation.  A general supercuspidal representation can be constructed from a minimal supercuspidal representation by a twist. The following constructions are taken from \cite[Chapter 4, 5, 13]{BushnellHenniart:06a} with slight modifications\footnote{Some of them come from the differences of the meaning of $c(\chi)$, the choice of the additive character $\psi$, and the normalization of $\vv$ when $\epsilon(\B)=-1$ and $e_\L=1$.}.



\begin{defn}\label{Defn:ii'&JHs}
Numerically, let $i=\lfloor\frac{c(\theta)-\epsiBL}{2}\rfloor+\frac{\epsiBL}{2}$, $i'=\lceil \frac{c(\theta)-\epsiBL}{2}\rceil+\frac{\epsiBL}{2}$.
Let $J=\L^\times  K_{\mathfrak{A}}(i)$,
$J^1=U_{\L}(1) K_\mathfrak{A}(i)$, $H=\L^\times  K_\mathfrak{A}(i')$ and $H^1= U_{\L}(1)K_\mathfrak{A}(i')$.

\end{defn}
\begin{rem}
	Note that $i'-i=0$ or $1$. When $i=i'$, we have $H=J$ and $H^1=J^1$. When $i'=i+1$, $J^1/H^1$ is a $2-$dimensional space over the residue field.
\end{rem}
\begin{rem}
 Intrinsically when $c(\pi)>2$, the values of $i, i'$ are determined by that $i'\in \vv((\L^\perp)^\times)$, $i'\geq c(\theta)/2$ is minimal, and  that $i+i'=c(\theta)$ (so that Lemma \ref{lem:mainsec:intertwining0}, \ref{Lem:DuallatticeEperp} are true).
\end{rem}

\begin{example}\label{Example:ListofHJi}
We have the following complete list of datum for minimal supercuspidal representations when $p\neq 2$.
\begin{enumerate}
\item $\epsilon(\B)=1$, $e_\L=1$, $c(\theta)=2n$, ($c(\pi)=4n$,) then $H=J$ and $i=n$.
\item $\epsilon(\B)=1$, $e_\L=1$, $c(\theta)=2n+1$, ($c(\pi)=4n+2$,) then $H\subsetneq J$ and $i=n$.
\item $\epsilon(\B)=1$, $e_\L=2$, $c(\theta)=2n$, ($c(\pi)=2n+1$,) then $H=J$ and $i=n$.
\item $\epsilon(\B)=-1$, $e_\L=1$, $c(\theta)=2n$, then $H\subsetneq J$ and $i=\frac{2n-1}{2}$.
\item $\epsilon(\B)=-1$, $e_\L=1$, $c(\theta)=2n+1$, then $H=J$ and $i=\frac{2n+1}{2}$.
\item $\epsilon(\B)=-1$, $e_\L=2$, $c(\theta)=2n$, then $H= J$ and $i=n$.
\end{enumerate}
\end{example}

Now we move on to discuss the extension of $\theta$ to a character  $\tilde{\theta}$ on $H^1$, and representation $\Lambda$ on $J$. 
We first focus on the case when $c_\L(\theta)\geq 2$, so $\theta $ gives rise to an element $\alpha_\theta\in \L^\times $ as an element of dual Lie algebra as in Lemma \ref{Lem:DualLiealgForChar}, with 
\begin{equation}\label{Eq:valphatheta}
 v_\L(\alpha_\theta)=-c(\theta)+c(\psi_\L).
\end{equation} 
\begin{defn}
Define the character $\tilde{\theta}$ on $H^1$ by
\begin{equation}\label{Eq:charextension}
\tilde{\theta}(l(1+t))=\theta(l)e^{<\alpha_\theta ,t>},
\end{equation}
where $l\in \L^\times$ and $1+t\in K_{\mathfrak{A}}(i')$. 
The pair $(H^1,\tilde{\theta})$ is called a simple character.
\end{defn}
The group $J$ can be alternatively describe by the following property:
\begin{lem}\label{lem:mainsec:intertwining0}
	 $g\in \B^\times $ intertwines $\tilde{\theta}$ on $H^1$ iff $g$ conjugates $\tilde{\theta}$ iff $g\in J$.
\end{lem}

Let $u_\E$ be the cyclic subgroup in $\E^\times$ lifting the nontrivial elements in $k_\E$ by Hensel's lemma, 
and let $u_\F$ be the similar subgroup for $\F^\times$.
The extension of $\tilde{\theta}$ to $\Lambda$ on $J$ satisfies the following properties, by which it can be uniquely determined:
\begin{lem}\label{lem:mainsec:intertwining}
\begin{enumerate}
\item $\dim \Lambda=\sqrt{[H^1:J^1]}$.
\item $\Lambda|_{H^1}$ is a multiple of $\tilde{\theta}$. $\Lambda|_{\F^\times}$ is a multiple of $\theta|_{\F^\times}$.
\item When $\dim\Lambda=1$, $\Lambda|_{\L^\times }=\theta$. When $\dim \Lambda=q$, $\Tr(\Lambda(x))=-\theta(x)$ for any $x\in u_\E-u_\F$; Equivalently, 
$\Lambda|_{\L^\times }=\oplus \theta\mu$ where $c(\mu)=1$ and $\mu|_{\F^\times }=1$. 
\end{enumerate}

\end{lem}

The construction of $\Lambda$ depends on whether $i=i'$. When $i=i'$, \eqref{Eq:charextension} can be extended to the whole subgroup $J$ and $\Lambda =\tilde{\theta}$. 
When $i'=i+1$, consider (any) intermediate group $B^1$ between $J^1$ and $H^1$ such that $B^1/H^1$ is a $1-$dimensional  subspace of $J^1/H^1$, giving a polarization of $J^1/H^1$ under the  pairing (which is alternating by \cite[Section 16.4]{BushnellHenniart:06a}) given by 
\begin{equation}\label{Eq3.1:alternatvepairing}
 \Psi_{\alpha_\theta}(1+x,1+y)\mapsto e^{<\alpha_\theta, [x,y]>}, \forall x,y\in\mathcal{B}^i.
\end{equation} 
Here $[x,y]=xy-yx$ is the Lie bracket.
Because of this one can extend $\tilde{\theta}$ to $B^1$ (not unique), which we still denote by $\tilde{\theta}$ by abuse of notation. Then 
\begin{equation}
 \Lambda|_{J^1}=\Ind_{B^1}^{J_1}\tilde{\theta}.
\end{equation}
This step is called the Heisenberg extension of $(H^1,\tilde{\theta})$. It has the required dimension and is independent of the choices of $(B^1,\tilde{\theta})$. See \cite[Section 16]{BushnellHenniart:06a} for more details.
It remains to identify the action of $\L^\times/U_\L(1)$ on the space $\Ind_{B^1}^{J_1}\tilde{\theta}$ which is consistent with Lemma \ref{lem:mainsec:intertwining} (3). See \cite[Section 19]{BushnellHenniart:06a} for the details of this step.

For uniformity we take $B^1=J^1$  in the case $i=i'$. 
\begin{defn}\label{Def3.1:cupsidaltype}
The triple $(\mathfrak{A}, J, \Lambda)$ is called a cuspidal type. It is said to be associated to $\pi$ if $\pi\simeq c-\Ind_{J}^G\Lambda$. 
\end{defn}
By \cite{BushnellHenniart:06a}, all cuspidal types associated to $\pi$ are conjugate to each other.
\begin{defn}\label{Defn:GeneralMinimalVec}
 An element $\varphi\in\pi$ is called a minimal vector if there
exists a cuspidal type $(\mathfrak{A}, J, \Lambda)$ associated
to $\pi$, such that $\varphi$ is an eigenvector for the  simple
character
$(H^1, \tilde{\theta})$ associated to $(\mathfrak{A}, J, \Lambda)$.
\end{defn}

When $i=i'$, the minimal vector is actually unique up to a constant. When $i'=i+1$, the space of eigenvectors for the simple character is $q-$dimensional.
To see these, just apply Mackey theory for the compact induction and use Lemma  \ref{lem:mainsec:intertwining0}.

For applications, it is sometimes convenient to specify a particular basis when $\dim \Lambda>1$.
\begin{defn}\label{Defn:generalMinimalVec}
A minimal vector $\varphi$ can be uniquely identified by that
\begin{enumerate}
\item[Type 1] either $B^1$ acts on it by $\tilde{\theta}$ for some intermediate polarizing subgroup $B^1$,
\item[Type 2] or $H$ acts on it by $\tilde{\theta}\mu$ for some character $\mu$ on $\L^\times$ with $c(\mu)=1$ and $\mu|_{\F^\times }=1$.
\end{enumerate}
Here $\tilde{\theta}\mu$ is the extension of $\theta\mu$ to $H$ similarly as in \eqref{Eq:charextension}.
Then $\{\pi(g)\varphi\}$ for type 1 minimal vector $\varphi$ and $g\in J^1/B^1$, and $\{\varphi\}$ for type 2 minimal vectors with all possible $\mu$ provide two orthogonal basis for the space of minimal vectors associated to a particular cuspidal type $(\mathfrak{A}, J, \Lambda)$.
\end{defn}
Note that these two types coincide when $\dim\Lambda=1$.


\begin{lem}\label{Cor:MCofGeneralMinimalVec}
For simplicity let $\pi$ be a unitary irreducible and supercuspidal representation, and let $\Phi_\varphi$ be the matrix coefficient associated to a minimal vector $\varphi$. Then $\Phi_\varphi$ is supported on $J$.   When $\varphi$ is type 1 and $b\in B^1$,
\begin{equation}
\Phi_\varphi(bx)=\Phi_\varphi(xb)=\tilde{\theta}(b)\Phi_\varphi(x).
\end{equation}
Furthermore $\Phi_\varphi|_{J^1}$ is supported on $B^1$. 
When $\varphi$ is a type 2 minimal vector associated to $\tilde{\theta}\mu$, $h\in H$,
\begin{equation}
\Phi_\varphi(hx)=\Phi_\varphi(xh)=\tilde{\theta}\mu(h)\Phi_\varphi(x).
\end{equation}
\end{lem}

Now consider the case when $c_\L(\theta)=1$, which only occurs for $\L$ inert and $c(\pi)=2$.  

In this case $\tilde{\theta}$ is defined to be the trivial character on $H^1$. When $\B$ is the division algebra, $H=J$ and we can take $\Lambda$ to be the extension of $\tilde{\theta}$ to $J$ similarly. 

When $\B$ is the matrix algebra, we fixed the embedding of $\L$ as in \eqref{Eq:StandardEmbeddingL} with $v(D)=0$. $\Lambda$ is a $q-1$ dimensional representation on $\L^\times\GL_2(O_\F)$, inflated from a cuspidal representation $\lambda$ of $\GL_2(k_\F)$ over the residue field (see \cite[Section 6]{BushnellHenniart:06a} for more details on $\lambda$). It has the property that $\Lambda|_{\L^\times }=\oplus_{c(\theta')=1} \theta'$ where $\theta'|_{\F^\times }=\theta|_{\F^\times }$ and $\theta'\neq \theta$ or $\overline{\theta}$. Here $\overline{\theta}(x)=\theta(\overline{x})$. 
The minimal vectors in this case 
are $q-1$ dimensional and $H^1$ invariant. One can get type 2 minimal vectors by requiring that $H$ acts by the character $\theta'$ as above.  

On the other hand, $\lambda$ contains one (and thus all $q-1$ of them) nontrivial character $\psi_N$ of a unipotent subgroup $N(k_\F)$. Let $B^1$ be the preimage of $N(k_\F)$ under the map $\GL_2(O_\F)\rightarrow \GL_2(k_\F)$, and let $\tilde{\theta}$ be the pull back character of $\psi_N$ on $B^1$. One can then obtain type 1 minimal vectors by requiring that $B^1$ acts by $\tilde{\theta}$. 

To sum up, the type 1 and type 2 minimal vectors are also well-defined when $c(\pi)=2$.



\subsection{Some results for lattices}

\begin{lem}\label{lem7.1:conjugationEperp}
Let $ \alpha\in \L^\times-\F^\times$, $t\in \mathcal{B}^n$ for $n>0$. Then 
\begin{equation}\label{Eq3.2:conjugation}
(1+t)^{-1}\alpha(1+t)\equiv \alpha+[\alpha,t]\mod\mathcal{B}^{2n+v_\L(\alpha)},
\end{equation}
whereas the map $t\mapsto [\alpha,t]$ induces an isomorphism  $\B /\L\rightarrow \L^\perp$. Further more if $\alpha$ is \imaginary, and $t\in \L^\perp$, then
\begin{equation}
\vv([\alpha,t])=v_\L(\alpha)+\vv(t).
\end{equation}
\end{lem}
\begin{proof}
	\eqref{Eq3.2:conjugation} follows from the Taylor expansion for $(1+t)^{-1}$ and $n>0$. 
$t\mapsto [\alpha,t]$ induces an isomorphism  $\B /\L\rightarrow \L^\perp$ as $\dim (\B/\L)=\dim \L^\perp$, while $$\Tr(l [\alpha,t])=\Tr(l\alpha t)-\Tr(lt\alpha)=\Tr(\alpha l t)-\Tr(t\alpha l)=0.$$

Now let
$t=lj$ by Lemma \ref{Lem2.1:Eperpbasic}, we have
\begin{equation}
[\alpha,t]=\alpha lj-lj\alpha=(\alpha-\overline{\alpha})lj.
\end{equation}
As $\alpha$ is \imaginary, we have
\begin{equation}
\vv(\alpha_0)=\vv(\alpha-\overline{\alpha})=\vv(\alpha),
\end{equation}
and the claim follows immediately from Lemma \ref{Lem3.1semivaluation}.
\end{proof}

\begin{defn}
For a minimal element $\alpha_\theta$, define a non-degenerate unitary pairing 
\begin{align}
\Psi_{\alpha_\theta}:\L^\perp\times \L^\perp &\rightarrow \C\\
(t,x)&\mapsto e^{<[\alpha_\theta,t], x>}=e^{<\alpha_\theta,[t,x]>}.\notag
\end{align}
For any $O_\F-$lattice $A$ in $\L^\perp$, define $A^{\ddager}=\{x\in \L^\perp, \Psi_{\alpha_\theta}(x,y)=1 \forall y\in A\}$.
For any element in $ \L^\times  K_{\mathfrak{A}}(i)$, one can uniquely write it in the form $e(1+t)$ for $e\in \L$ and $t\in \L^\perp$ with $\vv(t)\geq i$. Define the map
\begin{align}
\pr: &\L^\times \backslash \L^\times  K_{\mathfrak{A}}(1) \rightarrow \L^\perp\\
&e(1+t)\mapsto t.\notag
\end{align}
\end{defn}

\begin{lem}\label{Lem:DuallatticeEperp}
$\pr(J)=\pr (J^1)$, $\pr (H^1)$, $\pr (B^1)$ are $O_\F-$lattices in $\L^\perp$, and $\pr (J^1)^{\ddager}=\pr(H^1)$, $\pr (B^1)^{\ddager}=\pr(B^1)$.
\end{lem}
\begin{proof}
The first part is direct. 
To prove $\pr (J^1)^{\ddager}=\pr(H^1)$, note first that by definition $\pr(J^1)=\{t\in \L^\perp| \vv(t)\geq i\}$, $\pr(H^1)=\{x\in \L^\perp| \vv(x)\geq i'\}$. For $t\in \pr(J^1)$, we have $[\alpha_\theta,t]\in \L^\perp$ with $\vv([\alpha_\theta,t])=v_\L(\alpha_\theta)+\vv(t)$ by Lemma \ref{lem7.1:conjugationEperp}. By Lemma \ref{Lem3.1semivaluation}(2), $[\alpha_\theta,t] x\in \L$ with  $v_\L([\alpha_\theta,t] x)=v_\L(\alpha_\theta)+\vv(t)+\vv(x)$. 
Note that
$e^{<[\alpha_\theta,t], x>}=\psi_\L([\alpha_\theta,t] x)$, with $c(\psi_\L)=-e_\L+1$ by \eqref{Eq:CpsiL}. Thus $\Psi_{\alpha_\theta}(t,x)=1$ for any $t\in \pr(J^1)$ iff $\vv(x)+i+v_\L(\alpha_\theta)\geq -e_\L+1$. 
This in turn is equivalent to that $\vv(x)\geq i'$ by Definition \ref{Defn:ii'&JHs} and \eqref{Eq:valphatheta}, concluding the proof for $\pr (J^1)^{\ddager}=\pr(H^1)$.

The relation $\pr (B^1)^{\ddager}=\pr(B^1)$ can be proven similarly and is directly related to that $B^1/H^1$ is a polarization of $J^1/H^1$ with the given pairing. 
\end{proof}

\begin{rem}
The relation $\pr (J^1)^{\ddager}=\pr(H^1)$ is basically a reformulation of Lemma \ref{lem:mainsec:intertwining0}. 
Indeed for $g=1+t\in J$ with $t\in \L^\perp$, $g$ conjugating $\tilde{\theta}$ on $H^1$ is equivalent, by Lemma \ref{lem7.1:conjugationEperp} while ignoring higher order terms, to that
\begin{equation}
\psi(\Tr([\alpha_\theta, t]x))=1, \forall x\in \pr(H^1).
\end{equation}

\end{rem}

\subsection{Lie algebra, linearisation  and Matrix coefficients in proper neighbourhoods}\label{Sec3.3:Liealg}

To make use of the Lie algebra language, we assume that $p$ is large enough and $J=\L^\times K_\mathfrak{A}(i)$ for $i>0$ throughout this section.
For $F\in C^\infty_c(G)$, $\varphi\in \pi$, let
\begin{equation}
\pi(F) \varphi=\int\limits_{x\in G}F(x)\pi(x)\varphi dx.
\end{equation}

\begin{lem}\label{Lem7.1:FrobCharacter}
Let $\pi=c-\Ind_J^G \Lambda$ for some compact (mod center) open subgroup $J$ and irreducible finite dimensional representation $\Lambda$. Then
\begin{align}
\Tr(\pi(F))&=\int_{g\in J\backslash G}\int_{j\in J}F(g^{-1}jg)\Tr(\Lambda(j)) djdg\\
&=\int_{g\in J\backslash G}\int_{j\in G}F(j)\Tr(\Lambda(gjg^{-1}))\Char_{J}(gjg^{-1}) djdg.\notag
\end{align}
\end{lem}
\begin{proof}
For the first equality, we view $\varphi\in \pi$ as a vector-valued function in the compact induction model satisfying that $\varphi(jg)=\Lambda(j)\varphi(g)$.
Then
\begin{equation}
\pi(F) \varphi(h)=\int_{g\in G}F(g)\varphi(hg)dg=\int_{g\in G}F(h^{-1}g)\varphi(g)dg=\int_{g\in J\backslash G}\int_{j\in J}F(h^{-1}jg)\Lambda(j)\varphi(g)djdg.
\end{equation}
$K_F(h,g)=\int_{j\in J}F(h^{-1}jg)\Lambda(j)dj$ is then the kernel of the integral operator and
\begin{equation}
\Tr(\pi(F))=\int_{g\in J\backslash G} \Tr(K_F(g,g))=\int_{g\in J\backslash G}\int_{j\in J}F(g^{-1}jg)\Tr(\Lambda(j)) djdg.
\end{equation}
For the second equality, we have
\begin{align}
 \int_{g\in J\backslash G}\int_{j\in J}F(g^{-1}jg)\Tr(\Lambda(j)) djdg= \int_{g\in J\backslash G}\int_{j\in G}F(g^{-1}jg)\Tr(\Lambda(j))\Char_{J}(j) djdg.
\end{align}
Then one can get the required equality by a change of variable.
\end{proof}
Now we introduce the Lie algebra language. See \cite{Howe:aa} for more details.
\begin{defn}\label{Defn:LiealgRange}
\begin{enumerate}
\item Let $\g=\B$ be the Lie algebra for $\B^\times$. Let $\exp$ denote the exponential map from $\g$ to $\B^\times$.
Let $\g_0$ be an $O_\F-$ lattice of $\g$ on which the Lie bracket is closed, and $\exp$ is convergent (with the usual  $\log$ being the inverse), inducing a group homomorphism where the group structure on $\g_0$ is given by the Campbell-Hausdorff 
formula (which also need to be convergent)
\begin{equation}\label{Eq3.3:CHlaw}
Z(x,y)=x+y+\frac{1}{2}[x,y]+O(x^2y,xy^2).
\end{equation}
Here the remainder terms $O(x^2y,xy^2)$ involve at least two commutators of $x,y$.
For simplicity we shall assume in the remaining of this section that $p$ is large enough. Then we can take $$\g_0=\{x\in \B|\vv(x)>0\}.$$ In general we can take $\g_0=\{x\in \B|\vv(x)\geq c_0\}$ for some absolutely bounded integer $c_0$.

\item  
Define $j_0=\{x\in \g_0, \exp(x)\in J\}$, $h_0=\{x\in \g_0,\exp(x)\in H\}$, $b _0=\{x\in \g_0, \exp(x)\in B^1\}$.
\item Define the set $\g_+=\{x\in \B| v_{\B,\L'}(x)>0 \text{\ for some embedded }\L'\}$. Define $\h_0=\g_+\cap \E$,  $\mathfrak{z}_0=\g_+\cap Z$. 

\item For $f\in C_c^\infty(\g_+)$, 
$\varphi\in \pi$, define
\begin{equation}
\pi(f) \varphi=\int\limits_{\g} f(x)\pi(\exp(x))\varphi dx.
\end{equation}

\item Let $\alpha_\theta$ be the element associated to $\theta$ by Lemma \ref{Lem:DualLiealgForChar}(2), considered as an element of $\B$ under the embedding $\L\hookrightarrow \B$.
 The $G-$coadjoint orbit of $\alpha_\theta$ is defined to be $O_\pi=\{g^{-1}\alpha_\theta g\}_{g\in G}$, or equivalently the set of all possible embeddings of $\alpha_\theta$. We normalize the invariant measure on $O_\pi$ so that 
\begin{equation}
\Vol(\{g^{-1}\alpha_\theta g, g\in G_{\alpha_\theta}\backslash J\})=\dim \Lambda.
\end{equation}
Here $G_{\alpha_\theta}\simeq \L^\times$ is the stabilizer of $\alpha_\theta$ in $G$ under the conjugation.
\end{enumerate}

\end{defn}
\begin{rem}\label{rem3.3:CHj0}
 $j_0$ can be alternatively described as
		$$j_0=\log (J^1)=\{x=x_\L+x^\perp, v_\L(x_\L)>0, \vv(x^\perp)\geq i\}.$$
		The $\L^\perp$ part of $j_0$ is the same as $\pr(J)$. Similar statements hold for $b_0$ and $h_0$.

\end{rem}
\begin{lem}\label{Lem3.3:g+h0}
	When $p$ is large enough, we have $\g_+=\{x\in \B| v(\Tr(x)), v(\Nm(x))>0\}$. 
	As a result, $\z_0=\varpi_\F O_\F$, and
	$$
	\h_0=\begin{cases}
	\varpi_\E O_\E, &\text{\ if $\E$ is a field},\\
	\iota^{-1}(\varpi_\F O_\F\times \varpi_\F O_\F), &\text{\ if }\iota(\E)=\F\times \F.
	\end{cases} $$
\end{lem}
\begin{proof}
	Denote $\g_+'=\{x\in \B| v(\Tr(x)), v(\Nm(x))>0\}$.
We shall only prove the statement that $\g_+=\g_+'$, as the rest statements are easy to check. 
In the division algebra case, all semi-valuations $\vv$ agrees with $v_\B$ up to a constant, and it is clear that $\g_+=\g_+'=\{x|v_\B(x)>0\}$. So we shall focus on the matrix algebra case. Note that both sets are invariant under the conjugation. The inclusion $\g_+\subset \g_+'$ follows from Lemma \ref{Lem3.1semivaluation}. On the other hand, if $x\in \g_+'$, $x$ is either a scalar or conjugate to $\zxz{\Tr(x)}{1}{-\Nm(x)}{0}$ according to Lemma \ref{Lem:ConjugacyRep}. In both cases after proper conjugation $x$ belongs to the particular $\mathcal{B}$ for the $e_\L=2$ case in Example \ref{Example:explicitideal}.

\end{proof}
Recall that 
 a type 1 minimal vector $\varphi$ can be uniquely determined by that $B^1$ acts on $\varphi$ by $\tilde{\theta}$. 
 On the other hand, any $\alpha\in O_\pi$ gives rise to an embedding $\L\rightarrow \F[\alpha]$, and thus to $\vv$, a cuspidal type $(\mathcal{U},J,\Lambda)$ associated to $\pi$, as well as the lattices $\g_0$, $j_0$, $h_0$.

\begin{cor}\label{Cor:type1MinVecAlpha}
For any fixed $\alpha\in O_\pi$ and a choice of intermediate lattice $h_0\subset b_0\subset j_0$, the function 
$x\mapsto e^{<\alpha,x>}$ defines a character on $b_0$ with respect to the group law in \eqref{Eq3.3:CHlaw}.
Then a type 1 minimal vector $\varphi$ is uniquely identified by that for any $x\in b _0$,
\begin{equation}
\pi(\exp(x))\varphi=e^{<\alpha,x>}\varphi.
\end{equation}
\end{cor}
\begin{proof}
To prove that $x\mapsto e^{<\alpha,x>}$ is a character on $b_0$,   it suffices to show  that $e^{<\alpha,\cdot>}=1$ on all commutator terms in \eqref{Eq3.3:CHlaw}.
For this, we first show that for any fixed $y$, 
	$e^{<\alpha,[x,y]>}=e^{<\alpha, [x^\perp, y]>}.$ 
Indeed we have $e^{<\alpha, x_\L y>}=\psi\circ\Tr(\alpha x_\L y)=\psi\circ\Tr(x_\L\alpha  y)=\psi\circ\Tr(\alpha  yx_\L)=e^{<\alpha,y x_\L>}$ as $\alpha, x_\L \in \L$. 
Similar argument holds for $y$.
Then $e^{<\alpha,[x,y]>}=e^{<\alpha,[x^\perp,y^\perp]>}=1$ follows from Lemma \ref{Lem:DuallatticeEperp}. 

Higher commutator terms can be dealt with using similar argument together with Jacobi identity. For example,
\begin{align}
e^{<\alpha,[x,[y,z]]>}=e^{<\alpha,[x^\perp,[y,z]]>}=e^{<\alpha,-[y,[z,x^\perp]]-[z,[x^\perp,y]]>}=
e^{<\alpha,-[y^\perp,[z,x^\perp]]-[z^\perp,[x^\perp,y]]>}.
\end{align}
It is then trivial according to Lemma \ref{Cor3.1:psitrace=1} since, for example, 
$\vv(\alpha y^\perp z x^\perp)> 2i-c(\theta)+c(\psi)\geq c(\psi)-1$.

	From this, we get that the push forward of $e^{<\alpha,x>}$ gives rise to a character on $B^1$. It is automatically consistent with $\tilde{\theta}$ on $H^1$ by \eqref{Eq:charextension} and Lemma \ref{Lem:DualLiealgForChar}(2). The lemma then follows from the discussion in Section \ref{Sec:CompactInd}.
\end{proof}

We shall now develop the analogue of the Kirillov formula representing the trace of integral operators as integrals along coadjoint orbit in a suitable neighborhood.
\begin{lem}\label{Lem3.3:vanishingorbitint} For any $x\in \g_+$, 
	 $x\notin j_0$, we have
	\begin{equation}
	\int_{g_0\in G_{\alpha_\theta}\backslash J}e^{<g_0^{-1}\alpha_\theta g_0, x>} dg_0=0.
	\end{equation}
	\begin{proof}
	We  write $x=x_\L+x^\perp$ with $x_\L\in \L$, $x^\perp \in  \L^\perp$.
	We claim that $\vv(x^\perp)<i$ if $x\notin j_0$ and $v(\Nm(x))>0$. This is because either $\vv(x_\L)\geq 1$, so $\vv(x^\perp)\geq i$ directly implies that $x\in j_0$, or $\vv(x_\L)\leq 0$, in which case $1\leq v(\Nm(x))=v(\Nm(x_\L)+\Nm(x^\perp))$ implies that $\vv(x^\perp)=\vv(x_\L)\leq 0$.
		
		By a change of variable, we have that for any $t\in \pr(J)$,
		\begin{equation}
		\int_{g_0\in G_{\alpha_\theta}\backslash J}e^{<g_0^{-1}\alpha_\theta g_0, x>} dg_0=\int_{g_0\in G_{\alpha_\theta}\backslash J}e^{<(1+t)^{-1}g_0^{-1}\alpha_\theta g_0(1+t), x>}dg_0
		\end{equation}
		By Lemma \ref{lem7.1:conjugationEperp}, \ref{Lem:DuallatticeEperp}, we can choose proper $t\in \pr(J)$ with largest possible $\vv(t)$ such that $e^{<[\alpha_\theta,t], x^\perp>}\neq 1$.
		In the current case we can choose such $t$ with 
		\begin{equation}\label{Eq3.3:vvt}
		\vv(t)=c(\theta)-\vv(x^\perp)-1\geq i'\geq i.
		\end{equation}

		We claim that for such $t$, \begin{equation}
		e^{<(1+t)^{-1}g_0^{-1}\alpha_\theta g_0(1+t), x>}=e^{<g_0^{-1}\alpha_\theta g_0, x>}e^{<[\alpha_\theta,t], x^\perp>}.
		\end{equation}
		 From this claim we have
		\begin{equation}
		\int_{g_0\in G_{\alpha_\theta}\backslash J}e^{<g_0^{-1}\alpha_\theta g_0, x>} dg_0=e^{<[\alpha_\theta,t], x^\perp>}\int_{g_0\in G_{\alpha_\theta}\backslash J}e^{<g_0^{-1}\alpha_\theta g_0, x>} dg_0.
		\end{equation}
		Thus the integral has to be zero.
		
		To prove the claim, we take the coset representative $g_0=1+t_0$ for $t_0\in \pr(J)$, and assume WLOG that $\vv(t_0)\leq \vv(t)$. Then as in the proof of Lemma \ref{lem7.1:conjugationEperp}, we get via Taylor expansion that
		\begin{equation}
		(1+t)^{-1}g_0^{-1}\alpha_\theta g_0(1+t)=g_0^{-1}\alpha_\theta g_0+[\alpha_\theta, t]+O(t t_0\alpha_\theta).
		\end{equation}
		Note that $e^{<[\alpha_\theta, t],x>}=e^{<[\alpha_\theta, t],x^\perp>}$ as $[\alpha_\theta, t]\in \L^\perp$.
		The remainder terms (and the order of the product for each of them) do not matter, because 
		we have either $\vv(x)>0$, in which case (regardless of the order of the product) 
		\begin{equation}\label{Eq3.3:vvRemainder}
		\vv(t t_0\alpha_\theta x)>2i-c(\theta)+c(\psi_\L)\geq c(\psi_\L)-1
		\end{equation} so that $e^{<O(t t_0\alpha_\theta), x>}=1$ by Lemma \ref{Cor3.1:psitrace=1}; 
		Or $\vv(x)\leq 0$, in which case the condition $v(\Nm(x))>0$ forces  $\vv(x_\L)=\vv(x^\perp)= \vv(x)$, and thus by \eqref{Eq3.3:vvt} \begin{equation}
		\vv(t t_0\alpha_\theta x)\geq \vv(t_0\alpha_\theta)+c(\theta)-\vv(x^\perp)-1+\vv(x^\perp)\geq i-1+c(\psi_\L)>c(\psi_\L)-1,
		\end{equation}
		so again $e^{<O(t t_0\alpha_\theta), x>}=1$.
	\end{proof}

\end{lem}
\begin{lem}\label{Lem3.3:Kirillovforcompact}
	For any $x\in j_0$,  
	\begin{equation}
	\int_{g_0\in G_{\alpha_\theta}\backslash J}e^{<g_0^{-1}\alpha_\theta g_0, x>} dg_0=\Tr(\Lambda(\exp(x))).
	\end{equation}
\end{lem}
Note that this is essentially the explicit Kirillov formula for compact groups. See \cite{Howe:aa}.
\begin{proof}
	
	The normalisation in Definition \ref{Defn:LiealgRange}(4) guarantees the equality when $x=0$ here. For $x\in j_0$ such that $\exp(x)\in H^1$, both sides are multiples of $\tilde{\theta}$ by Lemma \ref{lem:mainsec:intertwining0}, \ref{lem:mainsec:intertwining} and Corollary \ref{Cor:type1MinVecAlpha}. 
	
	Now we show that both sides are vanishing when $\dim\Lambda=q$ and $\exp(x)\in J^1-H^1$. 
	For the left hand side, we use the same arguments as in the proof for Lemma \ref{Lem3.3:vanishingorbitint}. Note that in this case $\vv(x)=i$, and we can still pick $t\in \pr(J)$ such that $e^{<[\alpha_\theta,t],x^\perp>}\neq 1$ and $\vv(t)=c(\theta)-\vv(x^\perp)-1\geq i$. The remaining arguments are the same, as \eqref{Eq3.3:vvRemainder} requires only $\vv(t)\geq i$.
	
 For the right hand side, recall that $\Lambda|_{J^1}=\Ind _{B^1}^{J^1}\tilde{\theta}$, and is actually independent of the choices of $(B^1, \tilde{\theta})$. 
 In particular we choose $B^1$ such that $\exp(x)\in J^1-B^1$, which is possible as any 1-dimensional subspace of $J^1/H^1$ gives a polarization of the alternate pairing \eqref{Eq3.1:alternatvepairing}.
 Pick a basis for $\Lambda|_{J^1}$ to be functions supported only on a single $B^1-$coset. We claim that the permutation of cosets induced by $\exp(x)$ does not fix any one of them, then it is clear that $\Lambda(\exp(x))$ has trace 0. 
 
 To prove the claim, suppose that $g=\exp(x)$ satisfies $B^1 g_i g=B^1 g_i$ for some coset representative $g_i$. Then $g\in g_i^{-1}B^1 g_i$. As $g=\exp(x)\notin B^1$, it suffices to show that actually any $g_i\in J^1$ normalizes $B^1$. If we write $h=(1+l)(1+z)\in B^1$ for $1+l\in U_\L(1)$ and $y\in \Pr(B^1)$ and $1+l'\in U_\L(1)$, 
\[ (1+l')^{-1}h(1+l')=(1+l)(1+l')^{-1}(1+z)(1+l')=(1+l)(1+z+O(l'z))=h(1+O(l'z))\]
by Taylor expansion, with $1+O(l'z)\in H^1$ as $\vv(l'z)\geq i+1$. Similarly for $1+y\in J^1$ with $y\in \pr(J^1)$, we have
 \[(1+y)^{-1} h (1+y)=h(1+O(ly+zy))\]
 with $1+O(ly+zy)\in H^1$ as $\vv(O(ly+zy))>i$.
   As any $g_i\in J^1$ is in the shape of $(1+y)(1+l')$, so indeed $g_i^{-1}hg_i\in B^1$.
 \yhn{Does this sound right? Also any reference is welcome to replace this argument.}

\end{proof}
%
Combining the previous two lemmas, we get the following result:
\begin{cor}\label{lem7.1:char=int}
	For any $x\in \g_+$, 
	we have
	\begin{equation}
	\Tr(\Lambda (\exp(x)))\Char_{J}(\exp(x))=\int_{g_0\in G_{\alpha_\theta}\backslash J}e^{<g_0^{-1}\alpha_\theta g_0,x>}dg_0.
	\end{equation}
\end{cor}

\begin{prop}\label{Prop:Kirillovformula}
For notations as above, we have
\begin{equation}
\Tr(\pi(f))=\int_{\xi\in O_\pi}\int_{x\in\g_0}f(x)e^{<\xi,x>} dx d\xi=\int_{\xi\in O_\pi}\check{f}(\xi) d\xi,
\end{equation}
where $\check{f}(\xi)=\int_{\g}f(x)e^{<\xi,x>}dx$. 
\end{prop}
\begin{proof}\yh{Need to say some thing about Jacobian for exponential map}
Define $$F(g)=\begin{cases}
f(x), &\text{\ if $g=\exp(x)$ for }x\in \Supp f\\
0, &\text{\ otherwise.}
\end{cases}$$
Then by Lemma \ref{Lem7.1:FrobCharacter}, 
\begin{equation}
\Tr(\pi(f))=\Tr(\pi(F))=\int_{g\in J\backslash G}\int_{x\in\g} f(x)\Tr(\Lambda (g\exp(x)g^{-1}))\Char_{J}(g\exp(x) g^{-1})dxdg.
\end{equation}

On the other hand, let $G_{\alpha_\theta}\simeq \L^\times $ be the $G-$stabiliser of $\alpha_\theta$. Then
\begin{align}
\int_{\xi\in O_\pi}\check{f}(\xi) d\xi=\int_{\xi\in O_\pi}\int_{x\in\g_0}f(x)e^{<\xi,x>} dx d\xi&=\int_{G_{\alpha_\theta}\backslash G}\int_{x\in\g}f(x)e^{<g^{-1}\alpha_\theta g,x>} dx dg \\
&=\int_{g_0\in G_{\alpha_\theta}\backslash J}\int_{g\in J\backslash G}\int_{x\in\g}f(x)e^{<g_0^{-1}\alpha_\theta g_0,gxg^{-1}>} dx dg dg_0 \notag
\end{align}
Then the conclusion  follows from Corollary \ref{lem7.1:char=int}.
\end{proof}

For local period integrals, it is easier to use matrix coefficient directly rather than the trace, so we also give the following lemma:
\begin{lem}\label{lem:char=intforB1}
	For any $x\in \g_+$, we have
	\begin{equation}
	\int\limits_{g_0\in U_\L(1)\backslash B^1}e^{<g_0^{-1}\alpha_\theta g_0,x>}dg_0=e^{<\alpha_\theta,x>}\Char_{B^1}(\exp(x)).
	\end{equation}
	
\end{lem}
\begin{proof}
One can argue similarly as in the proofs for Lemma \ref{Lem3.3:vanishingorbitint} and \ref{Lem3.3:Kirillovforcompact}, and make use of the fact that	 $\pr(B^1)^{\ddager}=\pr(B^1)$ by Lemma \ref{Lem:DuallatticeEperp}.

\end{proof}

\section{Local period integral}\label{Sec:mainmethod}
After proper twisting for both $\pi$ and $\chi$, we can assume that $\pi$ is minimal. Furthermore when $\pi$ is associated to a character $\theta$ over $\L$ as in the last section with $c(\pi)\geq 3$ and $p>2$,  the associated element in the dual Lie algebra $\alpha_\theta$ is \imaginary \text{\ }as in Definition \ref{Defn:imaginary}.

In this section we are concerned about the local period integral
\begin{equation}
I(\varphi,\chi)=\int\limits_{\F^\times\backslash\E^\times}\Phi_{\varphi}(t)\chi^{-1}(t)dt,
\end{equation}
where $\varphi$ is a minimal vector associated to some cuspidal type $(\mathfrak{A},J,\Lambda)$. From Lemma \ref{Cor:MCofGeneralMinimalVec}, we know that $\Phi_{\varphi}(t)$ is supported on $J$.
We first explain the possibility for the support of the local integral $J\cap \E^\times $.
\begin{lem}\label{Lem:JcapEwholeorwhat}
Suppose that $c(\pi)\geq 3$, \yh{Depth 0 need adjustment} and $p$ is large enough so that we can take $\g_0=\{\vv(x)>0\}$. Then $J\cap \E^\times $ is either the whole $\E^\times $ or $J\cap \E^\times \subset Z\exp(j_0) \cap Z\exp (\h_0)$.
\end{lem}
\begin{proof}
	Note that $Z\exp(j_0)=ZJ^1=ZU_\L(1)K_\mathfrak{A}(i)$, and
	$Z\exp(\h_0)=ZU_\E(1)$ by Lemma \ref{Lem3.3:g+h0} and the convention on $U_\E(n)$ in Section \ref{Basic notations} when $\E$ splits.
	Suppose that $J\cap \E^\times \not\subset Z\exp(j_0) \cap Z\exp (\h_0)$, and let $g=l(1+t)=e \in J\cap \E^\times- ZJ^1\cap Z\exp(\h_0)$ for $l\in \L$, $t\in \mathcal{B}^i\cap \L^\perp$, $e\in\E$. This implies that either $l\notin ZU_\L(1)$, or $e\notin ZU_\E(1)$.
	
	We first discuss the characteristic polynomial of $g=l(1+t)$.
If $\L$ is unramified, then after multiplying with an element in $Z$, we can assume that $l\in O_\L^\times$. Then $g=l(1+t)$ has associated characteristic polynomial
\begin{equation}\label{Eq4.1:charpoly}
f(\lambda)=\Nm(g-\lambda)=\Nm(l+lt-\lambda)=\Nm(l-\lambda)+\Nm(lt)\equiv \lambda^2-\Tr (l)\lambda+\Nm(l)\mod{\varpi_\F}.
\end{equation}
Here we have used that $v_\F(\Nm(lt))=\frac{2}{e_\L}\vv(lt)>0$ by Lemma \ref{Lem3.1semivaluation}(2).
When $l\in O_\F^\times U_\L(1)$, $f(\lambda)$ reduces to a polynomial with double roots when $\mod\varpi$. When $l\notin O_\F^\times U_\L(1)$, $l$ is actually a minimal element.
In this case $f(\lambda)$ reduces to an irreducible quadratic polynomial over the residue field, giving rise to an inert quadratic extension.

When $\L$ is ramified, note first that $O_\L^\times= O_\F^\times U_\L(1)$.
Then after multiplying with an element in $Z$ we have either $l\in O_\F^\times U_\L(1)$ in which case the associated characteristic polynomial has double roots when $\mod{\varpi}$ (using a similar computation as in \eqref{Eq4.1:charpoly}); or $l\in \varpi_\L \F^\times O_\L^\times$, in which case $l$ is minimal, 
and the associated characteristic polynomial
always gives rise to a ramified quadratic extension.

The discussion of characteristic polynomial for $g=e$ is similar and easier (as there is no $t$ part). Note that when $\E$ splits,  $e$ (after a proper multiple) gives rise to a reducible characteristic polynomial with distinct roots over the residue field.

The discussion so far implies that if $\E\not\simeq \L$, such $g\in J\cap \E^\times- ZJ^1\cap Z\exp(\h_0)$ does not exist.
On the other hand if $\E\simeq \L$, then 
$l\notin ZU_\L(1)$ implies $e\notin ZU_\E(1)$, and vice versa.

We suppose from now on that $\E\simeq \L$ and $g=e=l(1+t)$ for  $e\notin ZU_\E(1)$ and $l\notin ZU_\L(1)$. As discussed above, $l$ is minimal in this case. On one hand, $\E^\times=\{a+be|ab\neq 0\}$. On the other hand,
\begin{equation}
a+be=a+b(l(1+t))=(a+bl)(1+\frac{blt}{a+bl})
\end{equation}
with $\vv(a+bl)\leq \vv(bl) $ according to Corollary \ref{Sec2.1Cor:vofimaginary} as $l$ is minimal. Thus $\vv(\frac{blt}{a+bl})\geq i$ and $a+be\in J$ as well. As a result we must have $\E^\times \subset J\cap\E^\times $.

\end{proof}

From this result we see that either $J\cap \E^\times = \E^\times $, or the nonzero contribution of the local integral comes from a range where we can use Lie algebra interpretation for both $\Phi_{\varphi}$ and $\chi$. From now on, we fix the embedding of $\E$, and assume that the local integral $I(\varphi,\chi)\neq 0$ for some minimal vector $\varphi$. We divide the discussions according to whether $J\cap \E^\times = \E^\times $ or not for this minimal vector.

\subsection{The case $J\cap \E^\times = \E^\times $}\label{Sec:wholetoruscase}
Since $\L\simeq \E$ in this case, there exists $g\in \L^\times\backslash \G$ such that $\E=g^{-1}\L g$. We shall use type 2 minimal vectors as in Definition \ref{Defn:generalMinimalVec}.
\begin{lem}\label{Lem4.1:JcapEwhole}
For $\E=g^{-1}\L g$ with $g\in \L^\times\backslash \G$, we have $J\cap \E^\times = \E^\times $ if and only if $g\in \epsilon J$, where $\epsilon =1$ or $j$, with $j$ as in Lemma \ref{Lem2.1:Eperpbasic}.
\end{lem}
\begin{proof}
	Note that  $j$ normalizes $\L^\times$ (sending $l\in \L$ to $\overline{l}$), while $\vv(j^{-1}xj)=\vv(x)$ for any $x\in G$ according to Definition \ref{Defn3.1:semivaluation}. Then the direction $g\in \epsilon J$ implying $J\cap \E^\times = \E^\times$ is clear. 
	
Now we  write $g=a+bj$ for $a,b\in \L$ according to Lemma \ref{Lem2.1:Eperpbasic}, and assume that $\vv(bj)\geq \vv(a)$ by multiplying $j$ on the left if necessary. It then suffices to show that $g\in J$ for the other direction of the lemma. For any $l\in \L^\times$, we have
\begin{equation}\label{Eq4.1:conjugationonl}
g^{-1} l g=\frac{1}{\Nm(g)}(\overline{a}-bj)l (a+bj)=\frac{1}{\Nm(g)} (\overline{a}al-\overline{b}bj^2 \overline{l}+ \overline{a}b (l-\overline{l})j).
\end{equation}
Then $g^{-1}l g\in J$ implies that 
\begin{equation}\label{Eq4.1:vvcondition}
\vv( \overline{a}b (l-\overline{l})j)\geq \vv(\overline{a}al-\overline{b}bj^2 \overline{l})+i
\end{equation} for any $l$. Choose a minimal $l$ (so that $\vv(l-\overline{l})=\vv(l)$), and note that $\vv(\overline{a}al-\overline{b}bj^2 \overline{l})\geq 2\vv(a)+\vv(l)$ as  $\vv(bj)\geq \vv(a)$. Thus \eqref{Eq4.1:vvcondition} is possible only when $\vv(bj)\geq \vv(a)+i$, meaning $g\in J$.
\end{proof}

\begin{cor}
	Suppose that $\E=g^{-1}\L g$ with $g\in \epsilon J$ as above. Then $$\Lambda|_{\E^\times}=\begin{cases}
	\text{$\theta$ or $\overline{\theta}$, if $\dim \Lambda=1$,}\\
	\text{$\oplus\theta\mu$ or $\oplus\overline{\theta\mu}$ with $\mu$ as in Lemma \ref{lem:mainsec:intertwining}, if $\dim \Lambda>1$.} 
	\end{cases}$$
\end{cor}
\begin{proof}
	Apparently conjugation by $j$ changes the character $\theta$ (or $\theta\mu$) to $\overline{\theta}$ (or $\overline{\theta\mu}$). Let $g\in J$ now. Then $\Lambda^g\simeq \Lambda$ as $g$ conjugates the simple character by Lemma \ref{lem:mainsec:intertwining0}, and thus $g$ conjugates $\Lambda$ by the uniqueness in Lemma \ref{lem:mainsec:intertwining}. Thus $\Lambda|_{\E^\times}\simeq \Lambda^g|_{\E^\times}\simeq \Lambda|_{\L^\times}$.
\end{proof}

The discussion so far leads us to the following result:
\begin{prop}\label{Prop4.1:type2int}
	Suppose that $I(\varphi,\chi)\neq 0$ for a type 2 minimal vector, and $J\cap \E^\times =\E^\times$. Then we must have $$\chi=\begin{cases}
	\text{$\theta$ or $\overline{\theta}$, if $\dim \Lambda=1$,}\\
	\text{$\theta\mu$ or $\overline{\theta\mu}$ with $\mu$ as in Lemma \ref{lem:mainsec:intertwining}, if $\dim \Lambda>1$.} 
	\end{cases}.$$
	In that case, $\Phi_\varphi(t)\chi^{-1}(t)$ is constant $1$ on the common support $\E^\times$, and $I(\varphi,\chi)=\Vol(\F^\times\backslash \E^\times).$
\end{prop}


\subsection{When $J\cap \E^\times \subset Z\exp(j_0)\cap Z\exp(\h_0)$}\label{Sec:Liealgrangecase}
In this case we shall use type 1 minimal vectors from Definition \ref{Defn:generalMinimalVec}, as they have better description in a neighbourhood near the identity.  

Recall that for any $\alpha\in O_\pi$ and a choice of $b_0$ between $j_0$ and $h_0$, we can uniquely identified a type 1 minimal vector by Corollary \ref{Cor:type1MinVecAlpha}.
On the other hand, we make the following convention to associate $\alpha_\chi$:
\begin{defn}\label{defn4.2:choiceofalphachi}
	\begin{enumerate}
		\item Let $\E$ be a field. If $c(\chi)\geq 2$, we can associate an element $\alpha_\chi$ as in Lemma \ref{Lem:DualLiealgForChar}(2); if $c(\chi)\leq 1$, we fix an association to $\alpha_\chi$ with $v_\E(\alpha_\chi)=-c(\chi)+c(\psi_\E)$.
		\item If $\E$ splits and $\iota:\E\rightarrow \F\times \F$, we identify $\chi\circ\iota^{-1}$ with $(\chi_1,\chi_2)$, associate $\alpha_{\chi_i}$ similarly as in (1) above, and associate $\alpha_\chi=\iota^{-1}(\alpha_{\chi_1},\alpha_{\chi_2}).$
	\end{enumerate}
\end{defn}
 Then the formula $\chi(\exp(x))=e^{<\alpha_\chi,x>}$ is always true for $x\in \h_0$, regardless of the  choice of $\alpha_\chi$ when $c(\chi)\leq 1$ for example.

\begin{prop}\label{Prop:localperiodintegral}
Suppose that $\varphi$ is a type 1 minimal vector and $J\cap \E^\times \subset Z\exp(j_0)\cap Z\exp(\h_0)$.
For notations as in Section \ref{Sec3.3:Liealg}, we have
\begin{equation}\label{eq4.2:localintformula}
  I(\varphi,\chi)=\Vol(\z_0\backslash \h_0)\int\limits_{g_0\in U_\L(1)\backslash B^1} \Char_{\alpha_\chi+\h_0^\dagger}(g_0^{-1}\alpha g_0)dg_0.
\end{equation}

\end{prop}
\begin{proof}
By the condition  $J\cap \E^\times \subset Z\exp(j_0)\cap Z\exp(\h_0)$, Lemma \ref{Lem:DualLiealgForChar} and Corollary \ref{lem:char=intforB1}, we have
\begin{align}
 I(\varphi,\chi)&=\int\limits_{\F^\times \backslash\E^\times }\Phi_\varphi(t)\chi^{-1}(t)dt=\int\limits_{h\in \z_0\backslash \h_0}e^{<\alpha , h>}\Char_{\B^1}(\exp(h))e^{<-\alpha_\chi, h>} dh \\
 &=\int\limits_{h\in \z_0\backslash \h_0}\int\limits_{g_0\in U_\L(1)\backslash B^1} e^{<g_0^{-1}\alpha g_0, h>} e^{<-\alpha_\chi, h>} dh \notag\\
 &=\int\limits_{g_0\in U_\L(1)\backslash B^1} \int\limits_{h\in\z_0\backslash \h_0}e^{<g_0^{-1}\alpha g_0-\alpha_\chi, h>} dh\notag\\
 &=\Vol(\z_0\backslash \h_0)\int\limits_{g_0\in U_\L(1)\backslash B^1} \Char_{\alpha_\chi+\h_0^\dagger}(g_0^{-1}\alpha g_0)dg_0.\notag
\end{align}

\end{proof}
Recall that type 1 or type 2 minimal vectors can both provide orthonormal basis for $\pi$.
\begin{cor}\label{Cor4.2:relatedcriterion}
	Suppose that $w_\pi=\chi|_{\F^\times}$, $c(\pi)\geq 3$, $\L\simeq \E$, $c(\theta\chi^{-1})$ or $c(\theta\overline{\chi}^{-1})\leq 1$.
	Then $I(\varphi,\chi)\neq 0$ for some type 2 minimal vector if and only if $J\cap \E^\times=\E^\times$ and 
	  $\chi$ is as given in Proposition \ref{Prop4.1:type2int}, in which case $\Phi_\varphi(t)\chi^{-1}(t)$ is constant $1$ on the common support $\E^\times$, and $I(\varphi,\chi)=\Vol(\F^\times\backslash \E^\times).$
\end{cor}
\begin{proof}
	By Lemma \ref{Lem:JcapEwholeorwhat} and Proposition \ref{Prop4.1:type2int}, it remains to prove that when $J\cap \E^\times \subset Z\exp(j_0)\cap Z\exp(\h_0)$, we have $I(\varphi,\chi)=0$ for any type 1 minimal vectors associated to the cuspidal datum $(\mathfrak{A}, J, \Lambda)$, of which type 2 minimal vectors associated to the same cuspidal datum can be written as a linear combination. We only check the case $c(\theta\chi^{-1})\leq 1$. In this case $\L\simeq\E$, and $\theta$, $\chi$ give arise to the same element $\alpha_\theta=\alpha_\chi$. Let $\alpha= g^{-1}\alpha_\chi g$. Then actually $\L^\times=g^{-1}\E^\times g$. We claim that $\alpha-\alpha_\chi\in \h_0^\dagger$ implies that $g\in J=\L^\times K_{ \mathfrak{A}_\L}(i)$. Then we get a contradiction since $J\cap \E^\times=\E^\times$ by Lemma \ref{Lem4.1:JcapEwhole}.
	
	To prove the claim, we first define $v_\E$ and $v_{\B,\E}$ as $\E$ is a field now. Then $\h_0=\{x\in \E,v_{\E}(x)\geq 1\}$, and $\h_0^\dagger=\{x+x^\perp\in \E+\E^\perp|v_\E(x)\geq -1+c(\psi_\E)\}$. We can further assume that $\alpha_\chi$ is imaginary as the conjugation fixed the trace part, so $\Tr(\alpha)-\Tr(\alpha_\chi)=0$ always holds. 
	By writing $g=a+bj\in \E+\E^\perp$ with $j\in \E^\perp$, 
	and taking $\alpha=g^{-1}l g$, $l=\alpha_\chi$ in \eqref{Eq4.1:conjugationonl}, we get that $$\alpha-\alpha_\chi\in \h_0^\dagger \text{\ if and only if }
	v_\E\left(\frac{1}{\Nm(g)}(\overline{a}a\alpha_\chi+\overline{b}bj^2 \alpha_\chi-\alpha_\chi)\right)\geq -1+c(\psi_\E).$$

	This is further equivalent to that $v_\E\left(\frac{-2\Nm(bj)}{\Nm(a)+\Nm(bj)}\right)\geq c(\chi)-1=c(\theta)-1$ as $\Nm(g)=\Nm(a)+\Nm(bj)$ and $v_\E(\alpha_\chi)=-c(\chi)+c(\psi_\E)$, which is only possible when $v_\E(\Nm(bj))\geq v_\E(\Nm(a))+c(\theta)-1$. We can further assume that 
	\begin{equation}\label{Eq4.2:ginJ}
	g=1+bj, \text{\ with }v_\E(\Nm(bj))\geq c(\theta)-1,
	\end{equation} as conjugation of $\E^\times$ on itself is trivial. \eqref{Eq4.2:ginJ} implies that $g\in K_{\mathcal{U}_\E}(i)$ for all possible cases in Example \ref{Example:ListofHJi}, due to the fact that $v_\E(\Nm(\E^\perp))$ is a single coset of $2\Z$.
	
	By Lemma \ref{Lem3.1:whensamevalua} and that $\L^\times=g^{-1}\E^\times g\subset \E^\times\mathfrak{A}_\E^\times$, we get $\mathfrak{A}_\E=\mathfrak{A}_\L$ and $g\in K_{\mathcal{U}_\E}(i)\subset J$.
%
\end{proof}

So far we have proven Theorem \ref{Theo:Main1} when $c(\theta\chi^{-1})$ or $c(\theta\overline{\chi}^{-1})\leq 1$. We shall assume from now on that
\begin{equation}\tag{$*$} c(\theta\chi^{-1}),c(\theta\overline{\chi}^{-1})\geq 2\text{\ if }\L\simeq \E.
\end{equation}

\begin{cor}\label{Cor4.2:unrelatedcriterion}
	Assume ($*$) and that $w_\pi=\chi|_{\F^\times}$. Then $I(\varphi,\chi)\neq 0$ for some type 1 minimal vector if and only if there exists $\alpha\in \alpha_\chi+\h_0^\dagger\cap O_\pi$. 
\end{cor}
\begin{proof}
	Suppose that  $I(\varphi,\chi)\neq 0$ for some type 1 minimal vector. Then we must have $J\cap \E^\times \subset Z\exp(j_0)\cap Z\exp(\h_0)$. This is because otherwise by Lemma \ref{Lem:JcapEwholeorwhat} we have $J\cap \E^\times=\E^\times$, and by Proposition \ref{Prop4.1:type2int} the integral is vanishing for any type 2 minimal vectors associated to the same cuspidal datum $(\mathfrak{A}, J, \Lambda)$, of which a type 1 minimal vector can be written as a linear combination.
	
	Then by Proposition \ref{Prop:localperiodintegral}, there must exists $\alpha\in \alpha_\chi+\h_0^\dagger\cap O_\pi$.

	Conversely, if there exists $\alpha\in \alpha_\chi+\h_0^\dagger\cap O_\pi$, then $I(\varphi,\chi)\neq 0$ by Proposition \ref{Prop:localperiodintegral} for the associated type 1 minimal vector, since $\alpha_\chi+\h_0^\dagger$ is open and the conjugation action is continuous.
\end{proof}

Now we relate the conductor $c(\pi\times\pi_{\chi^{-1}})$ to the geometric information $\Nm(\alpha)-\Nm(\alpha_\chi)$, when $c(\pi)\geq 3$ and ($*$) holds.
In this case one can easily check that 
\begin{equation}\label{sec4.4Eq:conductorlowbd}
c(\pi\times \pi_{\chi^{-1}})\geq c(\pi_{\chi^{-1}})+3.
\end{equation}
\begin{lem}\label{Sec4.4Lem:conductorvsLiealg}
	Let $\pi$ be the supercuspidal representation associated to a minimal character $\theta$ on $\L$, 
	with $c(\theta)\geq 2$, and let $\alpha_\theta$ be associated to $\theta$. Then
	\begin{equation}\label{Eq:cpi=valpha}
	c(\pi)=-v(\Nm(\alpha_\theta)).
	\end{equation}
	Let $\pi_{\chi^{-1}}$ be the representation associated to a 
	character $\chi^{-1}$ on an \'{e}tale quadratic algebra $\E$, such that $\chi|_{\F^\times}=w_\pi$ and ($*$) holds. Then 
	\begin{equation}\label{Eq:conductorofRS}
	c(\pi\times\pi_{\chi^{-1}})=-2v(\Nm(\alpha_\theta)-\Nm(\alpha_\chi)).
	\end{equation}
\end{lem}
Note that when $c(\chi)\leq 1$, we can take any $\alpha_\chi\in \varpi_\E^{-1}O_\E$ and the result remains true.
\begin{proof}
	The first part follows directly from \eqref{Eq:valphatheta} and Example \ref{Example:ListofHJi}. 
	
	To prove \eqref{Eq:conductorofRS}, we first reduce the discussion to the case where $\alpha_\theta$ and $\alpha_\chi$ are both imaginary. By the assumption $\chi|_{\F^\times}=w_\pi$, we can assume WLOG that $\Tr(\alpha_\theta)=\Tr(\alpha_\chi)$. 
	Let $\mu$ be a character of $\F^\times$ associated to $\Tr(\alpha_\theta)/2$ by Lemma \ref{Lem:alphagivechi}. Then
	$c(\pi\times\pi_{\chi^{-1}})=c((\pi\otimes\mu^{-1})\times \pi_{\chi^{-1}\mu_\E})$. $\pi\otimes\mu^{-1}$ is associated to the character $\theta\mu_\L^{-1}$ on $\L$ which is further associated to $\alpha_\theta-\alpha_\mu=\alpha_{\theta,0}$, where $\alpha_{\theta,0}$ is the imaginary part of $\alpha_\theta$ as in Definition \ref{Sec2.1Def:imaginarypart}. Similarly $\chi\mu_\E^{-1}$ is associated to $\alpha_\chi-\alpha_\mu=\alpha_\chi-\Tr(\alpha_\chi)/2=\alpha_{\chi,0}$. Lastly,
	\begin{equation}
	\Nm(\alpha_\theta)-\Nm(\alpha_\chi)=\left(\Nm(\alpha_\theta)-\frac{\Tr(\alpha_\theta)^2}{4}\right)-\left(\Nm(\alpha_\chi) -\frac{\Tr(\alpha_\chi)^2}{4}\right)=\Nm(\alpha_{\theta,0})-\Nm(\alpha_{\chi,0}).
	\end{equation}
	Thus from now on, we assume $\alpha_\theta$ and $\alpha_\chi$ to be both imaginary.

	Now consider the case $\E\not\simeq \L$. Then $\pi$ and $\pi_{\chi^{-1}}$ are not related, and 
	\begin{equation}
	c(\pi\times\pi_{\chi^{-1}})=2\max\{c(\pi),c(\pi_{\chi^{-1}})\}=-2\min\{v(\Nm(\alpha_\theta)),v(\Nm(\alpha_\chi))\}=-2v(\Nm(\alpha_\theta)-\Nm(\alpha_\chi))
	.
	\end{equation}
	The last equality follows from Lemma \ref{Sec2.1Lem:disjointimaginaryNm}.

	
	When $\L\simeq\E$, we have
	\begin{equation}
	c(\pi\times\pi_{\chi^{-1}})=c(\pi_{\theta\chi^{-1}})+c(\pi_{\theta\overline{\chi}^{-1}})=-v(\Nm(\alpha_{\theta\chi^{-1}}))-v(\Nm(\alpha_{\theta\overline{\chi}^{-1}})).
	\end{equation}
	Note that $\alpha_{\theta\chi^{-1}}=\alpha_\theta-\alpha_\chi$, $\alpha_{\theta\overline{\chi}^{-1}}=\alpha_\theta-\overline{\alpha_\chi}$. Since $\alpha_\theta$, $\alpha_\chi$ are imaginary, we have
	\begin{align}
	-v(\Nm(\alpha_{\theta\chi^{-1}}))-v(\Nm(\alpha_{\theta\overline{\chi}^{-1}}))&=-v(\Nm(\alpha_\theta-\alpha_\chi)\Nm(\alpha_\theta-\overline{\alpha_\chi}))\\
	&=-2v((\alpha_\theta-\alpha_\chi)(\alpha_\theta-\overline{\alpha_\chi})) \notag\\
	&=-2v(\alpha_\theta^2-\alpha_\chi^2)\notag\\
	&=-2v(\Nm(\alpha_\theta)-\Nm(\alpha_\chi))\notag.
	\end{align}
\end{proof}

The criterion in Corollary \ref{Cor4.2:unrelatedcriterion} can be actually made more precise. 

\begin{lem}\label{Lem4.2:unrelatedbetter}
	For the same conditions as in Corollary \ref{Cor4.2:unrelatedcriterion}, we have that there exists $\alpha\in \alpha_\chi+\h_0^\dagger$ if and only if there exists $\alpha\in \alpha_\chi+\E^\perp$.
\end{lem}
\begin{proof}
	The existence of $\alpha\in \alpha_\chi+\E^\perp$ apparently implies $\alpha\in \alpha_\chi+\h_0^\dagger$. 
	For the inverse direction,  
	note that from the condition on the central character, we have $\Tr(\alpha)=\Tr(\alpha_\chi)$. 
	By Lemma \ref{Lem:ConjugacyRep}, it suffice to show the existence of $\alpha'=\alpha_\chi+\alpha^\perp$ with $\alpha^\perp \in \E^\perp$ and $\Nm(\alpha')=\Nm(\alpha)$.
	 Equivalently we need to show that $\Nm(\alpha)-\Nm(\alpha_\chi)\in \Nm(\E^\perp)$.
	This is always true when $\E$ splits. So in the following we assume $\E$ to be a field. Then we can define $v_\E$ on $\E$ and $\h_0=\{e\in \E, v_\E(e)\geq 1\}$.
	
	By the condition, there exists $\alpha\in \alpha_\chi+\h_0^\dagger\cap O_\pi$. 
	This implies that we can write $\alpha=\alpha_\chi+\Delta+\beta^\perp$ for $\Delta\in \varpi_\E^{-1+c(\psi_\E)} O_\E$ imaginary, and $\beta^\perp\in \E^\perp$. This implies that
$$\Nm(\alpha)-\Nm(\alpha_\chi+\Delta)= \Nm(\alpha)-\Nm(\alpha_\chi)-(\Tr(\alpha_\chi\overline{\Delta})+\Nm(\Delta))\in \Nm(\E^\perp).$$
By Lemma \ref{Sec4.4Lem:conductorvsLiealg}, $v(\Nm(\alpha)-\Nm(\alpha_\chi))=-\frac{c(\pi\times\pi_{\chi^{-1}})}{2}$, thus it suffices to show that $v(\Tr(\alpha_\chi\overline{\Delta})+\Nm(\Delta))>v(\Nm(\alpha)-\Nm(\alpha_\chi))$, as $\Nm(\E^\perp)$ is closed under multiplication by $U_\F(1)$. 

Indeed $v(\Nm(\Delta))\geq -2>-\frac{c(\pi\times\pi_{\chi^{-1}})}{2}$, and 
$$v(\Tr(\alpha_\chi\overline{\Delta}))\geq \frac{v(\Nm(\alpha_\chi\Delta))}{2}\geq -\frac{c(\pi_{\chi^{-1}})+2}{2}>-\frac{c(\pi\times\pi_{\chi^{-1}})}{2}.$$
The last inequality can be verified case  by case as in the proof of Lemma \ref{Sec4.4Lem:conductorvsLiealg}, and using ($*$) when $\L\simeq \E$.
\end{proof}


\subsection{Dichotomy and Tunnell-Saito's $\epsilon-$value test}
It remains to compute the size of the local integral for the remaining range. Before that, we first do a brief discussion on the dichotomy and Tunnell-Saito's $\epsilon-$value test.

As type 1 or type 2 minimal vectors form a basis for $\pi$, the conditions in Corollary \ref{Cor4.2:relatedcriterion} and \ref{Cor4.2:unrelatedcriterion} are equivalent to Tunnell-Saito's $\epsilon-$value test. We shall review the relation of the compact induction datum with the Langlands/Jacquet-Langlands correspondence in Appendix \ref{SubSec:Langlands-compactInd}, and explicitly verify the equivalence to the Tunnell-Saito's $\epsilon-$value test in the most interesting range in Appendix \ref{Sec:AppendixB}. 

Meanwhile, we show here that in the context of Corollary \ref{Cor4.2:unrelatedcriterion}, one can give another convenient proof for the dichotomy.
\begin{prop}\label{Prop:Dichotomy}
Assume ($*$) and that $w_\pi=\chi|_{\F^\times}$. Then 
	$\dim\Hom_{\E^\times }(\pi^\B\otimes \chi^{-1},\C)>0$ for exactly one possible $\B$.

\end{prop}

\begin{proof}
	It is clear that $\dim\Hom_{\E^\times }(\pi^\B\otimes \chi^{-1},\C)>0$ if and only if $I(\varphi,\chi)\neq 0$ for some test vector.
From Corollary \ref{Cor4.2:unrelatedcriterion} and the proof of Lemma \ref{Lem4.2:unrelatedbetter}, we see that the existence of a test vector for $I(\varphi,\chi)$ is equivalent to whether $ \Nm(\alpha)-\Nm(\alpha_\chi)\in (\Nm(\E^\perp ))^\times$.

%
 Note that whether $\B$ splits or not is equivalent to whether $0$ can be represented as the norm of a non-trivial element in $\B$, which in turn is equivalent to whether $(-\Nm(\E))^\times$ intersects with $(\Nm(\E^\perp))^\times$ non-trivially or not when $\E$ is a field. Thus for different $\B$, $(\Nm(\E^\perp ))^\times$ have disjoint complementary images in $\F^\times $,  giving the dichotomy. 

This argument also works when $\E$ splits. In this case, $(\Nm(\E^\perp ))^\times=\F^\times$ if $\B$ is the matrix algebra, and empty if $\B$ is the division algebra.

\end{proof}

%


\subsection{Size of the local period integral}
In this subsection, we shall figure out the size of the local integral in the setting of Corollary \ref{Cor4.2:unrelatedcriterion}. 


We first show the following main result on the geometry behind the test vector problem for the local period integral.

\begin{prop}\label{Sec4.4Prop:differenceinalpha}
Let $\alpha_\theta$ be \imaginary \text{\ }as in Definition \ref{Defn:imaginary} and $O_\pi$ be the associated coadjoint orbit, $\E$ be an \'{e}tale quadratic algebra, $C(\pi\times\pi_{\chi^{-1}})=q^{c(\pi)+l}$ for $l\geq 2$. Suppose that there exists $\alpha\in \alpha_\chi+\E^\perp\cap O_\pi$. Then we can write $\alpha=\alpha_\chi+\alpha^\perp$ for 
 $\alpha^{\perp}\in \E^\perp$ with
\begin{equation}
v(\Nm(\alpha^\perp))=c(\pi\times\pi_{\chi^{-1}})=-\frac{c(\pi)+l}{2}.
\end{equation}
Alternatively for $\beta\in \E$ imaginary
, we can write
\begin{equation}
\beta=\beta_{\L}+\beta^\perp
\end{equation}
with $\beta_{\L}\in \L$ imaginary and $\beta^\perp\in \L^\perp$, then 
\begin{equation}\label{Sec4.4Eq:secondvalRelation}
v(\Nm(\beta^\perp))
=v(\Nm(\beta))+\frac{c(\pi)-l}{2}.
\end{equation}
\end{prop}

\begin{proof}
We already see from Lemma \ref{Sec4.4Lem:conductorvsLiealg} and the proof of Lemma \ref{Lem4.2:unrelatedbetter} that
\begin{equation}\label{Eq:Inproof4.5}
v(\Nm(\alpha^\perp))=v(\Nm(\alpha_\theta)-\Nm(\alpha_\chi))=-\frac{c(\pi)+l}{2}.
\end{equation}
Now we prove \eqref{Sec4.4Eq:secondvalRelation}.
By the expansions for $\alpha$ and $\beta$, we have
\begin{equation}
[\alpha^\perp, \beta]=[\alpha,\beta]=[\alpha, \beta^\perp].
\end{equation}
We first consider the case where $\E$ is a field extension, so $v_{\B,\E}$ can be defined similarly as $\vv$.
Since $\theta$ is minmimal and  $\beta$ is \imaginary, by Lemma \ref{Lem3.1semivaluation}(2) and
 applying Lemma \ref{lem7.1:conjugationEperp} to both $\vv$ and $v_{\B,\E}$, we get 
 $$v(\Nm(\alpha^\perp))+v(\Nm(\beta))=v(\Nm([\alpha^\perp, \beta]))=v(\Nm([\alpha,\beta^\perp]))=v(\Nm(\beta^\perp))+v(\Nm(\alpha_\theta)).$$ Using \eqref{Eq:Inproof4.5} and Lemma \ref{Sec4.4Lem:conductorvsLiealg}, we have
\begin{equation}
v(\Nm(\beta^\perp))=v(\Nm(\beta))-\frac{c(\pi)+l}{2}-v(\Nm(\alpha_\theta))=v(\Nm(\beta))+\frac{c(\pi)-l}{2}.
\end{equation}
Now if $\E$ is split, we can't compute $v(\Nm([\alpha^\perp, \beta]))$ as above due to the lack of $v_{\B,\E}$, but we can do computations more explicitly.
Note that this case only occurs when $\B$ is the matrix algebra. Up to a conjugation, we may assume that $\E$ is the diagonal torus for $\B$, and  $\alpha_\chi=\zxz{a_1}{0}{0}{a_2}$. 
 Elements in $O_\pi$ have fixed trace and norm. Recall that $\Tr(\alpha_\chi)=\Tr(\alpha_\theta)$ as $\chi|_{\F^\times}=w_\pi$.
  So we can pick an element $\alpha\in O_\pi$ having the shape $\alpha=\zxz{a_1}{b}{c}{a_2}$ with $\Tr(\alpha)=a_1+a_2=\Tr(\alpha_\chi)$ and
 \begin{equation}\label{Eq:proof4.6eq1}
  \det(\alpha)=a_1a_2-bc=\Nm(\alpha_\theta)
 \end{equation} 
 being fixed. As long as $b\neq 0$, we can take $c=\frac{a_1a_2-\Nm(\alpha_\theta)}{b}$.
 Thus we have written $\alpha=\alpha_\chi+\alpha^\perp$ with $\alpha^\perp=\zxz{0}{b}{c}{0}\in \E^\perp$. 

For $\beta\in \E$ imaginary, we can write $\beta=\zxz{\beta_1}{0}{0}{-\beta_1}$ with $\Nm(\beta)=-\beta_1^2$. Then 
we compute directly that $[\alpha^\perp,\beta]=2\beta_1\zxz{0}{-b}{c}{0}$, and $$v(\Nm([\alpha^\perp, \beta]))=v(4\beta_1^2bc)=v(\Nm(\beta))+v(\Nm(\alpha^\perp)).$$
The remaining argument will be the same as in the field extension case.

\end{proof}
The quantity $v(\Nm (\alpha^\perp))$  compared with $\min\{v(\Nm(\alpha)), v(\Nm(\alpha_\chi))\}$ (and similarly  $v(\Nm(\beta^\perp))$ v.s. $\min\{v(\Nm(\beta)), v(\Nm(\beta_\L))\}$) measures how closely the embeddings of $\L$ and $\E$ align with each other.

\begin{prop}\label{Prop:SizeofLocalint}
Assume that $p$ is large enough, $\alpha_\theta$ is \imaginary, $C(\pi\times\pi_{\chi^{-1}})=q^{c(\pi)+l}$ for $l\geq 2$ and 
$c(\theta\chi^{-1}), c(\theta\overline{\chi}^{-1}) \geq 2$ if $\E\simeq \L$. Suppose that
there exists $\alpha\in O_\pi$ such that $\alpha\in \alpha_\chi+\h_0^\dagger$. Then for type 1 minimal vector $\varphi$ associated to $\alpha$ as in Corollary \ref{Cor:type1MinVecAlpha},   we have that $\Phi_\varphi\chi^{-1}$ is constant $1$ on the common support $\Supp \Phi_\varphi \cap \E^\times$, and
\begin{equation}
 I(\varphi,\chi)\asymp \frac{1}{q^{l/4}}.
\end{equation}
\end{prop}

\begin{proof}
	By Corollary \ref{Cor4.2:unrelatedcriterion} and its proof, we already know that $J\cap \E^\times \subset Z\exp(j_0)\cap Z\exp(\h_0)$ and $I(\varphi,\chi)\neq 0$. Thus  $\Phi_\varphi\chi^{-1}$ is multiplicative on $\Supp \Phi_\varphi \cap \E^\times$ by Lemma \ref{Cor:MCofGeneralMinimalVec}, and consequently
	must be constant $1$ on $\Supp \Phi_\varphi \cap \E^\times$ for the integral to be nonzero. This implies that
\begin{equation}
 I(\varphi,\chi)=\Vol(ZB^1\cap \E^\times).
\end{equation}
Thus it remains to figure out $\Vol(ZB^1\cap \E^\times)$ when $\E$ and $\L$ are aligned as in Proposition \ref{Sec4.4Prop:differenceinalpha}.

Now any element in $e\in \E^\times$ can be written as $e=a+b\beta$, where $\beta$ is imaginary, and for convenience we assume that $v(\beta^2)=e_\E-1$. As $ZB^1\cap \E^\times \subset Z\exp(j_0)\cap Z\exp(\h_0)$, we can assume that $v(a)=0$ and $v(b)\geq 2-e_\E$. Using the decomposition for $\beta$ as in Proposition \ref{Sec4.4Prop:differenceinalpha}, we can further write
\begin{equation}
e=a+b\beta_\L+b\beta^\perp.
\end{equation}

Consider first the case $B^1=H^1$. By Definition \ref{Defn:ii'&JHs}, we have that $e\in Z B^1$ if and only if the following two equations are satisfied.
\begin{equation}\label{Sec4.4Eq:EcapJEq1}
\vv(b\beta_\L)\geq 1,
\end{equation}
\begin{equation}\label{Sec4.4Eq:EcapJEq2}
 \vv(b\beta^\perp)\geq i'.
\end{equation}
We claim that \eqref{Sec4.4Eq:EcapJEq1} is true as long as \eqref{Sec4.4Eq:EcapJEq2} is true. 
To see this, first note that $i'\geq 1$. As in Proposition \ref{Sec4.4Prop:differenceinalpha}, $v(\Nm(\beta))=v(\Nm(\beta_\L)+\Nm(\beta^\perp))$, so either $v(\Nm(\beta_\L)),v(\Nm(\beta^\perp))\geq v(\Nm(\beta))=e_\E-1$, or $v(\Nm(\beta_\L))=v(\Nm(\beta^\perp))<e_\E-1$.  In the first case, one can check that \eqref{Sec4.4Eq:EcapJEq1} is true anyway. In the second case, we have $\vv(b\beta_\L)= \vv(b\beta^\perp)\geq i'$. So the claim is also true.

By Lemma \ref{Lem3.1semivaluation}(2) 
$ \vv(b\beta^\perp)=\frac{e_\L}{2}v(\Nm(b\beta^\perp))=\frac{e_\L}{2}v(b^2\Nm(\beta^\perp))$.
Then by Proposition \ref{Sec4.4Prop:differenceinalpha}, \eqref{Sec4.4Eq:EcapJEq2} is now equivalent to that
\begin{equation}
v(b)\geq \frac{i'}{e_\L}-\frac{v(\Nm(\beta^\perp))}{2}=\frac{i'}{e_\L}-\frac{e_\E-1+\frac{c(\pi)-l}{2}}{2}.
\end{equation}
Then by Lemma \ref{Sec2.3Lem:VolumeofE}, we get
\begin{equation}\label{Sec4.4Eq:SizeofIntCase1}
I(\varphi,\chi)=L(\eta_{\E/\F},1)\frac{1}{q^{\lceil \frac{i'}{e_\L}-\frac{e_\E-1+\frac{c(\pi)-l}{2}}{2}\rceil }}.
\end{equation}

Now if $B^1$ is a proper intermediate subgroup between $H^1$ and $J^1$, depending on how $B^1$ aligns with the line $\F\beta^\perp$, we may need to adjust \eqref{Sec4.4Eq:EcapJEq2} and get slightly different answers. If the space $B^1/H^1$ is not generated by $\beta^\perp$, then we still get \eqref{Sec4.4Eq:EcapJEq2} and \eqref{Sec4.4Eq:SizeofIntCase1}. On the other hand if  $B^1/H^1$ is generated by $\beta^\perp$, then $e\in ZB^1$ if and only if 
\begin{equation}\label{Eq4.4:vvb}
\vv(b\beta^\perp)\geq i.
\end{equation}

Then we get $v(b)\geq \lceil\frac{i}{e_\L}-\frac{e_\E-1+\frac{c(\pi)-l}{2}}{2}\rceil,$ and
\begin{equation}\label{Sec4.4Eq:SizeofIntCase2}
I(\varphi,\chi)=L(\eta_{\E/\F},1)\frac{1}{q^{\lceil \frac{i}{e_\L}-\frac{e_\E-1+\frac{c(\pi)-l}{2}}{2}\rceil }}.
\end{equation}

Note that there might seem to be an exception when   $n=1$ in Example \ref{Example:ListofHJi}(4), so \eqref{Sec4.4Eq:EcapJEq1} might be stronger than \eqref{Eq4.4:vvb}. But in this case $\B$ is a division algebra, so we must either be in situation 1, when $e_\E=1$ and
$v(\Nm(\beta^\perp))>v(\Nm(\beta_\L)) =v(\Nm(\beta))=0$; or in situation 2, when $e_\E=2$ and $v(\Nm(\beta_\L))>v(\Nm(\beta^\perp)) =v(\Nm(\beta))=1$. As we have assumed ($*$), situation 1 does not occur for $n=1$ in Example \ref{Example:ListofHJi}(4).
In situation 2 one can check directly that \eqref{Eq4.4:vvb}   implies $\vv(b)\geq 0$, which implies \eqref{Sec4.4Eq:EcapJEq1}.

For all cases in Example \ref{Example:ListofHJi}, we have 
\begin{equation}\label{Eq:Assmpfori}
 i, i'\asymp \frac{c(\pi)e_\L}{4}.
\end{equation}
Then
\begin{equation}
  I(\varphi,\chi)\asymp_q \frac{1}{q^{\frac{l}{4}}}.
\end{equation}

For potential application, we shall also list here the maximal size (so $B^1/H^1$ aligns with $\beta^\perp$ when $B^1\neq H^1$) of the local integral for each individual case in Example \ref{Example:ListofHJi}. In the appendix we shall see that the computations from another approach are consistent with the formulae here.
\begin{equation}\label{Eq:SizeofIntAllCase}
I(\varphi,\chi)= L(\eta_{\E/\F},1)\begin{cases}
\frac{1}{q^{\lceil \frac{l}{4}+\frac{-e_\E+1}{2}\rceil}}, &\text{\ Case (1),}\\
\frac{1}{q^{\lceil \frac{l}{4}-\frac{1}{2}+\frac{-e_\E+1}{2}\rceil}}, &\text{\ Case (2),}\\
\frac{1}{q^{\lceil \frac{l}{4}-\frac{1}{4}+\frac{-e_\E+1}{2}\rceil}}, &\text{\ Case (3),}\\
\frac{1}{q^{\lceil \frac{l}{4}-\frac{1}{2}+\frac{-e_\E+1}{2}\rceil}}, &\text{\ Case (4),}\\
\frac{1}{q^{\lceil \frac{l}{4}+\frac{-e_\E+1}{2}\rceil}}, &\text{\ Case (5),}\\
\frac{1}{q^{\lceil \frac{l}{4}-\frac{1}{4}+\frac{-e_\E+1}{2}\rceil}}, &\text{\ Case (6).}
\end{cases}
\end{equation}
\end{proof}



\begin{rem}
	It is also possible to directly compute $\Vol((\alpha_\chi+\h_0^\dagger)\cap\{g_0^{-1}\alpha g_0\}_{g_0\in U_\L(1)\backslash B^1})$ directly. But the approach we take above is simpler, and also consistent with some computations in the next section.
\end{rem}


\section{Application to hybrid subconvexity bound}\label{Sec:hybridSubConv}
In this section we shall work with the global setting and notations, and use subscript $v$ to denote the corresponding local notations.
\subsection{Global notations and basic results}\label{Sec:hySubGlobal}

Let $\F$ be a totally real number field, or just $\Q$ for simplicity. 
Fix a finite place $\mathfrak{p}$ over some $p$ that we are interested in (i.e. allow ramifications at this place), such that 
$p$ is large enough. Let $\A_\F$ be the Adele ring, and $\A_\F^\times$ be the Idele group associated to $\F$ respectively.

Let $\B/\F$ be a quaternion algebra, equipped with an involution $x\mapsto \overline{x}$ as in the local case. We assume that $\B$ is locally a division algebra at all Archimedean places and possibly some fixed finite places away from $\mathfrak{p}$. Let $G=\PGL_2$ and $G'=\PB^\times$. Let $\E$ be a fixed quadratic field extension of $\F$ whose embedding in $\B$ is also fixed. Similar to the local situation, we can choose $\epsi\in \B^\times$ such that $j\in \E^\perp$, $\overline{j}=-j$, $j$ normalizes $\E$ and $\B=\E\oplus \epsi \E$. 

Denote $[\F^\times\backslash \E^\times]=\A_\F^\times \E^\times \backslash \A_\E^\times$, $[ G']=G'(\F)\backslash G'(\A)$. We shall take the Tamagawa measures on  both sets. 

Let $\mathcal{A}(G')$ denote the set of irreducible automorphic representations of $G'(\A)$. 
For any $\pi'\in \mathcal{A}(G')$, let $\pi\in \mathcal{A}(G)$ be its image under the Jacquet-Langlands correspondence. Let $\chi$ be a character of $\A_\F^\times \backslash \A_\E^\times$. Let $\mathfrak{c}(\chi)$ be the finite conductor for $\chi$, and $C(\chi)$ be the absolute norm of it. It is clear that $C(\chi)\asymp_{\E} C(\pi_{\chi^{-1}})$.

For two automorphic forms $\varphi_1, \varphi_2\in \pi' \in \mathcal{A}(G')$, define
\begin{equation}
(\varphi_1,\varphi_2)=\int\limits_{x\in [G']}\varphi_1(x)\overline{\varphi_2(x)}dx.
\end{equation}


Let $f$ be a Schwartz function on $G'(\A)$. For $\varphi\in \pi'$ an automorphic form on $G'$, define the action
\begin{equation}
R(f)\varphi(g)=\int\limits_{x\in G'(\A)} f(x)\varphi(gx)dx.
\end{equation}
We make similar definitions for local components $f_v$. The global action $R(f)$ is directly factorized into local actions once $f$ itself is factorisable.

When $v$ is an Archimedean place of $\F$, by our assumption on $\B$, we have for an Archimedean place $v$,
\begin{equation}
 \pi'_v\simeq \text{Sym}^{2k_v-2}V\otimes \det{}^{-(k_v-1)},
\end{equation}
where $V$ is the standard 2-dimensional complex representation of $\B_v^\times$.

When $v$ is a non-Archimedean place, let $q_v$ denote the size of the residue field of $\F_v$. Let $q=q_\mathfrak{p}$.

At the fixed place $\mathfrak{p}$, suppose that $\pi_\mathfrak{p}'$ is associated to  a character $\theta$ over a quadratic field extension $\L_\mathfrak{p}$ as in Section \ref{Sec:CompactInd}. 
\subsubsection{Global Waldspurger's period integral}
By \cite{Walds} we have the following special case of the global Waldspurger's period integral formula.
\begin{equation}\label{Sec5.1Eq:GlobalWalds}
\left|\int_{t\in [\F^\times\backslash\E^\times]}\varphi(t)\chi^{-1}(t)dt\right|^2=\frac{\zeta(2)L(\pi\times \pi_{\chi^{-1}},1/2)}{2L(\pi,Ad,1)}\prod\limits_{v}I^0_v(\varphi,\chi) .
\end{equation}
Here $\Pi$ is the base change of $\pi$ to $\E$, $v$ is over Archimedean places as well as any finite places with ramifications, $I^0_v(\varphi,\chi)=I_v(\varphi,\chi)$ for $v|\infty$, and
\begin{equation}
I^0_v(\varphi,\chi)=\frac{L_v(\pi,Ad,1)L_v(\eta,1)}{\zeta_v(2)L_v(\pi\times \pi_{\chi^{-1}},1/2)}I_v(\varphi,\chi)
\end{equation}
for $v$ finite, where $I_v(\varphi,\chi)$ is as in Theorem \ref{Theo:Main1}. 
\subsubsection{The relative trace formula}
We recall Jacquet's relative trace formula for Waldspurger's period integral from \cite{jacquet_sur_1987}, under some specific settings from \cite{feigon_averages_2009}. \cite{FileMartin:17a} also slightly generalizes the computations in \cite{feigon_averages_2009}, but will be of no direct use for us. 
\begin{theo}\label{Theo:RelativeTrace}
We have
\begin{equation}
J(f,\chi)=\I(f,\chi),
\end{equation}
where on the spectrum side
\begin{equation}\label{Eq:Specside}
J(f,\chi)=\sum_{\pi'\in \mathcal{A}(G')}\sum_{\varphi\in \mathcal{B}(\pi')}  \int_{t\in [\F^\times\backslash\E^\times]}R(f)\varphi(t)\chi^{-1}(t)dt\overline{  \int_{t\in [\F^\times\backslash\E^\times]}\varphi(t)\chi^{-1}(t)dt },\end{equation}
and on the geometric side
\begin{equation}\label{Eq:Geoside}
\I(f,\chi)=\Vol(\A_\F^\times \E^\times\backslash \A_\E^\times)(\I(0,f)+{\delta(\chi^2)}\I(\infty, f))+\sum\limits_{\xi\in \epsi^2 \Nm(\E^\times)}\I(\xi,f).
\end{equation}
Here $\mathcal{B}(\pi')$ is an orthogonal basis for $\pi'$. {$\delta(\Omega)=1$ if $\Omega=1$ and $\delta(\Omega)=0$ otherwise.}
The orbit integrals in \eqref{Eq:Geoside} are factorisable and defined as follows,
\begin{equation}
\I(0,f)=\int_{\A_\F^\times\backslash \A_\E^\times}f(t)\chi^{-1}(t)dt,
\end{equation}
\begin{equation}
\I(\infty, f)=\int_{\A_\F^\times\backslash \A_\E^\times}f(t\epsi)\chi^{-1}(t)dt,
\end{equation}
For $x\in \E^\times$ and $\xi =-\Nm(\epsi x)=\epsi^2\Nm(x)\in \F^\times$, 
\begin{equation}
 \I(\xi,f)=\int_{\A_\F^\times\backslash \A_\E^\times}\int_{\A_\F^\times\backslash \A_\E^\times}f(t_1(1+\epsi x) t_2)\chi^{-1}(t_1t_2) dt_1dt_2.
\end{equation}
\end{theo}
Note that $\I(\xi,f)$ is independent of the choice of $x$. 
Let $\I_v(\xi,f_v)$ denote the local integrals corresponding to the orbital integrals defined above when $f=\otimes f_v$.
The sum on the geometric side is actually finite, using the vanishing results for $\I_v(\xi, f_v)$. See \cite{feigon_averages_2009} for more details.

On the other hand, with proper choice of $f$ and $\mathcal{B}(\pi')$, $R(f)\varphi$ is either 0 or a multiple of $\varphi$, so the spectral side $J(f,\chi)$ of the relative trace formula becomes a sum of $\frac{L(\pi\times \pi_{\chi^{-1}},1/2)}{L(\pi,Ad,1)}$ by \eqref{Sec5.1Eq:GlobalWalds}.

For simplicity we shall work in the following two settings:
\begin{asmp}[Disjoint ramifications]\label{Assumption:Localdisjoint}

\begin{enumerate}
\item When $v\nmid \infty \mathfrak{p}$, we assume $\pi_v$ 
to be unramified.
 
\item When $v|\infty$, we assume $\chi_v(z)=(\frac{z}{\overline{z}})^{m_v}$ for $m_v\neq 0$ and $\pi_v$ to be a discrete series representation with lowest weight 
$k_v>|m_v|$. 

\item for $v=\mathfrak{p}$, we assume that $\pi_\mathfrak{p}$ is a supercuspidal representation, 
and $c(\chi_v)=0$. 

\end{enumerate}
Without confusion, we shall identify $f_v$ from now on with a function on $\B_v^\times$ which is invariant by $\F_v^\times$ and compactly supported $\mod \F_v^\times$.
Then we choose the test functions $f_v$ as follows:
\begin{enumerate}
\item When $v\nmid \infty \mathfrak{p}$, $f_v$ is the characteristic function of
$\F_v^\times R_v^\times$ where  $R_v$ is a maximal order such that $R_v\cap \E_v=O_{\F_v}+\varpi_{\E_v}^{c(\chi_v)}O_{\E_v}$. 
\item When $v|\infty$, $f_v$ is the complex conjugate of the matrix coefficient of an element in $\pi'_v$, on which $\E_v^\times$ acts by $\chi_v$. 
\item When $v=\mathfrak{p}$, let $\varphi_\mathfrak{p}$ be the minimal vector which is a test vector for the local Waldspurger's period integral as in Theorem \ref{Theo:Main1} and let $B^1$ be the compact open subgroup associated to $\varphi_\mathfrak{p}$ as in Section \ref{Sec:CompactinductionSummary+Kirillov}. Then
\begin{equation}
f_\mathfrak{p}=\frac{1}{\Vol(Z_\mathfrak{p}\backslash Z_\mathfrak{p}B^1)}\overline{\Phi}_{\varphi_\mathfrak{p}} |_{Z_vB^1}.
\end{equation}

\end{enumerate}

\end{asmp}

Note that $\frac{1}{\Vol(Z_\mathfrak{p}\backslash Z_\mathfrak{p}B^1)}
\asymp_q q^{c(\pi)/2}$ by the proof of Corollary \ref{cor:matrixcoeff}.

\begin{asmp}[joint ramifications]\label{Assumption:Local}
	
	\begin{enumerate}
		\item[(1')] When $v\nmid \infty \mathfrak{p}$, we assume $\pi_v$ and $\chi_v$ to be unramified.
		
		\item[(2)] When $v|\infty$, we assume $\chi_v(z)=(\frac{z}{\overline{z}})^{m_v}$ for $m_v\neq 0$ and $\pi_v$ to be a discrete series representation with lowest weight 
		$k_v>|m_v|$. 
		
		\item[(3')] for $v=\mathfrak{p}$, we assume that $\pi_\mathfrak{p}$ is a supercuspidal representation, 
		and $c(\pi_{\chi,\mathfrak{p}})>c(\pi_\mathfrak{p})$. 
	\end{enumerate}
The choices of the test functions $f_v$ are essentially the same as for Setting \ref{Assumption:Localdisjoint}. Note that for (3'), $f_\mathfrak{p}$ is constructed similarly, except that the related minimal vector is different.

\end{asmp}

\begin{rem}\label{Sec5Rem:Setting} 

\begin{enumerate}
	\item[1.] The settings in 
	(1)/(1') and (2) are mainly for simplicity, and can be easily relaxed. In particular the assumption $m_v\neq 0$ is to guarantee that $\I(\infty, f)=0$, 
	which can also be achieved by assuming that $\delta(\chi^2)=0$ as in Theorem \ref{Theo:RelativeTrace}. On the other hand,
	one can also similarly control $\I(\infty, f)$ as for $\I(0,f)$ when it is non-vanishing. At finite places, one can also allow some joint ramifications with bounded conductors.
	\item[2.] Note that by the assumption that $\B_v$ is a division algebra when $v|\infty$, we must have $ \iota_v(\Nm(\epsi))<0$ for any $v|\infty$ and the corresponding real embedding $\iota_v$.
	\item[3.] (3)/(3') is to match the main topic of this paper. We expect similar results when $\pi_v$ is a principal series representation. 
	Note that by the condition 
	and the discussions in Section \ref{Sec:mainmethod}, we will only be using Type 1 minimal vectors. In that case,
	$\overline{\Phi}_{\varphi_\mathfrak{p}} |_{Z_vB^1}$ is actually a character, and 
	\begin{equation}
	R(f_\mathfrak{p})\varphi_{\mathfrak{p}}=\varphi_{\mathfrak{p}}.
	\end{equation}
\end{enumerate}

\end{rem}

We collect the following basic computations from \cite{feigon_averages_2009}. All of them can be made more precise, which would be necessary if one pursues an asymptotic formula. But for our purpose, we need only  primitive results for some computations.
\begin{lem}\label{Lem5.1:constantmultiple}
	$R(f)\varphi=0$ unless the followings hold:
	\begin{enumerate}
		\item For $v\nmid \infty \mathfrak{p}$, $\pi_v'$ is unramified and $\varphi_p$ is the unique element in $\pi_v$ which is $R_v^\times-$invariant.
		\item For $v|\infty$, $\pi_v'$ is the prescribed representation $\pi_v$ and $\varphi_v$ is the unique weight $m_v$ element in it.
		\item For $v=\mathfrak{p}$, $\pi_v'$ is construct from $\theta'$ over $\L$ where $\theta'=\theta\mu$ for $c(\mu)\leq 1$, and $\varphi_\mathfrak{p}$ is the unique element on which $B^1$ acts by $\tilde{\theta}|_{ZB^1}$.
	\end{enumerate}
In that case, we have $R(f)\varphi=C_{\B,f}\varphi$ for a positive constant $C_{\B,f}$ which is bounded from $0$ and $\infty$ as $C(\pi')\rightarrow \infty$.
\end{lem}
\begin{proof}
	Part (1),(2) and the contribution to $C_{\B,f}$ from these places follow from \cite[Section 3.2]{feigon_averages_2009}. Part (3) and its contribution follows from Lemma \ref{Cor:MCofGeneralMinimalVec} and Remark \ref{Sec5Rem:Setting} above. Note that $\theta'$ differ from $\theta$ by a character of level at most $1$, as $\theta'|_{U_\L(1)}=\tilde{\theta}|_{U_\L(1)}$.
\end{proof}
\begin{rem}
	The ambiguity for the local compact induction datum $\theta'$ here is a trade-off for the choice of $f_\mathfrak{p}$. In the case when $B^1=H^1$ for the local compact induction datum, one can choose $f_\mathfrak{p}$ to be  $\overline{\Phi}_{\varphi_\mathfrak{p}}$ up to a constant, then one get exactly $\theta'=\theta$.
\end{rem}
\begin{lem}\label{Lem5.1:LocalWalds}
	We have $I^0_v(\varphi,\chi)>0$, and
	\begin{equation}
	I^0_v(\varphi,\chi)\begin{cases}
	=1, &\text{\ when $\pi_v$, $\chi_v$ and $\E_v$ are unramified;}\\
	=\Vol(\F_v^\times\backslash \E_v^\times), &\text{\  when $v|\infty$;}\\
	\asymp_{\F,\E,\B} \frac{1}{\sqrt{C(\chi_v)}}, &\text{\ $v\nmid \infty\mathfrak{p}$;}\\
	\asymp_q \frac{1}{q^{l/4}}, &\text{\ when $v=\mathfrak{p}$ by Theorem \ref{Theo:Main1}.}
	\end{cases}
	\end{equation}
\end{lem}
\begin{rem}
	Note that when $c(\chi_v)>0$, the contribution to the integral comes essentially from $\E_v\cap R_v$. See  \cite[Section 3.3]{feigon_averages_2009} and reference therein for more details.
\end{rem}
\begin{lem}\label{lem:localorbitint1}
For $v|\infty$, 
\begin{equation}
\I_v(0,f_v)=\Vol(\F_v^\times\backslash \E_v^\times).
\end{equation}
\begin{equation}
 \I_v(\infty,f_v)=\begin{cases}
 \Vol(\F_v^\times\backslash \E_v^\times)(-1)^{k_v-1}, &\text{ \ if $m_v=0$},\\
 0, &\text{\ otherwise.}
 \end{cases}
\end{equation}
Due to this and our assumption that $m_v\neq 0$, we won't consider $\I_v(\infty, f_v)$ at other places. 

For $v|\infty$ and $\xi\neq 0,\infty$, $\iota_v(\xi)<0$, 
\begin{equation}\label{Eq:regularorbitsinf}
\I_v(\xi, f_v)=\frac{\Vol(\F_v^\times\backslash \E_v^\times)^2}{(1-\xi)^{k_v-1}}\sum\limits_{i=0}^{k_v-|m_v|-1}\left(\begin{array}{c}
k_v+m_v-1\\
i
\end{array}\right)
\left(\begin{array}{c}
k_v-m_v-1\\
i
\end{array}\right)
(-\xi)^i\ll_{k_v,m_v} 1.
\end{equation}

For $v\nmid \infty \mathfrak{p}$,
\begin{equation}
\I_v(0,f_v)\begin{cases}
=1, &\text{\ if $\pi_v,\chi_v,\E_v$ are unramified };\\
\ll_{\F,\E,\B} \frac{1}{\sqrt{C(\chi_v)}}, &\text{\ otherwise}.
\end{cases}
\end{equation}
\begin{equation}\label{Eq:regularorbitsunramifed}
\I_v(\xi,f_v)\begin{cases}
=0, &\text{\ if $v(1-\xi)>v(\mathfrak{d}_{\E/\F}\mathfrak{c}(\chi))$},\\
=1,&\text{\ if $v(1-\xi)=v(\xi)=0$ and $\chi_v$ is unramified},\\
\ll_{\F,\E,\B} \frac{1}{\sqrt{C(\chi_v)q^{v(1-\xi)}}}|v(\xi)|, &\text{\ in general}.
\end{cases}
\end{equation}
Here $\mathfrak{d}_{\E/\F}$ is the discriminant of $\E/\F$. $q=|\varpi_v|_v^{-1}$.
\end{lem}
Note that \eqref{Eq:regularorbitsunramifed} follows directly from \cite[Lemma 4.12, 4.13, 4.19]{feigon_averages_2009}. For our purpose, we need either a vanishing result or a suitable upper bound for the local orbital integrals.

\subsection{The proof of Theorem \ref{Theo:HybridsubDisjointram} and \ref{Theo:Hybridsubconvexity}}
We shall need the following two local lemmas.
\begin{lem}\label{Lem:RegularLocal0}
	For $v=\mathfrak{p}$ as in Setting \ref{Assumption:Localdisjoint}(3),
	\begin{equation}
	\I_\mathfrak{p}(0,f_\mathfrak{p})\asymp_q q^{c(\pi_\mathfrak{p})/4},
	\end{equation}
	\begin{equation}\label{Eq5.2:localintatpSetting1regular}
	\I_\mathfrak{p}(\xi,f_\mathfrak{p})\begin{cases}
	\ll_q q^{c(\pi_\mathfrak{p})/4} &\text{\ if $v(\xi)\geq c(\pi_\mathfrak{p})/2+O(1)$},\\
	\ll_q q^{v(\xi)/2}, &\text{\ if $0<v(\xi)\leq c(\pi_\mathfrak{p})/2+O(1)$}, \\
	=0, &\text{\ otherwise.}
	\end{cases}
	\end{equation}
\end{lem}
\begin{lem}\label{Lem:RegularLocal}
For $v=\mathfrak{p}$ as in Setting \ref{Assumption:Local}(3'), denote $m=c(\pi_{\chi,v})-c(\pi_v)$. Then
\begin{equation}
\I_\mathfrak{p}(0,f_\mathfrak{p})\asymp_q \frac{1}{q^{c(\pi_{\chi,\mathfrak{p}})/2-3c(\pi_\mathfrak{p})/4}}
\end{equation}
\begin{equation}\label{Eq5.2:Notsharpbound}
\I_\mathfrak{p}(\xi,f_\mathfrak{p})\begin{cases}
\ll_q \frac{1}{q^{c(\pi_{\chi,\mathfrak{p}})/2-3c(\pi_\mathfrak{p})/4}}, &\text{\ if $v(1-\xi)\leq 0$},\\
\ll_q \frac{1}{q^{(d+m)/2} }, &\text{\ if $0<v(1-\xi)=d\leq m+O(1)$}, \\
=0, &\text{\ otherwise.}
\end{cases}
\end{equation}
\end{lem}
Here $O(1)$ denotes some absolutely bounded number which may depend on the ramification of the fields.
We are not very careful in the first bound in \eqref{Eq5.2:Notsharpbound}. It is possible to get better bound, but  not necessary as we shall see in Section \ref{Sec5.2.2}.
 We show first how to get the hybrid subconvexity bounds using these lemmas, and then prove them in Section \ref{Section:LocalLemOrbitInt}.

\subsubsection{The proof of Theorem \ref{Theo:HybridsubDisjointram}}
	For simplicity we omit the $\epsilon$ terms from our calculations.
	
	By Waldspurger's formula \eqref{Sec5.1Eq:GlobalWalds}, Lemma \ref{Lem5.1:constantmultiple},  we have for $f$ as in Setting \ref{Assumption:Localdisjoint} and $\varphi$ as in Lemma \ref{Lem5.1:constantmultiple},
	\begin{equation}
		\int_{t\in [\F^\times\backslash\E^\times]}R(f)\varphi(t)\chi^{-1}(t)dt\overline{  \int_{t\in [\F^\times\backslash\E^\times]}\varphi(t)\chi^{-1}(t)dt }=C_{\B,f}\frac{\zeta(2)L(\pi\times \pi_{\chi^{-1}},1/2)}{2L(\pi,Ad,1)}\prod\limits_{v}I^0_v(\varphi,\chi),
	\end{equation}
	where $I^0_v(\varphi,\chi)$ is as described in Lemma \ref{Lem5.1:LocalWalds}.
Note that each term on the right-hand side is non-negative. At $\mathfrak{p}$, $l=c(\pi_\mathfrak{p})$ and $I^0_v(\varphi,\chi)\asymp_q\frac{1}{q^{c(\pi_\mathfrak{p})/4}}$ in this case.

By Theorem \ref{Theo:RelativeTrace}, dropping all but one term on the spectral side, and using that $L(\pi,Ad,1)\ll_\epsilon C(\pi)^\epsilon$, we get that

\begin{equation}\label{Eq:relativespectrumineq0}
L(\pi\times \pi_{\chi^{-1}},1/2) \frac{1}{q^{c(\pi_{\mathfrak{p}})/4}}\frac{1}{\sqrt{C(\pi_{\chi^{-1}})}}\ll \I(f,\chi).
\end{equation}
Then we apply Lemma \ref{lem:localorbitint1} and \ref{Lem:RegularLocal0} to control the orbit integrals on the geometric side. In particular $$\I(0,f)\ll q^{c(\pi_\mathfrak{p})/4}\frac{1}{\sqrt{C(\pi_{\chi^{-1}})}},$$
whereas $\I(\infty,f)=0$ with the assumption that $m_v\neq 0$ at the archimedean places.
For $v|\mathfrak{c}(\chi)$, fix integers $0\leq d_v\leq v(\mathfrak{d}_{\E/\F}\mathfrak{c}(\chi))$, and $d_\mathfrak{p}>0$ such that $q_\mathfrak{p}^{d_\mathfrak{p}}\ll \prod q_v^{d_v} $. Then for the set $$S_{d_v,d_\mathfrak{p}}= \{\eta=\frac{1}{1-\xi}| \I(\xi,f)\neq 0, v(1-\xi)=d_v\text{\ for $v|\mathfrak{c}(\chi)$}, v_\mathfrak{p}(\xi)=d_\mathfrak{p} \},$$ we claim that
\begin{equation}\label{Eq:numRegular0}
\sharp S_{d_v,d_\mathfrak{p}}\ll \frac{\prod q_v^{d_v}}{q_\mathfrak{p}^{d_\mathfrak{p}}}.
\end{equation}
Note that the restriction $d_v\leq v(\mathfrak{d}_{\E/\F}\mathfrak{c}(\chi))$ is natural by the vanishing result in Lemma \ref{lem:localorbitint1}, and $v(1-\xi)=d_v$ implies that $v(\eta)\geq -d_v$. By the same Lemma, for any $v\nmid \mathfrak{c}(\chi)\mathfrak{p}\infty$, $\xi\in S_{d_v,d_\mathfrak{p}}$ if and only if $v(\eta)\geq -v(\mathfrak{d}_{\E/\F})$. 
For $v=\mathfrak{p}$, $v_\mathfrak{p}(\xi)=d_\mathfrak{p}$ implies that $\eta\equiv 1\mod \mathfrak{p}^{d_\mathfrak{p}}$, and in particular $v(\eta)\geq 0$.
Then the set of $\eta$ in $S_{d_v,d_\mathfrak{p}}$  lies in a lattice of covolume$\asymp \frac{1}{\prod q_v^{d_v}}$, and has to further satisfy a congruence condition $\mod \mathfrak{p}^{d_\mathfrak{p}}$.
 On the other hand by Setting \ref{Assumption:Local} (2), $\iota_v(\xi)<0$ for any $v|\infty$ and associated real embedding $\iota_v$. This implies that $\iota_v(\frac{\eta-1}{\eta})<0$, and in particular $0<\iota_v(\eta)<1$. Then \eqref{Eq:numRegular0} follows from the basic algebraic number theory, counting lattice points $\{\eta\}$ in a given box with a particular congruence condition.
\yh{By product formula and basic inequality, the lattice $\mathfrak{d}_{\E/\F}^{-1}\mathfrak{p}^{-d}$ has minimal distance $\asymp q^{-d/N}$ where $N$ is the degree of field extension. Then the volume of each disjoint balls is $\asymp q^{-d}$. The number of such balls in a fixed box is $\asymp q^d$.}

Now for every $\xi$ giving an element in $S_{d_v,d_\mathfrak{p}}$, we apply the bounds in \eqref{Eq:regularorbitsinf}, \eqref{Eq:regularorbitsunramifed} and Lemma \ref{Lem:RegularLocal0}, and get that up to $C(\pi\times\pi_{\chi^{-1}})^\epsilon$ terms,
\begin{align}\label{Eq5.2:GeoUpper1}
\I(f,\chi)&=\I(0,f)+\sum\limits_{\xi}\I(\xi,f)\ll_q q^{c(\pi_\mathfrak{p})/4}\frac{1}{\sqrt{C(\pi_{\chi^{-1}})}}+\sum\limits_{d_v,d_\mathfrak{p}}\sharp S_{d_v,d_\mathfrak{p}}\prod\frac{1}{\sqrt{C(\chi_v)q_v^{d_v}}}\min\{q_\mathfrak{p}^{d_\mathfrak{p}/2},q_\mathfrak{p}^{c(\pi_\mathfrak{p})/4}\} \\
&\ll q^{c(\pi_\mathfrak{p})/4}\frac{1}{\sqrt{C(\pi_{\chi^{-1}})}}+1.\notag
\end{align}
Note that in this case $C(\pi\times\pi_{\chi^{-1}})=(C(\pi)C(\pi_{\chi^{-1}}))^2$, and the convexity bound expects $C(\pi_{\chi^{-1}})^{\frac{1}{2}(1+\delta)}$.
Thus when we write $C(\pi)=C(\pi_{\chi^{-1}})^\delta$, we have by \eqref{Eq:relativespectrumineq0} \eqref{Eq5.2:GeoUpper1},
\begin{equation}
L(\Pi\otimes \chi^{-1},\frac{1}{2})\ll_q q^{c(\pi_\mathfrak{p})/2}+q^{c(\pi_\mathfrak{p})/4}\sqrt{C(\pi_{\chi^{-1}})}=C(\pi)^{1/2}+ C(\pi)^{1/4}C(\pi_{\chi^{-1}})^{1/2}\asymp C(\pi_{\chi^{-1}})^{\frac{1}{2}\max\{\delta, \frac{\delta}{2}+1\}}.
\end{equation}
 Thus we obtain a subconvexity bound for $0<\delta<\infty$, which can be as strong as Weyl bound when $\delta=2$.

\subsubsection{The proof of Theorem \ref{Theo:Hybridsubconvexity}}\label{Sec5.2.2}
The proof for Theorem \ref{Theo:Hybridsubconvexity} is very similar to that for Theorem \ref{Theo:HybridsubDisjointram}, so we shall mainly focus on the differences.

In this case, $l=2c(\pi_{\chi^{-1},\mathfrak{p}})-c(\pi_\mathfrak{p})$. By Theorem \ref{Theo:Main1},
	\begin{equation}\label{Eq5.2:localWaldsjointram}
	I_\mathfrak{p}(\varphi,\chi)\asymp_q  \frac{1}{q^{c(\pi_{\chi,\mathfrak{p}})/2-c(\pi_{\mathfrak{p}})/4}}.
	\end{equation}
	By Waldspurger's formula \eqref{Sec5.1Eq:GlobalWalds}, Lemma \ref{Lem5.1:constantmultiple}, \eqref{Eq5.2:localWaldsjointram}, Theorem \ref{Theo:RelativeTrace}
, and dropping all but one term, we get that
	\begin{equation}\label{Eq:relativespectrumineq}
	L(\pi\times \pi_{\chi^{-1}},1/2) \frac{1}{q^{c(\pi_{\chi,\mathfrak{p}})/2-c(\pi_{\mathfrak{p}})/4}}\ll \I(f,\chi).
	\end{equation}
	Then we apply Lemma \ref{lem:localorbitint1} and \ref{Lem:RegularLocal} to control the orbit integrals on the geometric side. In particular $$\I(0,f)\ll \frac{1}{q^{c(\pi_{\chi,\mathfrak{p}})/2-3c(\pi_\mathfrak{p})/4}},$$
	whereas $\I(\infty,f)=0$.
	For the general orbital integral $\I(\xi,f)$ with $\xi\neq 0,\infty$, let $d=d_\mathfrak{p}\leq m+O(1)$,
	where $m=c(\pi_{\chi,v})-c(\pi_v)$ as in Lemma \ref{Lem:RegularLocal}. Denote
	\begin{equation}\label{Eq:numRegular}
	S_d= \{\eta=\frac{1}{1-\xi}| \I(\xi,f)\neq 0, v_\mathfrak{p}(1-\xi)=d \}. 
	\end{equation}
	Then $\sharp S_d\ll q^d$.
	This is because by Lemma \ref{lem:localorbitint1} and that $v_\mathfrak{p}(1-\xi)=d$, $\I(\xi,f)$ is nonvanishing only when $\eta$ belongs to $\mathfrak{d}_{\E/\F}^{-1}\mathfrak{p}^{-d}$, which is a lattice of covolume$\asymp \frac{1}{q^d}$. On the other hand  $0<\iota_v(\eta)<1$ for any real embedding $v$. Then the bound for $S_d$ again follows from the  lattice point counting. From this argument, we shall also care about $d\geq -O(1)$ only.
	\yh{By product formula and basic inequality, the lattice $\mathfrak{d}_{\E/\F}^{-1}\mathfrak{p}^{-d}$ has minimal distance $\asymp q^{-d/N}$ where $N$ is the degree of field extension. Then the volume of each disjoint balls is $\asymp q^{-d}$. The number of such balls in a fixed box is $\asymp q^d$.}
	
	Now for every $\xi\in S_d$, we apply the bounds in \eqref{Eq:regularorbitsinf}, \eqref{Eq:regularorbitsunramifed} and Lemma \ref{Lem:RegularLocal}, and get that up to $q^\epsilon$ terms,
	\begin{equation}
	\I(f,\chi)=\I(0,f)+\sum\limits_{\xi}\I(\xi,f)\ll_q \frac{1}{q^{c(\pi_{\chi,\mathfrak{p}})/2-3c(\pi_\mathfrak{p})/4}}+\sum\limits_{-O(1)\leq d\leq m+O(1)} \sharp S_d\frac{1}{q^{(d+m)/2}}\ll  \frac{1}{q^{c(\pi_{\chi,\mathfrak{p}})/2-3c(\pi_\mathfrak{p})/4}}+1.
	\end{equation}
	Note that $-O(1)\leq d\leq 0$ terms can be incorporated into the main term by Lemma \ref{Lem:RegularLocal}.
	Then by \eqref{Eq:relativespectrumineq},
	\begin{equation}
	L(\Pi\otimes \chi^{-1},\frac{1}{2})\ll_q q^{\frac{c(\pi_\mathfrak{p})}{2}}+q^{\frac{c(\pi_{\chi,\mathfrak{p}})}{2}-\frac{c(\pi_\mathfrak{p})}{4}}\asymp C(\pi_\chi)^{\frac{1}{2}(\max\{\delta, 1-\frac{\delta}{2}\})}. 
	\end{equation}

\subsection{Proof of Lemma \ref{Lem:RegularLocal0} and \ref{Lem:RegularLocal}}\label{Section:LocalLemOrbitInt}
Now we go back to the local setting, and omit subscript $v$ for all the notations.
We need  a possibly different element $j'$ from that in Lemma \ref{Lem2.1:Eperpbasic}.
\begin{lem}\label{Lem:commonJ}
	For $\E$ and $\L$ aligned as in Proposition \ref{Sec4.4Prop:differenceinalpha}, there exists $j'\in \E^\perp \cap \L^\perp$ such that $v(j'^2)=0$ or $1$, and $v(j'^2)\equiv e_\E-1-\frac{c(\pi\times\pi_{\chi^{-1}})}{2}\mod(2)$.
\end{lem}
\begin{proof}
	Indeed for $\alpha$ and $\beta$ as in Proposition \ref{Sec4.4Prop:differenceinalpha},  
	we have a nontrivial element $ [\alpha_{0},\beta]=[\alpha,\beta]\in \E^\perp\cap \L^\perp$, so
	$j'$ should be parallel to $[\alpha,\beta]$. 
	Then
	\begin{equation}
	v(j'^2)\equiv v(\Nm([\alpha,\beta]))=v(\Nm(\beta))+v(\Nm(\alpha^\perp))\equiv e_\E-1-\frac{c(\pi\times\pi_{\chi^{-1}})}{2}\mod(2).
	\end{equation}
\end{proof}

Recall from Proposition \ref{Sec4.4Prop:differenceinalpha} that for $\beta\in \E^\times$ imaginary, we write
\begin{equation}
\beta=\beta_\L+\beta^\perp,
\end{equation}
where $\beta_\L\in \L$, $\beta^\perp\in \L^\perp$, $v(\Nm(\beta))=e_\E-1$ and $v(\Nm(\beta^\perp))=v(\Nm(\beta))+\frac{c(\pi)-l}{2}$. 

From now on let $j$ be $j'$ as in Lemma \ref{Lem:commonJ}.

\begin{cor}\label{Sec5.3Corjbeta}
	For notations as above,
	$
	j\beta=j\beta^\perp+j\beta_\L
	$
	with $j\beta^\perp\in \L$ and $j\beta_\L\in \L^\perp$. Furthermore $j\beta^\perp \in \L^0$, the set of trace 0 elements in $\L$.
\end{cor}
\begin{proof}
	The decomposition for $j\beta$ is clear as $j\in \L^\perp$. To see that $j\beta^\perp \in \L^0$, it suffices to show that both $j\beta_\L$ and $j\beta$ have trace 0. Note that $\Tr(\beta)=\Tr(\beta_\L)=0$, so
	$\overline{j\beta_\L}=\overline{\beta_\L}\overline{j}=\beta_\L j=-j\beta_\L$, where the last equality follows from that $j\in \L^\perp$. On the other hand,
	$\overline{j\beta}=-\overline{\beta}j=-j\beta$, where the last equality follows from that $j\in \E^\perp$.
\end{proof}

\subsubsection{Proof of Lemma \ref{Lem:RegularLocal0} }
For $v=\mathfrak{p}$ as in Setting \ref{Assumption:Localdisjoint}(3), 
 we have $c(\pi\times \pi_{\chi^{-1}})=2c(\pi)$, so $l=c(\pi)$.
 
We can take $\alpha_\chi=0$ without loss of generality. 
Thus $\alpha\in \E^\perp$ for the test vector in Theorem \ref{Theo:Liealgformulation}. Then for an imaginary element $\beta \in \E$, we  have $\beta\in \L^\perp$. For $j$ as in Lemma \ref{Lem:commonJ}, it is parallel to $[\alpha,\beta]$ with $v(j^2)\equiv e_\E-1-c(\pi) \mod (2)$. We also have $j\beta\in \L^0$.

There are many possible combinations of $v(\Nm(\beta))$ and $v(j^2)$ in this case. So we shall not keep close track of them and simply control their contributions by $O(1)$.

Now we start the proof for Lemma \ref{Lem:RegularLocal0}.
First of all for $\I(0,f)$, the integral is essentially the local Waldspurger's period integral as $f$ is the restriction of the matrix coefficient of a minimal vector to $ZB^1$. Also recall from Theorem \ref{Theo:Main1}, a minimal is a test vector iff the integrand is constant on the common support, and in particular on $ZB^1\cap \E^\times$. Thus considering the normalisation factor, we have
\begin{equation}
\I(0,f)\asymp_q q^{c(\pi)/2}\frac{1}{q^{l/4}}=q^{c(\pi)/4}.
\end{equation}

 As in \cite{feigon_averages_2009}, we have
\begin{align}
\I(\xi,f)&=\int_{\F^\times\backslash \E^\times}\int_{\F^\times\backslash \E^\times}f(t_1(1+jx) t_2)\chi^{-1}(t_1t_2) dt_1dt_2\\
&=\int_{\F^\times\backslash \E^\times}\int_{\F^\times\backslash \E^\times}f(t_1t_2(1+jxt_2/\overline{t_2}))\chi^{-1}(t_1t_2) dt_1dt_2\notag\\
&=\int_{\F^\times\backslash \E^\times}\int_{\E^1}f(e(1+jxe'))\chi^{-1}(e) de'de,\notag
\end{align}
where $e=t_1t_2$ and $e'=t_2/\overline{t_2}$.
Note that the local integral only depends on $\xi=-\Nm(j x)$, by a change of variable in $e'$ if necessary. In particular we can assume that $j$ is as in Lemma \ref{Lem:commonJ}.


	We write 
	\begin{equation}\label{Eq5.3:ee'}
	e=a_1+b_1\beta, \text{\ }xe'=a_2+b_2\beta
	\end{equation} for $a_i,b_i\in \F^\times$. 
	For conciseness we shall also write \begin{equation}\label{Eq5.3:e30}
	\overline{e} xe'=a_3+b_3\beta.
	\end{equation}
	
	Then \begin{equation}\label{Eq:XInorm}
	\xi=-\Nm(jx)=j^2(a_2^2+b_2^2\Nm(\beta)).
	\end{equation}
	\begin{equation}\label{Sec5.3Eq:a3b3}
	a_3=a_1a_2+b_1b_2\Nm(\beta), \text{\ }b_3=(a_1b_2-a_2b_1).
	\end{equation}
	We can rewrite $e(1+jxe')$ according to the decomposition $\L\oplus \L^\perp$ as follows.
	\begin{align}\label{Eq:DecompwrtL0}
	e(1+jxe')&=e+j\overline{e}xe'=a_1+b_1\beta+j(a_3+b_3\beta )\notag\\
	&=[a_1+b_3j\beta]+[b_1\beta+ja_3].\notag
	\end{align}
	Multiplying by a constant if necessary, we may assume that $g=e(1+jxe')\in O_\F^\times B^1$. 
	Since $B^1\subset J^1$, we loose the condition a little and get the necessary conditions  as follows.
	\begin{equation}\label{Eq:doubleorbitssupport10}
	\vv(a_1)= 0.
	\end{equation}
	\begin{equation}\label{Eq:doubleorbitssupport30}
	\vv(b_3 j\beta)\geq 1,
	\end{equation}
	\begin{equation}\label{Eq:doubleorbitssupport20}
	\vv(b_1\beta)\geq i,
	\end{equation}
	\begin{equation}\label{Eq:doubleorbitssupport40}
	\vv(a_3 j)\geq i
	\end{equation}
	where $i $ is as in Definition \ref{Defn:ii'&JHs}. Here we have used Corollary \ref{Sec2.1Cor:vofimaginary} and Definition \ref{Defn3.1:semivaluation} to obtain \eqref{Eq:doubleorbitssupport20} and \eqref{Eq:doubleorbitssupport40}.
	Note that the possible error from using $J^1$ can be absolutely controlled by a fixed power of $q$.
	
		The discussion so far implies that, by Lemma \ref{Sec2.3Lem:VolumeofE},
	\begin{equation}\label{Sec5.3Eq:VolE*0}
	\Vol(\{e=a_1+b_1\beta\text{\   satisfying  \eqref{Eq:doubleorbitssupport10},\eqref{Eq:doubleorbitssupport20} }\}) \asymp_q \frac{1}{q^{i/e_\L}}\asymp_q \frac{1}{q^{c(\pi)/4}}.
	\end{equation}
	
	Furthermore by \eqref{Eq:doubleorbitssupport10}, \eqref{Eq:doubleorbitssupport20}, we get $v(\Nm(e))=v(\Nm(a_1+b_1\beta))=0$.
	By  \eqref{Eq:doubleorbitssupport30}, \eqref{Eq:doubleorbitssupport40} and \eqref{Eq5.3:e30}, we have $v(\Nm(j\overline{e}xe'))>0$. Then by \eqref{Eq5.3:ee'}  and \eqref{Eq:XInorm},  we get that
	\begin{equation}\label{Eq5.2:vxi>0}
	v(\xi)=v(-\Nm(jx))=v(\Nm(j(a_3+b_3\beta)))-v(\Nm(e))>0
	\end{equation} 
	 if there ever exists $g=e(1+jxe')\in O_\F^\times B^1$. 
	In that case, we have
	\begin{equation}\label{Eq5.3:vxi>0implies}
	v(a_2^2+b_2^2\Nm(\beta))=v(\xi)-v(j^2)=v(\xi)+O(1).
	\end{equation}

 Substituting \eqref{Sec5.3Eq:a3b3} into
  \eqref{Eq:doubleorbitssupport40}, we obtain that
	\begin{equation}\label{Sec5.3Eq:CongMid20}
	v(a_1a_2+b_1b_2\Nm(\beta))\geq \frac{i-\vv(j)}{e_\L}=\frac{i}{e_\L}+O(1).
	\end{equation}
	Combining with \eqref{Eq:doubleorbitssupport20}, we obtain that
	\begin{equation}\label{Eq5.3:va20}
	v(a_2)\geq 
	\frac{i}{e_\L}+O(1).
	\end{equation}

	Now if $v(\xi)/2\leq i/e_\L+O(1)=c(\pi)/4+O(1),$ we get 
	\begin{equation}\label{Eq5.3:vb20}
	v(b_2)=\frac{v(\xi)}{2}+O(1)
	\end{equation}
	By Lemma \ref{Sec2.3Lem:VolumeE1},
	\begin{equation}\label{Eq5.2:Vole'case1}
	\Vol\{e'\in \E^1|e' \text{\ maintains \eqref{Eq5.3:va20} and \eqref{Eq5.3:vb20}}\}\ll_q \frac{1}{q^{i/e_\L-v(\xi)/2}}\asymp_q \frac{1}{q^{c(\pi)/4-v(\xi)/2}}.
	\end{equation}
	
	On the other hand if $v(\xi)/2\geq i/e_\L+O(1)$,  
	we simply use that
	\begin{equation}\label{Eq5.2:Vole'case2}
	\Vol \{e'\in \E^1\}\leq 1.
	\end{equation}
	
	Then \eqref{Eq5.2:localintatpSetting1regular} follows from \eqref{Eq5.2:vxi>0}, \eqref{Sec5.3Eq:VolE*0}, \eqref{Eq5.2:Vole'case1}/\eqref{Eq5.2:Vole'case2} and the normalization factor for $f_v$.

\subsubsection{Proof of Lemma \ref{Lem:RegularLocal} }
Assume from now on that $c(\pi_\chi)>c(\pi)$. Recall that $m=c(\pi_{\chi})-c(\pi)$.
In this case $\beta_\L\in \L^0$ is non-trivial. So by Corollary \ref{Sec5.3Corjbeta} there exists $u\in \F^\times$ such that \begin{equation}\label{Eq5.3:scalaru}
j\beta^\perp=u\beta_\L.
\end{equation}

\begin{cor}\label{Sec5.3Cor:Specialsetting}
	When $c(\pi_\chi)>c(\pi)$, we have
\begin{equation}\label{Sec5.3Eq:vNmbetaperp}
v(\Nm(\beta^\perp))=v(\Nm(\beta_\L)=v(\Nm(\beta))+c(\pi)-c(\pi_\chi)=v(\Nm(\beta))-m.
\end{equation}
For $j$ as in Lemma \ref{Lem:commonJ}, 
\begin{equation}\label{Sec5.3Eq:vjsquare}
v(j^2)\equiv e_\E-1-\frac{2c(\pi_{\chi^{-1}})}{2}\equiv 0\mod(2).
\end{equation}
So we can in particular pick from now on $v(j^2)=0$.
As a result, for $u$ as in \eqref{Eq5.3:scalaru},
\begin{equation}\label{Sec5.3Eq:vu}
v(u)=0,\text{\ } v(1-\frac{u^2}{j^2})=m.
\end{equation}
\end{cor}
\begin{proof}
When $c(\pi_\chi)>c(\pi)$, we have that $c(\pi\times \pi_{\chi^{-1}})=2c(\pi_{\chi})$. \eqref{Sec5.3Eq:vNmbetaperp} and \eqref{Sec5.3Eq:vjsquare} follow easily from Proposition \ref{Sec4.4Prop:differenceinalpha} and Lemma \ref{Lem:commonJ}. \eqref{Sec5.3Eq:vu} then follows from \eqref{Eq5.3:scalaru}, \eqref{Sec5.3Eq:vNmbetaperp} and that $\Nm(\beta)=\Nm(\beta_\L)+\Nm(\beta^\perp)$.
\end{proof}

Now we start the proof for Lemma \ref{Lem:RegularLocal}. Some of the arguments are similar to those for Lemma \ref{Lem:RegularLocal0}.

First of all we have $l=2c(\pi_{\chi^{-1}})-c(\pi)$.
For $\I(0,f)$, the integral is again essentially the local Waldspurger's period integral plus the normalisation factor, so
\begin{equation}
\I(0,f)=q^{c(\pi)/2}\Vol(\F^\times\backslash(ZB^1\cap \E^\times))\asymp_q q^{c(\pi)/2}\frac{1}{q^{l/4}}=\frac{1}{q^{c(\pi_\chi)/2-3c(\pi)/4}}.
\end{equation}

When $v(1-\xi)\leq 0$, we simply use the following control
\begin{equation}
|\I(\xi,f)|\leq \int_{\F^\times\backslash \E^\times}|f(t_1(1+jx) t_2)| dt_1.
\end{equation}
Note that if both $t_1(1+jx) t_2$ and $t_1'(1+jx) t_2$ are in the support of $f$, then $t_1'^{-1}t_1\in \F^\times B^1$ by the multiplicative nature of $f$. Thus
\begin{equation}
|\I(\xi,f)| \ll_q q^{c(\pi)/2}\Vol(\F^\times\backslash(ZB^1\cap \E^\times))\asymp_q \frac{1}{q^{c(\pi_\chi)/2-3c(\pi)/4}}.
\end{equation}

We focus on the case $\I(\xi,f)$ when $v(1-\xi)>0$ from now on. Again we have
\begin{align}
 \I(\xi,f)=\int_{\F^\times\backslash \E^\times}\int_{\E^1}f(e(1+jxe'))\chi^{-1}(e) dede'.
\end{align}

The analysis for the volume of $e$ and $e'$ is more complicated in this case. So we formulate the following lemma.
Lemma \ref{Lem:RegularLocal}  will then follow directly from this lemma and the normalisation factor for $f$.
\begin{lem}
Suppose $v(1-\xi)=d>0$ where $\xi=-\Nm(jx)$.
Then the set $S=\{(e,e')\in  \F^\times\backslash \E^\times\times \E^1, e(1+jxe') \in \Supp f \}$ is non-empty only if $d\leq m+O(1)$. In that case, it has volume
$\ll_q\frac{1}{q^{(c(\pi)+d+m)/2}}.$
\end{lem}

\begin{proof}
	As in the proof for  Lemma \ref{Lem:RegularLocal0}
we write $e=a_1+b_1\beta$ , $xe'=a_2+b_2\beta$ 
and $\overline{e} xe'=a_3+b_3\beta$ for $a_i,b_i\in \F^\times$. \eqref{Eq:XInorm} and \eqref{Sec5.3Eq:a3b3} remain true.
By $\beta=\beta_\L+\beta^\perp$ and Corollary \ref{Sec5.3Corjbeta}, we can rewrite $e(1+jxe')$ according to the decomposition $\L\oplus \L^\perp$ as following:
\begin{align}\label{Eq:DecompwrtL}
 e(1+jxe')&=e+j\overline{e}xe'=a_1+b_1\beta+j(a_3+b_3\beta )\notag\\
 &=[a_1+b_1\beta_\L+jb_3\beta^\perp]+[b_1\beta^\perp+ja_3+jb_3\beta_\L]\notag\\
 &=[a_1+(b_1+ub_3)\beta_\L]+[ja_3+j\beta_\L(\frac{u}{j^2}b_1+b_3)]\notag
\end{align}
where the last equality follows from that $j\beta^\perp=u\beta_\L$.
Multiplying by a constant if necessary, we may assume that $g=e(1+jxe')\in O_\F^\times B^1\subset O_\F^\times J^1$. 
Then
\begin{equation}\label{Eq:doubleorbitssupport1}
\vv(a_1)= 0,
\end{equation}
\begin{equation}\label{Eq:doubleorbitssupport3}
\vv((b_1+ub_3)\beta_\L)\geq 1,
\end{equation}
\begin{equation}\label{Eq:doubleorbitssupport2}
\vv(ja_3)\geq i,
\end{equation}
\begin{equation}\label{Eq:doubleorbitssupport4}
 \vv(j\beta_\L(\frac{u}{j^2}b_1+b_3))\geq i
\end{equation}
where $i $ is as in Definition \ref{Defn:ii'&JHs}. Here we have used Corollary \ref{Sec2.1Cor:vofimaginary} and Definition \ref{Defn3.1:semivaluation} to obtain \eqref{Eq:doubleorbitssupport2} and \eqref{Eq:doubleorbitssupport4}.
Note that the possible error from using $J^1$ can be absolutely controlled by a fixed power of $q$.

Furthermore, by checking the determinant, we get that $v(\det(g))=v(\Nm(e))(1-\xi)=0$.
Since $v(1-\xi)=d>0$ by the assumption, we get $v(\Nm(e))=-d<0$. From \eqref{Eq:doubleorbitssupport1}, we get 
\begin{equation}\label{Sec5.3eq:va1b1}
\text{$v(a_1)=0$ and $v(b_1)=-\frac{d+v(\Nm(\beta))}{2}.$ }
\end{equation}
In particular we need $d\equiv v(\Nm(\beta)=e_\E-1 \mod 2$ for the set $S$ to be non-empty.

From \eqref{Eq:doubleorbitssupport3} and \eqref{Eq:doubleorbitssupport4} we obtain that 
\begin{equation}\label{Eq:Trivialcong}
 \vv(jb_1\beta_\L(1-\frac{u^2}{j^2}))\geq 1.
 \end{equation}
By \eqref{Sec5.3eq:va1b1} and Corollary \ref{Sec5.3Cor:Specialsetting}, \eqref{Eq:Trivialcong} is equivalent to that
\begin{equation}\label{Sec5.3Eq:MainPfEq7}
d\leq m-\frac{2}{e_\L}.
\end{equation}
This proves the first part of the lemma.
Note that conversely, \eqref{Eq:doubleorbitssupport3} is automatic when  \eqref{Eq:doubleorbitssupport4} and  \eqref{Sec5.3Eq:MainPfEq7} hold. 

Assume \eqref{Sec5.3Eq:MainPfEq7}  from now on. Substituting \eqref{Sec5.3Eq:a3b3} into \eqref{Eq:doubleorbitssupport2} and \eqref{Eq:doubleorbitssupport4}, we obtain that
\begin{equation}\label{Sec5.3Eq:CongMid1}
v(a_1a_2+b_1b_2\Nm(\beta))\geq \frac{i}{e_\L},
\end{equation}
\begin{equation}\label{Sec5.3Eq:CongMid2}
v(\frac{u}{j^2}b_1+a_1b_2-a_2b_1)\geq \frac{i}{e_\L}+\frac{m-v(\Nm(\beta))}{2}.
\end{equation}
Here we have used Lemma \ref{Lem3.1semivaluation}(2) and Corollary \ref{Sec5.3Cor:Specialsetting}.
On the other hand, $v(1-\xi)=d>0$ implies that $v(\xi)=0$. 
By \ref{Eq:XInorm} and $v(j^2)=0$, we get that $\min\{v(a_2^2),v(\Nm(b_2\beta))\}=0$. Then to obtain the congruence in \eqref{Sec5.3Eq:CongMid1}, we must have 
\begin{equation}\label{Sec5.3Eq:va2b2}
v(a_2)=0, v(b_2)=-v(b_1).
\end{equation}
For conciseness, denote 
\begin{equation}
n_1=\lceil \frac{i}{e_\L}\rceil-v(b_1)=\lceil \frac{i}{e_\L}\rceil+ \frac{d+v(\Nm(\beta))}{2},\text{\ } n_2=\lceil \frac{i}{e_\L}+\frac{m-v(\Nm(\beta))}{2}\rceil.
\end{equation}
Note that \eqref{Sec5.3Eq:MainPfEq7} implies that $n_1<n_2$.
Now the explicit valuations in \eqref{Sec5.3eq:va1b1} and \eqref{Sec5.3Eq:va2b2} allow us to rewrite \eqref{Sec5.3Eq:CongMid1} and \eqref{Sec5.3Eq:CongMid2} as follows:
\begin{equation}\label{Sec5.3Eq:FinalEq1}
\frac{a_1}{b_1}\equiv -\frac{b_2\Nm(\beta)}{a_2}\mod \varpi^{n_1}O_\F
\end{equation}
\begin{equation}\label{Sec5.3Eq:FinalEq2}
\frac{a_1}{b_1}\equiv \frac{a_2-\frac{u}{j^2}}{b_2}\mod\varpi^{n_2}O_\F
\end{equation}
The system of equations \eqref{Sec5.3Eq:FinalEq1} \eqref{Sec5.3Eq:FinalEq2} is equivalent to the system of equations  \eqref{Sec5.3Eq:FinalEq2} with the following.
\begin{equation}\label{Eq:Congfore'}
-\frac{b_2\Nm(\beta)}{a_2}\equiv \frac{a_2-\frac{u}{j^2}}{b_2} \mod \varpi^{n_1}.
\end{equation}

For each fixed $(a_2,b_2)$ satisfying \eqref{Eq:Congfore'}, Lemma \ref{Sec2.3Lem:VolumeofE} implies that
\begin{equation}\label{Sec5.3Eq:VolE*}
\Vol(\{e=a_1+b_1\beta\text{\   satisfying \eqref{Sec5.3Eq:FinalEq2} }\}) \asymp_q \frac{1}{q^{n_2}}.
\end{equation}

On the other hand, recall that $-b_2^2\Nm(\beta)=a_2^2-\frac{\xi}{j^2}$ from \eqref{Eq:XInorm}. So \eqref{Eq:Congfore'} is equivalent to that
$$
 a_2^2-\frac{\xi}{j^2}\equiv a_2^2-\frac{u}{j^2}a_2\mod \varpi^{n_1+v(b_2)}O_\F,
$$
or simply \begin{equation}\label{Sec5.3Eq:Solvea2}
a_2\equiv \frac{\xi}{u} \mod \varpi^{n_1+v(b_2)}O_\F.
\end{equation}
For all possible $a_2$ satisfying \eqref{Sec5.3Eq:Solvea2}, we have
\begin{equation}
\sharp\{b_2\mod\varpi^{n_1}O_\F| -b_2^2\Nm(\beta)=a_2^2-\frac{\xi}{j^2} \text{\ for some $a_2$ satisfying \eqref{Sec5.3Eq:Solvea2}}\}\ll_q 1.
\end{equation}
The subset of $e'\in \E^1$ must maintain the set of congruence classes of $(a_2,b_2)\mod\varpi^{n_1}O_\F$. By Lemma \ref{Sec2.3Lem:VolumeE1}, this subset of $\E^1$ has volume $\asymp_q \frac{1}{q^{n_1}}$.
Combining with \eqref{Sec5.3Eq:VolE*}, we get that
\begin{equation}
\Vol(S)\asymp_q \frac{1}{q^{n_1+n_2}}.
\end{equation}
Lastly by \eqref{Eq:Assmpfori}, we have $n_1+n_2=\frac{c(\pi)+d+m}{2}+O(1)$.  This concludes the proof of the second part of the lemma.

\end{proof}

\appendix
\section{More preparations}\label{Sec:AppendixA}
The goal in the appendix is to prove a variant of Theorem \ref{Theo:Main1} using relatively elementary method in a more restrictive setting but only requiring $p\neq 2$. In particular we shall prove the following theorem.
\begin{theo}\label{Appendix:Maintheo}
We assume that $p\neq 2$, $w_\pi=\chi|_{\F^\times}=1$, $\E$ not split, $c(\pi)\geq c(\pi_{\chi^{-1}})$ and $\B$ is the matrix algebra.
Then there exists a nontrivial test vector for local Waldspurger's period integral, if and only if there exists a minimal vector $\varphi\in \pi$ such that
\begin{equation}
\Phi_{\varphi}\chi^{-1}=1 \text{\ on } \Supp \Phi_{\varphi} \cap \E^\times .
\end{equation}
Moreover if $l$ is an integer such that $C(\pi\times\pi_{\chi^{-1}})=q^{c(\pi)+l}$, then
\begin{equation}
\max_{\text{minimal vector } \varphi'\in \pi }\{I(\varphi',\chi)\}\asymp I(\varphi,\chi)=\Vol(\Supp \Phi_{\varphi} \cap \E^\times )\asymp \frac{1}{q^{l/4}}.
\end{equation}
\end{theo}
Note that the situation we consider here is exactly where the local period integral might always be vanishing by Tunnell-Saito's epsilon value test. 

In Appendix \ref{Sec:AppendixA}, we set up explicit coordinates and make other preparations. The proof of this theorem will be given in Appendix \ref{Sec:AppendixB}.
\subsection{Setting up}\label{Sec:Specialembedding}

From now on we fix the standard embedding of $\L$ and $\E$ into $\GL_2$ such that
$$\E=\{\zxz{a}{b}{bD}{a}\}, \text{\ \ }\L=\{\zxz{a}{b}{bD'}{a}\},$$ with $0\leq v(D),v(D')\leq 1$ depending on whether the quadratic extension is inert or ramified. Any different embeddings differ by a conjugation from the standard embedding, and can be reduced to a single translate for test vectors when computing the local period integral. 

Further if  $\theta |_{\F^\times }=1$, then $\alpha_\theta$ can be chosen to be imaginary, that is
\begin{equation}
\overline{\alpha_\theta}=-\alpha_\theta.
\end{equation}

For simplicity we choose $$D'=\frac{1}{\alpha_\theta^2\varpi_\L^{2c(\theta)}},$$
identify $\frac{1}{\alpha_\theta\varpi_\L^{c(\theta)}}$ with $\sqrt{D'}$ and $\L$ with $\F(\sqrt{D'})$.

One can check that indeed  $v_\F(D')=0,1$, and under the standard embedding,
\begin{equation}\label{eq2.1:specialembedding}
\alpha_\theta= \frac{1}{\varpi_\L^{c(\theta)}}\frac{1}{\sqrt{D'}} \mapsto \frac{1}{\varpi^{c(\theta)/e_\L}}\zxz{0}{\frac{1}{D'}}{1}{0}.
\end{equation}

We remark here that the particular choice of $D'$ is mainly for convenience, and a different choice will change 
the intertwining operator in \eqref{eq:3.4:IntertwiningtoWhittaker}, and explicit  computations in, for example, Corollary \ref{CorA1:thetatilde2} and Lemma \ref{lem:supportimpliesclosetorus} accordingly. 
\begin{rem}\label{RemA2:specialorder}
As mentioned in Example \ref{Example:explicitideal}, when
$e_\L=1$, we have that $\mathfrak{A}_1=M_2(O_\F)$ and  $\mathcal{B}_1=\varpi M_2(O_\F)$. When $e_\L=2$,
$\mathfrak{A}_2=\zxz{O_\F}{O_\F}{\varpi O_\F}{O_\F}$  and $\mathcal{B}_2=\zxz{\varpi O_\F}{O_\F}{\varpi O_\F}{\varpi O_\F}$. 

\end{rem}

Let $K_\L(n)=\L^\times K_\mathfrak{A}(n)$ 
and similarly $K_\L^1(n)=U_\L(1)K_\mathfrak{A}(n)$. In the case $c(\pi)=4n$ or $2n+1$, $B^1=K_\L^1(n)$. Using that $\GL_2=\L^\times B$ for the Borel subgroup $B$, we can alternatively write
$$K_\L^1(n)=\{tb| , t\in U_\L(1), b=\zxz{a}{m}{0}{1}\text{\  where }a\in U_\F(n), v(m)\geq n\}.   $$

In the case $c(\pi)=4n+2$, we shall pick the following special intermediate subgroup
\begin{equation}\label{eq3-4:AnotherCompactopenSubgroup}
 B^1=K_\L^1(n+1,n)=U_\L(1)K_{\mathfrak{A}_2}(2n+1).
\end{equation}
That is, we are using the other quadratic extension in Remark \ref{RemA2:specialorder} to define the intermediate subgroup.
Now we define a character $\tilde{\theta}$ on $B^1$ similarly as in \eqref{Eq:charextension}:
\begin{equation}\label{EqA1:thetatilde}
\tilde{\theta}(l(1+t))=\theta(l)e^{<\alpha_\theta ,t>}.
\end{equation}
\begin{lem}\label{LemA1:char}
	When $B^1\neq H^1$, \eqref{EqA1:thetatilde} indeed gives a character on $B^1$.
\end{lem}
\begin{proof}
	Let $\L_2$ be the ramified quadratic field extension with associated semi-valuation $v_2$ as in Definition \ref{Defn3.1:semivaluation}, so that $B^1=U_\L(1)K_{\mathfrak{A}_2}(2n+1)=l(1+t)$ with $l\in U_\L(1)$ and $v_2(t)\geq 2n+1$. We observe here that if $v_2(t)\geq 2m$, then $\vv(t)\geq m$. Also when $\vv(t)\geq 1$, $v_2(t)\geq 1$.
	
	First of all, $\tilde{\theta}$ is indeed a character on $K_{\mathfrak{A}_2}(2n+1)$. Indeed 
	$$\tilde{\theta}((1+t_1)(1+t_2))=e^{<\alpha_\theta, t_1+t_2+t_1t_2>}=e^{<\alpha_\theta, t_1+t_2>}=\tilde{\theta}(1+t_1)\tilde{\theta}(1+t_2).$$
	Here in the second equality we have used that $v_2(t_1t_2)\geq 4n+2$, so that $\vv(t_1t_2)\geq 2n+1=-\vv(\alpha_\theta)$.
	
	Secondly, it is easy to check that there is no ambiguity on $U_\L(1)\cap K_{\mathfrak{A}_2}(2n+1)=U_\L(n+1)$.
	
	Lastly, for $l=1+\Delta l=1+\zxz{x}{y}{yD'}{x}$ with $v(x),v(y)\geq 1$, and $1+t\in  K_{\mathfrak{A}_2}(2n+1)$, we need to check that \begin{equation}
	\tilde{\theta}((1+t)l)=\theta(l)e^{<\alpha_\theta,t>}.
	\end{equation} We first claim that $l^{-1}(1+t)l\in K_{\mathfrak{A}_2}(2n+1)$. It is clear that
	$
	l^{-1}(1+t)l=1+\frac{\overline{l}tl}{\Nm(l)}.$
	Then indeed $\Nm(l)\in O_\F^\times$, and $v_2(\overline{l}tl)=v_2(t+\overline{\Delta l}t+t\Delta l+\overline{\Delta l}t\Delta l)\geq 2n+1$. Thus
	$$\tilde{\theta}((1+t)l)=\theta(l) e^{<\alpha_\theta, \frac{\overline{l}tl}{\Nm(l)}>},$$
	while
	$$e^{<\alpha_\theta, \frac{\overline{l}tl}{\Nm(l)}>}=\psi\circ\Tr(\alpha_\theta\frac{\overline{l}tl}{\Nm(l)})=\psi\circ\Tr(\alpha_\theta\frac{tl\overline{l}}{\Nm(l)})=e^{<\alpha_\theta,t>}.$$
	
\end{proof}

\begin{cor}\label{CorA1:thetatilde2}
	Let $g=1+t=\zxz{1+a}{b}{c}{1+d}$ with $v(a),v(d)\geq n+1$ and $v(b), v(c)\geq n\geq 1$. Then $g\in B^1$ and
	$$\tilde{\theta}(g)=e^{<\alpha_\theta,t>}=\psi(\frac{1}{\varpi^{2n+1}}(b+\frac{c}{D'})).$$
\end{cor}
\begin{proof}
	One can directly check that
	$$g=\frac{1}{\Nm(1+a+\frac{c}{D'}\sqrt{D'})}\zxz{1+a}{\frac{c}{D'}}{c}{1+a}\zxz{(1+a)^2-\frac{c^2}{D'}}{b(1+a)-\frac{c(1+d)}{D'}}{0}{(1+a)(1+d)-bc},$$
	where $\zxz{1+a}{\frac{c}{D'}}{c}{1+a}\in \L^\times$ and $\zxz{1+a}{\frac{c}{D'}}{c}{1+a}\zxz{(1+a)^2-\frac{c^2}{D'}}{b(1+a)-\frac{c(1+d)}{D'}}{0}{(1+a)(1+d)-bc}\in B\cap K_{\mathfrak{A}_2}(2n+1)$. Thus by Lemma \ref{LemA1:char} and direct computations, we have
	$$\tilde{\theta}(g)=\theta(1+a+\frac{c}{D'}\sqrt{D'})\psi(\frac{1}{\varpi^{2n+1}}(b(1+a)-\frac{c(1+d)}{D'})).$$
	On one hand, as $v(a),v(d)\geq n+1$ and $v(b), v(c)\geq n$, we have
	$$\psi(\frac{1}{\varpi^{2n+1}}(b(1+a)-\frac{c(1+d)}{D'}))=\psi(\frac{1}{\varpi^{2n+1}}(b-\frac{c}{D'})).$$
	On the other  hand, by  \eqref{eq:alphathetaless}, we have
	\begin{equation}
	\theta(1+a+\frac{c}{D'}\sqrt{D'})=\psi_\L(\alpha_\theta(\frac{c\sqrt{D'}}{(1+a)D'}-\frac{1}{2}(\frac{c\sqrt{D'}}{(1+a)D'})^2))=\psi(\frac{1}{\varpi^{2n+1}}\frac{2c}{(1+a)D'})=\psi(\frac{1}{\varpi^{2n+1}}\frac{2c}{D'}).
	\end{equation}
	Here we have used \eqref{eq2.1:specialembedding} and that $\alpha_\theta$ is imaginary, so $\psi_\L(-\alpha_\theta\frac{1}{2}(\frac{c\sqrt{D'}}{(1+a)D'})^2)=1$. The claimed result then follows easily.
\end{proof}


\subsection{Matrix coefficient}\label{Sec:MatrixcoeffSC}
Here we explain the basics on matrix coefficients for compactly-induced representations.

In general let $G$ be a unimodular locally profinite group with center $Z$. Let $H\subset G$ be an open and closed subgroup containing $Z$ with $H/Z$ compact. Let $\rho$ be an irreducible smooth representation of $H$ with unitary central character and $\pi=c-\Ind_H^G(\rho)$.  By the assumption on $H/Z$, $\rho$ is automatically unitarisable, and we shall denote the unitary pairing on $\rho$ by $<\cdot,\cdot>_{\rho}$. Then one can define the unitary pairing on $\pi$ by
\begin{equation}
<\phi,\psi>_\pi=\sum\limits_{x\in H\backslash G}<\phi(x),\psi(x)>_{\rho}.
\end{equation}
If we let $y\in H\backslash G$ and $\{v_i\}$ be a basis for $\rho$, the elements
$$ f_{y,v_i}(g)=\begin{cases}
\rho(h)v_i,&\text{\ if \ }g=hy\in Hy;\\
0,&\text{\ otherwise.}
\end{cases}$$
form a basis for $\pi$.
\begin{lem}\label{lem3.3:matrixcoeffInduction}
For $y,z\in H\backslash G$,
\begin{equation}
<\pi(g)f_{y,v_i},f_{z,v_j}>_\pi=\begin{cases}
<\rho(h)v_i,v_j>_{\rho}, &\text{\ if\ }g=z^{-1}hy\in z^{-1}Hy;\\
0,&\text{\ otherwise}.
\end{cases}
\end{equation}
\end{lem}
See, for example, \cite{Knightly:aa} for more details.

Now we pick $G=\B^\times$, and $H=J$, $\rho=\Lambda$.  The minimal vector $\varphi$  in Definition \ref{Defn:GeneralMinimalVec} corresponds to the coset $y=1$. 
Type 1, 2 minimal vectors correspond to different choices of basis for $\Lambda$.
Then the lemma above immediately explains Lemma \ref{Cor:MCofGeneralMinimalVec}. Further more we have the following:

\begin{cor}\label{cor:matrixcoeff}
	Let $K$ be the standard maximal compact subgroup, and let $\Vol(Z\backslash ZK)=1$. Then
\begin{equation}
\int\limits_{Z\backslash \GL_2}|\Phi_\varphi(g)|^2dg\asymp \frac{1}{C(\pi\times\pi)^{1/2}}.
\end{equation}
In particular the formal degree of $\pi$ is asymptotically $C(\pi\times\pi)^{1/2}$.

\end{cor}
\begin{proof}

When $e_\L=1$,
\begin{equation}
[ZK:ZK_\L(1)]=[\GL_2(k): k_\E^\times]=q^2-q,
\end{equation}
\begin{equation}
[ZK_\L(i):ZK_\L(i+1)]=[ZO_\L^\times :ZU_\L(n)]=q^2
\end{equation}
for any $i>0$. Here $k$ is the residue field for $\F$. Thus 
\begin{equation}
\Vol(Z\backslash ZK_\L(n))=\frac{1}{[ZK:ZK_\L(n)]}=\frac{1}{(q^2-q)q^{2(n-1)}}=\frac{1}{1-q^{-1}}\frac{1}{q^{2n}}.
\end{equation}
When $c(\pi)=4n$, $C(\pi\times\pi)=q^{4n}$, and 
\begin{equation}
\int\limits_{Z\backslash \GL_2}|\Phi_\varphi(g)|^2dg=\Vol(Z\backslash ZK_{\L}(n))=\frac{1}{1-q^{-1}}\frac{1}{q^{2n}}=\frac{1}{1-q^{-1}} \frac{1}{C(\pi\times\pi)^{1/2}}.
\end{equation}
When $c(\pi)=4n+2$, $C(\pi\times\pi)=q^{4n+2}$, $\dim \Lambda \asymp q$, and
\begin{equation}
\int\limits_{Z\backslash \GL_2}|\Phi_\varphi(g)|^2dg=\frac{1}{\dim \Lambda}\Vol(Z\backslash ZK_{\L}(n))=\frac{q}{(1-q^{-1})\dim \Lambda} \frac{1}{C(\pi\times\pi)^{1/2}}.
\end{equation}

When $e_\L=2$, note that $[ZK_\L(n):ZK_\L^1(n)]=[\L:ZU_\L(1)]=2$, and
\begin{equation}
[ZK:ZK_\L^1(1)]=[\GL_2(k):k^\times N(k)]=q^2-1,
\end{equation}
\begin{equation}
[ZK_\L^1(i):ZK_\L^1(i+1)]=q
\end{equation}
for any $i>0$.
Thus
\begin{equation}
\Vol(Z\backslash ZK_\L(n))=\frac{2}{1-q^{-2}}\frac{1}{q^{n+1}}
\end{equation}
When $c(\pi)=2n+1$, $C(\pi\times\pi)=q^{2n+2}$, and
\begin{equation}
\int\limits_{Z\backslash \GL_2}|\Phi_\varphi(g)|^2dg=\Vol(Z\backslash ZK_{\L}(n))=\frac{2}{1-q^{-2}} \frac{1}{C(\pi\times\pi)^{1/2}}.
\end{equation}
\end{proof}

\subsection{Kirillov model}\label{Section:Kirillov}
Now we  discuss the relation between the compact induction model and the Kirillov model by constructing an intertwining operator explicitly. 
Let $\varphi$ be the type 1 minimal vector associated to $B^1=K_\L^1(n+1,n)$, with $\tilde{\theta}$  given by \eqref{EqA1:thetatilde} and $\alpha_\theta$  given by \eqref{eq2.1:specialembedding}. Thus we have fixed a cuspidal type in the sense of Definition \ref{Def3.1:cupsidaltype}.
We define an intertwining operator from $\pi$ to its Whittaker model by
\begin{equation}\label{eq:3.4:IntertwiningtoWhittaker}
\varphi' \mapsto W_{\varphi'}(g)=\int\limits_{\F}\Phi_{\varphi',\varphi}(\zxz{\varpi^{\lfloor c(\pi)/2\rfloor}}{0}{0}{1}\zxz{1}{n}{0}{1}g)\psi(-n)dn.
\end{equation}

\begin{lem}\label{cor:toricnewforminKirillov}
Up to a constant, the type 1 minimal vectors associated to the fixed cuspidal type can be given in the Kirillov model as follows:
\begin{enumerate}
\item When $c(\pi)=4n$, $\varphi=\Char(\varpi^{-2n}U_\F(n)) $.
\item When $c(\pi)=2n+1$, $\varphi=\Char(\varpi^{-n}U_\F(\lceil n/2 \rceil))$.
\item When $c(\pi)=4n+2$,  
$\varphi_x=\Char(\varpi^{-2n-1}(1+x\varpi^n)U_\F(n+1))$ for $x\in O_\F/\varpi O_\F$.
\end{enumerate}
\end{lem}
\begin{proof}
We only prove part (3) here. The first two results are similar and much easier. 
By the intertwining operator defined above, we have that
\begin{equation}
W_{\varphi }(\zxz{a}{0}{0}{1})=\int_\F \Phi_{\varphi }(\zxz{\varpi^{2n+1}a}{\varpi^{2n+1}n}{0}{1})\psi(-n)dn.
\end{equation}
By Lemma \ref{Cor:MCofGeneralMinimalVec}, we know that $\Phi_{\varphi }(\zxz{\varpi^{2n+1}a}{\varpi^{2n+1}n}{0}{1})$ is non-vanishing only when $a\in \varpi^{-2n-1}U_\F(n+1)$ and $n\in \varpi^{-n-1}O_\F$, in which case it becomes the character $\tilde{\theta}$ on $B^1$. By \eqref{eq2.1:specialembedding} and Lemma \ref{LemA1:char},
\begin{equation}
W_{\varphi }(\zxz{a}{0}{0}{1})=\int_{n\in \varpi^{-n-1}O_\F} \psi(\varpi^{-2n-1}\varpi^{2n+1}n)\psi(-n)dn=q^{n+1}.
\end{equation}
 So up to a constant, $\varphi $ in the Kirillov model is $\varphi =\Char(\varpi^{-2n-1}U_\F(n+1))$, and the intertwining operator is not trivial. 
 All the other type 1 minimal vectors for the fixed cuspidal type should be the translates of $\varphi $ by $K_\L^1(n)/K_\L^1(n+1,n)=\{\zxz{1+x\varpi^n}{0}{0}{1}|x\in O_\F/\varpi O_\F\}$, so the result follows from the definition of the Kirillov model.
\end{proof}

Using the definition for the Kirillov model, one can directly recover the following property for $\varphi_x$:
\begin{equation}\label{Eq3-4:Compactaction4n+2Alternativex}
\pi(ztb)\varphi_x=\theta(zt)\psi(\varpi^{-2n-1}(1+x\varpi^n)m)\varphi_x
\end{equation}
for $z\in Z, t\in U_\L(1), b=\zxz{a}{m}{0}{1}\text{\  where }a\in U_\F(n+1), v(m)\geq n$. When $x=0$, it is indeed consistent with Corollary \ref{CorA1:thetatilde2}. 
The characters $\theta(zt)\psi(\varpi^{-2n-1}(1+x\varpi^n)m)$ for different $x\in O_\F/\varpi O_\F$ actually correspond to different extensions of $\tilde{\theta}$ to $B^1$.

\subsection{Local Langlands correspondence, Jacquet-Langlands correspondence and compact induction}\label{SubSec:Langlands-compactInd}
Here we describe the relation between the compact induction parametrisation and the local Langlands/Jacquet-Langlands correspondence. See \cite{BushnellHenniart:06a} Section 34, 56 for more details.

For a field extension $\L/\F$ and an additive character $\psi$ over $\F$, let $\lambda_{\L/\F}(\psi)$ be the Langlands $\lambda-$function as in \cite{LL}. When $\L/\F$ is a quadratic field extension, let $\eta_{\L/\F}$ be the associated quadratic character. By \cite{LL}, we have for $\psi_\beta(x)=\psi(\beta x)$,
\begin{equation}\label{Eq:LanglandsLambdaFun}
\lambda_{\L/\F}(\psi_\beta)=\eta_{\L/\F}(\beta)\lambda_{\L/\F}(\psi).
\end{equation}

\begin{defn}
\begin{enumerate}
\item If $\L$ is inert, define $\Delta_\theta$ to be the unique unramified character of $\L^\times $ of order 2.

\item If $\L$ is ramified and $\theta$ is a character over $\L$ with $c(\theta)$ even, associate $\alpha_\theta$ to $\theta$ as in Lemma \ref{Lem:DualLiealgForChar}.
Then define $\Delta_\theta$ to be the unique level 1 character of $\L^\times $ such that
\begin{align}
\Delta_\theta|_{\F^\times }=\eta_{\L/\F}, \Delta_\theta(\varpi_\L)=\eta_{\L/\F}(\varpi_\L^{c(\theta)-1}\alpha_\theta)\lambda_{\L/\F}^{c(\theta)-1}(\psi).
\end{align}
\end{enumerate}

\end{defn}
Note that in  \cite{BushnellHenniart:06a}   $\psi$ is chosen to be level 1. We have adapted the formula there to our choice of $\psi$ using \eqref{Eq:LanglandsLambdaFun}.
The definition is also independent of the choice of $\varpi_\L$.    
\begin{theo}
If $\pi$ is associated by compact induction to a character $\theta$ over a quadratic extension $\L$, then its associated Deligne-Weil representation by local Langlands correspondence is $\sigma=\Ind_{\L}^{\F} (\Theta)$, where $\Theta=\theta\Delta_\theta^{-1}$.
\end{theo}
Note here that $\Theta$ and $\theta$ differ by a at most level 1 character, so $\alpha_\Theta$ is congruent to $\alpha_\theta$ and $\Delta_\Theta=\Delta_\theta$.

For Jacquet-Langlands correspondence, we have the following.
\begin{theo} If $\pi$ on $\GL_2$ is constructed from $(\theta,\L)$, then $\pi^\B$ is constructed from $(\theta^\B,\L)$, where
$\theta^\B=\theta$ if $e_\L=1$, and $\theta^\B|_{O_\L^\times}=\theta|_{O_\L^\times}$, $\theta^\B(\varpi_\L)=-\theta(\varpi_\L)$ if $e_\L=2$.
\end{theo}

\subsection{Explicit $\epsilon-$value test}\label{Sec:epsilontest}
Let $\pi$ be associated to $\sigma=\Ind_{\L}^{\F}\Theta$ for a character $\Theta $ over $\L$ with $c(\Theta)=n$. $\chi$ be a character over $\E$ with $c(\chi)=m$.
As noted before, we assume $p\neq 2$ and trivial central character. Further we restrict ourselves to the case $c(\pi)\geq c(\pi_{\chi^{-1}})$ (otherwise the local integral is always non-trivial on the $\GL_2$ side).
Then we have the following table on $\epsilon (\pi_\E\times \chi^{-1})$  according to \cite[proof in Proposition 2.6 and Lemma 3.2]{Tunnell:83a}.

\begin{tabular}{|c|c|c|c|}
\hline
Case &$\L$ 	& $\E$ 	&$\epsilon(\pi_\E\times\chi^{-1})$  \\ \hline
1 &inert	& inert	&$(-1)^{c(\Theta\chi^{-1})+c(\Theta^\tau\chi^{-1})}$ 	\\ \hline
2 &inert	& ramified	& always $-1$\\ \hline
3 &ramified& inert	& always $-1$\\ \hline
\end{tabular}

The case when $\E$ and $\L$ are both ramified is more complicated. We take  the following two lemmas directly from \cite{Tunnell:83a}.
\begin{lem}\label{lem:EpsilontestDistinctramified}
Let $\L$ and $\E$ be distinct ramified quadratic extensions of $\F$ with uniformizers $\varpi_\L$ and $\varpi_\E$. Suppose that $\varpi_\L^2=\varpi_\F$. Let $\Theta$ be a character over $\L$ with $c(\Theta)$ even.
Let $\xi\in O_\F^\times $ such that $\xi \varpi_\L^2=\varpi_\E^2$. Then $\epsilon(\pi_\E\times\chi^{-1})=-1$ iff for $j=c(\Theta)/2-1$ and any $x\in O_\F^\times $,
\begin{equation}
\chi(1+\varpi_\F^j \varpi_\E x)=\Theta(1+\varpi_\F^j\varpi_\L \delta x)
\end{equation}
for certain $\delta \in O_\F$ such that $\delta^2\xi^{-1}-1$ is a square in $O_\F/\varpi O_\F$.
\end{lem}
\begin{rem}
The statement here is slightly different from   \cite[Proposition 2.9]{Tunnell:83a}. We believe this is due to a mistake in \cite{Tunnell:83a}. According to the notation there, $\theta(1+x)=\psi_\L(y_1 x)$ and $\chi^{-1}(1+x)=\psi_\E(y_2 x)$. Then 
$y_2/y_1=-\alpha_\chi/\alpha_\theta$ should be congruent to $-\delta\sqrt{\xi}^{-1}$, instead of $\delta\sqrt{\xi}$ in \cite{Tunnell:83a}. Another evidence  is that in this paper we shall come to the same criterion from a more direct approach.
\end{rem}

\begin{lem}\label{lem:Epsilontestsameramified}
Let $\E$ and $\L$ be the same ramified quadratic extensions with uniformizer $\varpi_\E$ such that $\varpi_\E^2=\varpi_\F$. 
Let $\eta$ be a character of $\E^\times $ extending the quadratic character $\eta_{\E/\F}$ of $\F^\times $. 
Then $\epsilon(\pi_\E\times\chi^{-1})=-1$ iff one of the followings is true
\begin{enumerate}
\item $c(\Theta\chi^{-1}\eta)=c(\Theta^\tau\chi^{-1}\eta)=c(\Theta)$ and, for all $x\in O_\F$, $\chi(1+\varpi_\E^jx)=\Theta(1+\varpi_\E^j \delta x)$ where $j=c(\Theta)-1$ and $\delta \in O_\F/\varpi O_\F$ satisfies $\delta ^2-1$ is not a square in $O_\F/\varpi O_\F$.
\item $0<c(\Theta\chi^{-1}\eta)<c(\Theta)$ and the character $\mu=\Theta\chi^{-1}\eta$ satisfies $\mu(1+\varpi_\E^{c(\mu)-1}x)=\Theta(1+\varpi_\E^{c(\Theta)-1} \delta x)$ for $x\in O_\F$ where $\delta \in O_\F^\times $ satisfies $-2\delta (-1)^{(c(\mu)+c(\Theta))/2}$ is not a square $\mod \varpi O_\F$.
\item $0<c(\Theta^\tau\chi^{-1}\eta)<c(\Theta)$ and the character $\mu=\Theta^\tau\chi^{-1}\eta$ satisfies $\mu(1+\varpi_\E^{c(\mu)-1}x)=\Theta(1+\varpi_\E^{c(\Theta)-1} \delta x)$ for $x\in O_\F$ where $\delta \in O_\F^\times $ satisfies $2\delta  (-1)^{(c(\mu)+c(\Theta))/2}$ is not a square $\mod \varpi O_\F$.
\item $c(\mu)=0$ for $\mu=\Theta\chi^{-1}\eta$ or $\Theta^\tau\chi^{-1}\eta$, and $\mu(\varpi_\E)=-1$.
\end{enumerate}

\end{lem}
Note that we don't need to worry about the case $c(\mu)=1$. This is because when $\mu |_{\F^\times }=1$ and $\E$ is ramified, such characters do not exist.

\section{Test vector for Waldspurger's period integral on $\GL(2)$ side}\label{Sec:AppendixB}
The goal of this section is to find test vector for $I(\varphi,\chi)$ on $\GL_2$ side and prove Theorem \ref{Appendix:Maintheo}, and explicitly show that the existence of such test vector is consistent with the explicit $\epsilon-$value test given in Section \ref{Sec:epsilontest}, while the size of the local integral is also consistent with the computations in Proposition \ref{Prop:SizeofLocalint}. As before 
$\pi$ is a supercuspidal representation associated to a character $\theta$ over a quadratic extension $\L$ by compact induction theory. $\chi$ is a character over the quadratic extension $\E$.


Let $\varphi_0$ be the type 1 minimal vector for the fixed cuspidal type as in Section \ref{Sec:Specialembedding} with matrix coefficient satisfying Lemma \ref{Cor:MCofGeneralMinimalVec}. Let $$\varphi=\pi(\zxz{1}{u}{0}{1}\zxz{v}{0}{0}{1})\varphi_0$$
for some $u,v\in \F^\times $.
Denote \begin{equation}\label{EqB:k}
k=\zxz{1}{u}{0}{1}\zxz{v}{0}{0}{1}.
\end{equation} Then for $\varphi$ we have
\begin{align}
I(\varphi,\chi)&=\int\limits_{\F^\times \backslash\E^\times }\Phi_\varphi(t)\chi^{-1}(t)dt\\
&=\int\limits_{\F^\times \backslash\E^\times }\Phi_{\varphi_0}(k^{-1}tk)\chi^{-1}(t)dt \notag
\end{align}

 We shall assume that $v(v)=0$ and $v(u)\geq 0$. (In principle we need to consider all possible valuations to cover all possible test vectors. But it turns out that the test vectors with these restrictions already suffice. With additional work, one can actually show that the local integral is always vanishing for the complementary range)



As in Section \ref{Sec:mainmethod}, we shall deal with the following two cases separately: 
\begin{enumerate}
	\item $\L\simeq \E$, $\Supp \Phi_\varphi\cap\E^\times=\E^\times$, and $c(\theta\chi^{-1})$ or $c(\theta\overline{\chi}^{-1})\leq 1$;
	\item $\Supp \Phi_\varphi\cap\E^\times\subsetneq\E^\times$, and $c(\theta\chi^{-1}),c(\theta\overline{\chi}^{-1})\geq 1$ when $\L\simeq \E$.
\end{enumerate}  
The discussion for Case (1) is essentially the same as in Section \ref{Sec:wholetoruscase}, so we shall skip the details here. 
Also note that Case (1) covers all situations when $c(\pi)=2\geq c(\pi_{\chi^{-1}})$ and $\E$ is a field, according to the table in Section \ref{Sec:epsilontest}. 
So we shall focus on Case (2) in this section, and assume $c(\pi)\geq 3$ in the following.

 The basic idea is to carefully identify $\Supp \Phi_\varphi\cap\E^\times$ for the family of $\varphi$, and check whether $\Phi_\varphi\chi^{-1}=1$ on the common support. In the case $\dim \Lambda>1$, we make use of the fact that $\Phi_{\varphi_0}$ for type 1 minimal vectors is still multiplicative when restricted to $J^1$.


According to Section \ref{Sec:epsilontest}, we know that the integral will be automatically zero if $e_\L\neq e_\E$. So for simplicity we shall assume $e_\L=e_\E$ and $v(\frac{D}{D'})=0$. But in the proof we shall also partially see why we can't find test vectors when $e_\L\neq e_\E$. Also when $e_\L=e_\E$, $c(\pi)\geq c(\pi_{\chi^{-1}})$ implies that $c(\theta)\geq c(\chi)$.

Let $t=\zxz{a}{b}{bD}{a}\in \E$ with $\min\{v(a),v(b)\}=0$.
According to the proof for Lemma \ref{Lem:JcapEwholeorwhat}, when $ \Supp \Phi_\varphi\cap\E^\times\subsetneq\E^\times$, we must have  
\begin{equation}\label{Eq:AppeBabrangeforNotwholeE}
 \text{ $v(b)\geq 1, v(a)=0$  when $e_\L=1$, $v(b)\geq 0, v(a)=0$ when $e_\L=2$.}
\end{equation}
In particular, if $t\in \Supp \Phi_\varphi\cap \E^\times\subsetneq\E^\times$, we must have $k^{-1}tk \in ZB^1$. This is because by considering $v(\Nm(k^{-1}tk-a))=v(\Nm(t-a))>0$, we know that $k^{-1}tk \in ZJ^1$. Further by Lemma \ref{Cor:MCofGeneralMinimalVec}, matrix coefficient for a type 1 minimal vector restricted to $ZJ^1$ is supported on $ZB^1$.

\begin{lem}\label{lem4.1:support}
Let $t=\zxz{a}{b}{bD}{a}\in \E$ satisfy \eqref{Eq:AppeBabrangeforNotwholeE}, $e_\E=e_\L$ and $B^1$ be fixed as in Section \ref{Sec:Specialembedding}.
Then $k^{-1}t k\in ZB^1$ if and only if
\begin{equation}\label{eq:relatetwotorus1}
bDu\equiv 0\mod{\varpi^{i}}
\end{equation}
\begin{equation}\label{eq:relatetwotorus2}
 b(v^2-\frac{D'}{D})  \equiv 0\mod{\varpi^{j}}.
\end{equation}
Here $i,j\in \Z$ are given as follows
\begin{equation}
\begin{cases}
i=j=n, &\text{\ if  $e_\L=1$, $c(\pi)=4n$}, 
\\
i=n+1, j=n, &\text{\ if $e_\L=1$, $c(\pi)=4n+2$}, 
\\
 i=\lceil n/2\rceil, j=\lfloor n/2\rfloor, &\text{\ if $e_\L=2$, $c(\pi)=2n+1$}.
\end{cases}
\end{equation}
\end{lem}
	\begin{proof}
Note that
	 \begin{equation}\label{EqB1:ktk}
 k^{-1}t k=\zxz{a-bDu}{v^{-1}b(1-Du^2)}{vbD}{a+bDu}.
 \end{equation}
So $k^{-1}t k\in ZB^1$ if and only if
\begin{equation}
a-bDu\equiv a+bDu \mod{\varpi^{i}},
\end{equation}
\begin{equation}
\frac{vbD}{D'}-v^{-1}b(1-Du^2)\equiv 0 \mod{\varpi^{j}}.
\end{equation}
Here one can do a case by case check for the values of $i,j$. Note that when $e_\L=2$,
$$\mathcal{B}_2^n=\zxz{\varpi^{\lceil n/2\rceil}O_\F}{\varpi^{\lfloor n/2\rfloor}O_\F}{\varpi^{\lfloor n/2\rfloor+1}O_\F}{\varpi^{\lceil n/2\rceil}O_\F}.$$
We shall skip the rest details here.
Manipulating these two equations using that $i\geq j$, we get the required congruence equations.
\end{proof}

\subsection{$\L\not\simeq \E$ both ramified}
When $\L\not\simeq \E$ but $e_\L=e_\E$, they must be distinct ramified extensions. 

We choose uniformizers so that $\varpi_\L^2=\varpi_\F$. $\chi$ is a character over a different ramified extension $\E$, with uniformizer $\varpi_\E$ such that $\varpi_\E^2=\xi \varpi_\F$ where $\xi$ is not a square in $(O_\F/\varpi O_\F)^\times $. Recall that the associated Langlands parameter for $\pi$ is $\Theta=\theta \Delta_\theta^{-1}$ where $\Delta_\theta$ is level 1. Then by Lemma \ref{lem:EpsilontestDistinctramified}, we shall compute $I(\varphi,\chi)$ on $\GL_2$ side if and only if for $l=c(\theta)/2-1$ and any $x\in O_\F^\times $,
\begin{equation}
\chi(1+\varpi_\F^l \varpi_\E x)=\theta(1+\varpi_\F^l\varpi_\L \delta  x)
\end{equation}
for certain $\delta \in O_\F$ such that 
\begin{equation}\label{eq:Epsilontestwhendistinct}
\delta ^2\xi^{-1}-1\text{\  is \emph{not} a square in \ }O_\F/\varpi O_\F.
\end{equation}
We now show it is indeed the case from a more direct approach.
\begin{prop}\label{Lem:testforWaldsnecessary}
Suppose that $\E,\L$ are distinct ramified extensions, $c(\pi)=2n+1\geq c(\pi_{\chi^{-1}})$, and $I(\varphi,\chi)\neq 0$ for $\varphi=\pi(k) \varphi_0$ where $k=\zxz{1}{u}{0}{1}\zxz{v}{0}{0}{1}$ with $v(v)=0$, $v(u)\geq 0$. Let $\alpha_\chi$ be any element satisfying 
$\chi(1+x)=\psi_\E(\alpha_\chi x)$ for $1+x\in U_\E(n)$ and $\overline{\alpha_\chi}=-\alpha_\chi$.
Then we must have
\begin{equation}\label{eq:necessarycondfortest}
\frac{D}{D'}v^2-2\varpi_\F^{n}\alpha_\chi \sqrt{D}v+(1-Du^2)\equiv 0 \mod{\varpi_\F^{ \lceil n/2\rceil}}.
\end{equation}
Whether this quadratic equation has solutions is consistent with Tunnell-Saito's $\epsilon-$value test. For each fixed $u$ solution, we get 2 solutions of $v\mod\varpi^{\lceil n/2\rceil}$, and the resulting integral is
\begin{equation}
I(\varphi,\chi)=\frac{1}{q^{\lfloor n/2\rfloor}},
\end{equation}
which is consistent with \eqref{Eq:SizeofIntAllCase} for $l=2n+1$.
\end{prop}
\begin{proof}
In this case, $v(v^2-\frac{D'}{D})=0$. So by Lemma \ref{lem4.1:support}, we must have $b\in \varpi^{\lfloor n/2\rfloor}$ and
$k^{-1}tk\in Z K_\mathfrak{A}(n)$. By \eqref{EqA1:thetatilde} and \eqref{EqB1:ktk}, we have
\begin{equation}
\Phi_{\varphi_0}(k^{-1}tk)=\psi(\varpi_\F^{-c(\theta)/e_\L}\frac{b}{a}(v\frac{D}{D'}+v^{-1}(1-Du^2))).
\end{equation}
By definition of $\alpha_\chi$, we have
\begin{equation}
\chi^{-1}(t)=\psi(-2\alpha_\chi\frac{b\sqrt{D}}{a}).
\end{equation}
$I(\varphi,\chi)\neq 0$ is equivalent to that 
$
\Phi_{\varphi_0}(k^{-1}tk)\chi^{-1}(t)=1
$
on the common support  $Z\{a+b\sqrt{D}|v(a)=0,v(b)\geq \lfloor n/2\rfloor\}=ZU_\E(n)$. 
As $c(\psi)=0$, we must have
\begin{equation}
(v\frac{D}{D'}+v^{-1}(1-Du^2)-2\varpi_\F^{c(\theta)/e_\L}\alpha_\chi \sqrt{D})b\equiv 0 \mod{\varpi_\F^{ n}}
\end{equation}
for $b\in \varpi^{\lfloor n/2\rfloor}$. Since $v\in O_\F^\times$, we get the quadratic equation as claimed.
This quadratic equation has discriminant 
\begin{equation}
\Delta(u)=4D\varpi_\F^{2c(\theta)/e_\L}\alpha_\chi^2-4\frac{D}{D'}+\frac{4D^2u^2}{D'}.
\end{equation}
Note that 
$v(4D\varpi_\F^{2c(\theta)/e_\L}\alpha_\chi^2-4\frac{D}{D'})=0$.  When $v(u)<0$, $\Delta(u)\equiv \frac{4D^2u^2}{D'}$ is never a square.
When $v(u)\geq 0$,
$v(\frac{4D^2u^2}{D'})>0$, and whether $\Delta(u)$ is a square is independent of the choice of $u$. For simplicity one can just pick $u=0$.

Recall that we have 
\begin{equation}
\chi(1+\varpi_\F^l \varpi_\E x)=\theta(1+\varpi_\F^l\varpi_\L \delta  x).
\end{equation}
This implies that 
\begin{equation}
\alpha_\chi=\alpha_\theta\frac{\varpi_\L}{\varpi_\E}\delta .
\end{equation}
Substitute this and $\varpi_\L^{c(\theta)}\alpha_\theta=\frac{1}{\sqrt{D'}}$ into $\Delta(0)$, we get
\begin{equation}
\Delta(0)=4\frac{D}{D'}(\xi^{-1}\delta ^2-1).
\end{equation}
Now $\frac{D}{D'}$ is already not a square. Thus whether $\Delta(0) $ is a square is equivalent to whether $\delta ^2\xi^{-1}-1$ is not a square, which is consistent with Lemma \ref{lem:EpsilontestDistinctramified}. So  the $\epsilon-$value test  for $\GL_2$ is equivalent to $\Delta(0)$ being a square. It's easy to check that one can get two solutions of $v\mod{\varpi_\F^{\lceil n/2\rceil }}$. 
For each solution, we have
\begin{equation}
I(\varphi,\chi)=\Vol (Z\backslash ZU_\E(n))=\frac{1}{q^{\lfloor n/2\rfloor}}.
\end{equation}

\end{proof}
\begin{rem}
One can do a similar computation even when $e_\L\neq e_\E$. Consider for example the case $e_\L=1$ and $e_\E=2$, i.e., $v(D')=0$ and $v(D)=1$. One can obtain a similar quadratic equation with similar discriminant. But in this case
\begin{equation}
\Delta(u)\equiv -4\frac{D}{D'}
\end{equation}
is never a square.
\end{rem}

\subsection{$\L\simeq \E$, $\min\{c(\theta\chi^{-1}), c(\theta\overline{\chi}^{-1})\}\geq 2$}

In this case $\frac{D}{D'}$ is a square and $c(\theta\chi^{-1})$, $c(\theta\overline{\chi}^{-1})$ make sense. Since conjugation by $\zxz{-1}{0}{0}{1}$ effectively change $\chi$ into $\overline{\chi}$, we shall always assume that $c(\theta\chi^{-1})\leq c(\theta\overline{\chi}^{-1})=c(\theta)$. From \eqref{eq:relatetwotorus2}, we get that
$b(v\pm\sqrt{\frac{D'}{D}})\equiv 0\mod \varpi^j$, 
We shall always assume that 
\begin{equation}\label{eq:relatetwotorus2'}
 b(v-\sqrt{\frac{D'}{D}})\equiv 0\mod \varpi^j
\end{equation}
which is closely related to the previous assumption that $c(\theta\chi^{-1})\leq c(\theta\overline{\chi}^{-1})$. Again  the action of $\zxz{-1}{0}{0}{1}$ switch the congruence condition for $v$.
\begin{lem}\label{lem:supportimpliesclosetorus}
	Suppose that $c(\pi)=4n, 2n+1$ or $4n+2$.
Let $k$ be as in \eqref{EqB:k} and let $t=\zxz{a}{b}{bD}{a}$ satisfy \eqref{Eq:AppeBabrangeforNotwholeE}, \eqref{eq:relatetwotorus1}, \eqref{eq:relatetwotorus2'}. Then we can write
\begin{align}\label{Eq:rewritingconjugation}
k^{-1}t k=\zxz{a}{b\sqrt{\frac{D}{D'}}}{b\sqrt{DD'}}{a}g,\end{align}
where $$g=
\frac{1}{a^2-b^2D}\zxz{a^2-abDu-vb^2\sqrt{\frac{D}{D'}}D}{ab(v^{-1}-\sqrt{\frac{D}{D'}})-bDu(b\sqrt{\frac{D}{D'}}+v^{-1}au)}{abD(v-\sqrt{\frac{D'}{D}})+b^2Du\sqrt{DD'}}{a^2+abDu-v^{-1}b^2\sqrt{DD'}(1-Du^2)},$$
$g\in K_{\mathfrak{A}}(n)$ when $B^1=H^1$, or $g$ satisfies the condition in Corollary \ref{CorA1:thetatilde2} when $B^1\neq H^1$.

Then
\begin{align}\label{EqB:Phivarphi}
\Phi_\varphi(t)&=\theta(a+b\sqrt{D})\psi(\varpi_\F^{-c(\theta)/e_\L} \frac{ab}{a^2-b^2D}(\frac{D}{D'}v+v^{-1}(1-Du^2)-2\sqrt{\frac{D}{D'}}))\\
&=\theta(a+b\sqrt{D})\psi(\varpi_\F^{-c(\theta)/e_\L} \frac{b}{av}((\sqrt{\frac{D}{D'}}v-1)^2-Du^2)). \notag\end{align}
\end{lem}
\begin{proof}[Sketch of proof]
	One can directly verify \eqref{Eq:rewritingconjugation}
using \eqref{EqB1:ktk}.
The congruence conditions for $t$ guarantee that $g\in K_{\mathfrak{A}}(n)$ when $B^1=H^1$, or $g$ satisfies the condition in Corollary \ref{CorA1:thetatilde2} when $B^1\neq H^1$. 

 For the first equality in \eqref{EqB:Phivarphi}, we use Lemma \ref{Cor:MCofGeneralMinimalVec}, \eqref{EqA1:thetatilde}, and Corollary \ref{CorA1:thetatilde2} when $B^1\neq H^1$. In particular note that $$\zxz{a}{b\sqrt{\frac{D}{D'}}}{b\sqrt{DD'}}{a}=a+b\sqrt{\frac{D}{D'}} \cdot \sqrt{D'}=a+b\sqrt{D}$$ under the embedding of $\L$. 

For the second equality in \eqref{EqB:Phivarphi}, note that $$\frac{ab}{(a^2-b^2D)v}=\frac{ab}{a^2v}(1+\frac{b^2D}{a^2}+\cdots).$$ One can show that all error terms do not matter by studying their valuations and using 
\eqref{Eq:AppeBabrangeforNotwholeE}, \eqref{eq:relatetwotorus1}, \eqref{eq:relatetwotorus2'}. 

\end{proof}
Now to check $\Phi_{\varphi_0}(k^{-1}tk)\chi^{-1}(t)=1$ on the common support, one can first compare the level of $\theta\chi^{-1}(a+b\sqrt{D})$ with the level of 
$\psi(\varpi_\F^{-c(\theta)/e_\L} \frac{b}{av}((\sqrt{\frac{D}{D'}}v-1)^2-Du^2))$. From this we can get the domain of the integrals. We shall now do a case by case discussion.

\subsubsection{$\L\simeq \E$ inert, $c(\pi)=4n$}\label{Sec:E=L4nNew}
\begin{lem}\label{lem:supportofintC4n}
If $c(\theta)=2n$, $c(\theta\chi^{-1})>1$ and $I(\varphi,\chi)\neq 0$, then $c(\theta\chi^{-1})=2\l $ must be even and $v((\sqrt{\frac{D}{D'}}v-1)^2-Du^2)= 2(n-\l ) $. In that case the support of the integral is $Z\{a+b\sqrt{D}|v(a)=0,v(b)\geq \l \}$.
\end{lem}
\begin{proof}
Since $v((\sqrt{\frac{D}{D'}}v-1)^2-Du^2)$ is always even, $\psi(\varpi_\F^{-c(\theta)/e_\L} \frac{b}{av}((\sqrt{\frac{D}{D'}}v-1)^2-Du^2))$ as a function in $b$ is of even level. Then $\theta\chi^{-1}$ must also have even level. For the support, simply use \eqref{eq:relatetwotorus1} and \eqref{eq:relatetwotorus2'}.
\end{proof}

Note that the condition on $c(\theta\chi^{-1})$ is already consistent with the $\epsilon-$value test in Section \ref{Sec:epsilontest}.
On this domain one can further write 
\begin{equation}
\theta\chi^{-1}(a+b\sqrt{D})=\psi_\E(\alpha_{\theta\chi^{-1}}\frac{b\sqrt{D}}{a})=\psi(2\alpha_{\theta\chi^{-1}}\frac{b\sqrt{D}}{a}).
\end{equation}
Thus $\Phi_{\varphi_0}(k^{-1}tk)\chi^{-1}(t)=1$ on the support of the integral is equivalent to that
\begin{equation}\label{eq:equivconditionC4n}
(\frac{D}{D'}v^2+2v(\varpi^{c(\theta)/e_\L}\alpha_{\theta\chi^{-1}}\sqrt{D}-\sqrt{\frac{D}{D'}})+1-Du^2)b\equiv 0\mod\varpi^{c(\theta)/e_\L}.
\end{equation}
Its discriminant is
\begin{equation}
{\Delta}(u)=4\varpi^{2n}\alpha_{\theta\chi^{-1}}\sqrt{D}(\varpi^{2n}\alpha_{\theta\chi^{-1}}\sqrt{D}-2\sqrt{\frac{D}{D'}})+4\frac{D}{D'}Du^2.
\end{equation}
Note that
$v(4\varpi^{2n}\alpha_{\theta\chi^{-1}}\sqrt{D})=2n-2d$, and $v(\varpi^{2n}\alpha_{\theta\chi^{-1}}\sqrt{D}-2\sqrt{\frac{D}{D'}})=0$ is also even. This is  because $$\varpi^{2n}\alpha_{\theta\chi^{-1}}\sqrt{D}-2\sqrt{\frac{D}{D'}}=\varpi^{2n}\sqrt{D}(\alpha_{\theta\chi^{-1}}-2\alpha_\theta)=\varpi^{2n}\sqrt{D}\alpha_{\overline{\theta}\chi^{-1}},$$ and $v_\L(\alpha_{\overline{\theta}\chi^{-1}})=-2n$ as we assumed that $2n=c(\theta\overline{\chi}^{-1})\geq c(\theta \chi^{-1})$.

Then $\Delta(u)$ being a square for proper $u$ is equivalent to that $$4\varpi^{2n}\alpha_{\theta\chi^{-1}}\sqrt{D}(\varpi^{2n}\alpha_{\theta\chi^{-1}}\sqrt{D}-2\sqrt{\frac{D}{D'}})=\Delta(u)-4\frac{D}{D'}Du^2$$ can be written as the norm of some element in $\L^\times$, which follows from the following lemma:
\begin{lem}\label{LemB.2:surjNorm}
The 
norm map on the residue field
$$\Nm: a+b\sqrt{D}\mapsto a^2-b^2D$$
is surjective when restricted to the domain $\{a+b\sqrt{D}| ,a\neq 0\}$. 
\end{lem}

We have $v({\Delta}(u))=2n-2\l $ by proper choice of $u$. Then the quadratic equation (\ref{eq:equivconditionC4n}) has two solutions $v\mod\varpi^n$. For each of these two solutions, we have
\begin{equation}
I(\varphi,\chi)=\Vol(Z\backslash ZU_\E(\l ))=\frac{1}{(q+1)q^{\l -1}},
\end{equation}
which is consistent with \eqref{Eq:SizeofIntAllCase} for $l=4\l$.
\subsubsection{$\L\simeq \E$ inert, $c(\pi)=4n+2$}

\begin{lem}
If $c(\theta)=2n+1$, $c(\theta\chi^{-1})>1$ and $I(\varphi,\chi)\neq 0$, then $c(\theta\chi^{-1})=2\l +1$ must be even and  $v(\sqrt{\frac{D}{D'}}v-1)^2-Du^2)= 2(n-\l ) $.
If one can pick $v(u)> n-\l $, then the support of the integral is $Z\{a+b\sqrt{D}|v(a)=0,v(b)\geq \l \}$. If one can pick $v(u)=n-\l $, then the support of the integral is $Z\{a+b\sqrt{D}|v(a)=0,v(b)\geq \l +1\}$.
\end{lem}
Again the condition on $c(\theta\chi^{-1})$ is consistent with the $\epsilon-$value test.
 Note that in the case $v(u)>n-\l $ and $v(b)\geq \l $, 
\begin{equation}
 \theta\chi^{-1}(a+b\sqrt{D})=\theta\chi^{-1}(1+\frac{b\sqrt{D}}{a})=\psi\circ\Tr(\alpha_{\theta\chi^{-1}}(\frac{b\sqrt{D}}{a}-\frac{b^2D}{2a}))=\psi(2\alpha_{\theta\chi^{-1}}\frac{b\sqrt{D}}{a})
\end{equation}
by \eqref{eq:alphathetaless} and that $\alpha_{\theta\chi^{-1}}$ is imaginary.
Thus in either cases one get the same expressions for the quadratic equation and its discriminant as in Section \ref{Sec:E=L4nNew}.

Once $c(\theta\chi^{-1})=2d+1$, the existence of $u$ to make $\Delta(u)$ a square is guaranteed by Lemma \ref{LemB.2:surjNorm}. But there is no guarantee that $v(u)=n-d$.
The difference in the domain for $b$ is related to the difference in power in 
\eqref{eq:relatetwotorus1}, \eqref{eq:relatetwotorus2'}.

If we can take $v(u)=n-\l $ and  $v(b)\geq \l +1$ for the support of the integral, we can solve for two solutions of $v\mod\varpi^{ n}$. This implies that there are  $2q$ solutions $\mod \varpi^{n+1}$. For each such solutions we have
\begin{equation}
I(\varphi,\chi)=\frac{1}{(q+1)q^\l }.
\end{equation}

If we can take $v(u)>n-\l $ and $v(b)\geq \l $ for the support of the integral. Solving \eqref{eq:equivconditionC4n} gives two solution of $v\mod{\varpi^{ n+1}}$. For each solution we have
\begin{equation}
I(\varphi,\chi)=\frac{1}{(q+1)q^{\l -1}}.
\end{equation}

$l=4\l +2$ in this case.
These results are not exactly consistent with \eqref{Eq:SizeofIntAllCase}, as our choice of $B^1$ are not necessarily optimal. But they are matching two possible outcomes given in  \eqref{Sec4.4Eq:SizeofIntCase1} and \eqref{Sec4.4Eq:SizeofIntCase2}.


\subsubsection{$\L\simeq \E$ ramified}
Note that $c(\theta\chi^{-1})$ must be even in this case.
\begin{lem}
Suppose that $c(\theta\chi^{-1})=2\l $,  $I(\varphi,\chi)\neq 0$. If $n-\l $ is even, then we must have $v_\F((\sqrt{\frac{D}{D'}}v-1)^2)=n-\l $, $v_\F(u)\geq (n-\l )/2$(actually we can pick $u=0$) and $v(b)\geq  \lfloor n/2\rfloor -\frac{n-\l }{2}=\lfloor \l /2\rfloor$. If $n-\l $ is odd, then $v_\F((\sqrt{\frac{D}{D'}}v-1)^2)>n-\l $, $v_\F(u)=\frac{n-\l -1}{2}$ and $v(b)\geq \lceil n/2\rceil-\frac{n-\l +1}{2}=\lceil \frac{\l -1}{2}\rceil =\lfloor \l /2\rfloor$.
\end{lem}
One can get similar quadratic equation as before:
\begin{equation}\label{eq:new-necessary-same-ramified}
b(  \frac{D}{D'}v^2+(2\varpi_\F^n\alpha_{\theta\chi^{-1}}\sqrt{D}-2\sqrt{\frac{D}{D'}})v+(1-Du^2))\equiv 0\mod\varpi_\F^n.
\end{equation}
The  discriminant is
\begin{equation}
{\Delta}(u)=4\varpi_\F^{n}\alpha_{\theta\chi^{-1}}\sqrt{D}(\varpi_\F^{n}\alpha_{\theta\chi^{-1}}\sqrt{D}-2\sqrt{\frac{D}{D'}})+4\frac{D}{D'}Du^2.
\end{equation}
Here we have used that $c(\theta)=2n$. Also note that $v_\E(\alpha_{\theta\chi^{-1}})=-2\l -1$.

Consider first the case $n-\l >0$ is even. We pick $u=0$ directly for simplicity. 
Then $v_\F({\Delta}(0))=n-\l $ and
\begin{equation}
{\Delta}(0)\equiv -8\varpi_\F^n\alpha_{\theta\chi^{-1}}\frac{D}{D'}\sqrt{D'}
\end{equation}
Recall that by Lemma \ref{lem:Epsilontestsameramified} (2), we write $\mu=\Theta\chi^{-1}\eta$, where $\eta $ is a character of $\E^\times $ extending $\eta_{\E/\F}$. In particular if we pick $\eta=\Delta_\theta$, then $\mu=\theta\chi^{-1}$. When we write $\mu(1+\varpi_\E^{c(\mu)-1}x)=\theta(1+\varpi_\E^{c(\theta)-1}\delta x)$, this implies that
\begin{equation}
\alpha_{\theta\chi^{-1}}\varpi_\E^{2\l +1}\equiv \alpha_\theta \varpi_\E^{2n+1}\delta  \equiv    \varpi_\E\frac{1}{\sqrt{D'}} \delta \mod\varpi.
\end{equation}
Thus 
\begin{equation}
{\Delta}(0)\equiv -8\varpi_\F^{n-\l } \frac{D}{D'} \delta .
\end{equation}
It's a square iff $-2\delta $ is a square, consistent with Lemma \ref{lem:Epsilontestsameramified} (2).

When $n-\l=0$, we still pick $u=0$ and $\Delta(0)=4\varpi_\F^{n}\alpha_{\theta\chi^{-1}}\sqrt{D}(\varpi_\F^{n}\alpha_{\theta\chi^{-1}}\sqrt{D}-2\sqrt{\frac{D}{D'}})$. By the notation in Lemma \ref{lem:Epsilontestsameramified}(1), 
$\chi(1+\varpi_\E^{c(\theta)-1} x)=\theta(1+\varpi_\E^{c(\theta)-1}\delta x)$, which implies that $\alpha_\chi\equiv\delta\alpha_\theta$. Then
\begin{equation}
\Delta(0)\equiv 4\varpi_\F^{2n}D\alpha_\theta^2(\delta^2-1).
\end{equation}
$4\varpi_\F^{2n}D\alpha_\theta^2$ is already a square, so $\Delta(0)$ is a square is equivalent to $\delta^2-1$ being a square, which is  consistent with Lemma \ref{lem:Epsilontestsameramified}(1).

When ${\Delta}(0)$ is indeed a square, we get two solutions of $v\mod \varpi_\F^{ \lceil n/2\rceil}$. 
For each of these two solutions we have
\begin{equation}
I(\varphi,\chi)=\frac{1}{q^{\lfloor n/2\rfloor -\frac{n-\l }{2}}}=\frac{1}{q^{\lfloor \l /2\rfloor}}.
\end{equation}

Now if $n-\l $ is odd, $v_\F({\Delta}(0))=n-\l $ is odd, thus ${\Delta}(0)$ can never be a square. We need to pick $u$ such that $v_\F(u)=\frac{n-\l -1}{2}$ and ${\Delta}(0)+4\frac{D}{D'}Du^2$ can be of higher evaluation. For this purpose we need that
\begin{equation}
-8\varpi_\F^{n-\l } \frac{D}{D'} \delta +4\frac{D}{D'}Du^2\equiv 0.
\end{equation}
Note that $D$ differs from $\varpi_\F$ by a square and $\frac{D}{D'}$ is also a square. So this being possible is equivalent to that $2\delta $ is a square. This again is consistent with Lemma \ref{lem:Epsilontestsameramified} (2).

Once $2\delta $ is a square, we can easily adjust $u$ so that ${\Delta}(u)$ is a square. In this case it is possible to get different number of solutions of $v\mod\varpi_\F^{\lceil n/2\rceil}$, depending on $v(\Delta(u))$. For each solution we have
\begin{equation}
I(\varphi,\chi)=\frac{1}{q^{\lceil n/2\rceil-\frac{n-\l +1}{2}}}=\frac{1}{q^{\lfloor \l /2\rfloor}}.
\end{equation}
In either case, the size of the local integral is consistent with \eqref{Eq:SizeofIntAllCase} with $l=2\l+1$.

\bibliographystyle{plain}

\end{document}